\documentclass[11pt, twoside]{article}

\usepackage{amssymb}
\usepackage{amsmath}
\usepackage{amsthm}
\usepackage{amsfonts}
\usepackage{mathrsfs}
\usepackage{color}
\usepackage{indentfirst}
\usepackage{txfonts}

\allowdisplaybreaks

\def\red{\color{red}}

\def\rr{{\mathbb R}}
\def\rn{{{\rr}^n}}

\def\zz{{\mathbb Z}}
\def\nn{{\mathbb N}}
\def\ca{{\mathcal A}}
\def\cc{{\mathcal C}}
\def\uu{{\mathcal U}}

\def\cg{{\mathcal G}}
\def\icgg{\mathcal{G}_0^\eta(\beta,\gamma)}

\def\cl{{\mathcal L}}
\def\cm{{\mathcal M}}

\def\cx{{\mathcal X}}
\def\cy{{\mathcal Y}}

\def\fz{\infty}
\def\az{\alpha}
\def\bz{\beta}
\def\dz{\delta}

\def\ez{\epsilon}
\def\gz{{\gamma}}

\def\lz{\lambda}

\def\oz{{\omega}}

\def\tz{\theta}
\def\sz{\sigma}

\def\vz{\varphi}

\def\lf{\left}
\def\r{\right}

\def\la{\langle}
\def\ra{\rangle}
\def\hs{\hspace{0.26cm}}

\def\ls{\lesssim}

\def\noz{\nonumber}

\def\st{\subset}
\def\com{\complement}

\def\supp{\mathop\mathrm{\,supp\,}}

\def\fin{{\mathop\mathrm{fin}}}

\def\ch1{\mathbf{1}}
\def\aa{{\mathbb A}}
\def\q1{q_1}
\def\Q1{q_2}

\def\hx{H_X(\rn)}

\def\ajk{{a_{j,k}}}

\def\fin{{{\mathop\mathrm{fin}}}}

\DeclareMathOperator*{\esssup}{ess\ sup}

\def\supp{{\mathop\mathrm{\,supp\,}}}
\def\diam{{\mathop\mathrm{\,diam\,}}}

\newtheorem{theorem}{Theorem}[section]
\newtheorem{lemma}[theorem]{Lemma}
\newtheorem{proposition}[theorem]{Proposition}
\newtheorem{corollary}[theorem]{Corollary}
\newtheorem{assumption}[theorem]{Assumption}
\theoremstyle{definition}
\newtheorem{definition}[theorem]{Definition}

\newtheorem{remark}[theorem]{Remark}

\numberwithin{equation}{section}
\numberwithin{equation}{section}

\begin{document}

\arraycolsep=1pt

\title{\bf\Large Hardy Spaces Associated with Ball
Quasi-Banach Function Spaces on Spaces of Homogeneous
Type: Characterizations of Maximal Functions,
Decompositions, and  Dual Spaces
\footnotetext{\hspace{-0.35cm} 2020
{\it Mathematics Subject Classification}.
Primary 42B30; Secondary 42B25, 42B35, 46E36, 30L99.
\endgraf {\it Key words and phrases.} space of homogeneous type,
ball quasi-Banach function space, Hardy space,  maximal function,
atom, molecule, ball Campanato-type function space.
\endgraf This project is partially supported by the National
Key Research and Development Program of China
(Grant No. 2020YFA0712900), the National Natural Science
Foundation of China (Grant Nos. 11971058, 12071197 and 11871100)
and the Fundamental Research
Funds for the Central Universities (Grant Nos. 500421359 and 500421126).}}
\author{Xianjie Yan, Ziyi He, Dachun Yang\footnote{Corresponding author,
E-mail: \texttt{dcyang@bnu.edu.cn}/{\red{August 10, 2021}}/Final version.}
\ and Wen Yuan}
\date{}
\maketitle

\vspace{-0.8cm}

\begin{center}
\begin{minipage}{13cm}
{\small {\bf Abstract}\quad
Let $({\mathcal X},\rho,\mu)$ be a space of homogeneous type
in the sense of Coifman and Weiss,
and $Y({\mathcal X})$ a ball quasi-Banach function space on ${\mathcal X}$,
which supports a Fefferman--Stein vector-valued maximal
inequality, and the boundedness of the powered Hardy--Littlewood
maximal operator on its associate space.
The authors first introduce the Hardy space $H_{Y}^*({\mathcal X})$,
associated with $Y({\mathcal X})$, via the grand maximal function,
and then establish its various real-variable characterizations,
respectively, in terms of radial or non-tangential maximal functions,
atoms or finite atoms, and molecules.
As an application, the authors give the dual space of
$H_{Y}^*({\mathcal X})$, which proves to be a ball Campanato-type function space
associated with $Y({\mathcal X})$.
All these results have a wide range of generality and, particularly,
even when they are applied to variable Hardy spaces,
the obtained results are also new.
The major novelties of this article exist in that, to escape the reverse
doubling condition of $\mu$
and the triangle inequality of $\rho$, the authors
cleverly construct admissible sequences of balls,
and fully use the geometrical properties of ${\mathcal X}$ expressed
by dyadic reference points or dyadic cubes and,
to overcome the difficulty caused by the lack of the good dense
subset of $H_{Y}^*({\mathcal X})$, the authors further prove
that $Y({\mathcal X})$ can be embedded into the weighted Lebesgue
space with certain special weight, and then can fully use the
known results of the weighted Lebesgue space.}
\end{minipage}
\end{center}

\vspace{0.1cm}

\tableofcontents

\vspace{0.1cm}

\section{Introduction\label{s-intro}}

It is well known that the real-variable theory
of Hardy-type spaces on $\rn$ including various real-variable
equivalent characterizations and the boundedness of
Calder\'on--Zygmund operators
plays a key role in harmonic analysis and partial
differential equations (see, for instance, \cite{mu94,stein93}).
Recall that the classical Hardy space $H^p(\rn)$ with
$p\in(0,1]$ was originally introduced by Stein and Weiss
\cite{sw60} and further studied by Fefferman and
Stein \cite{fs72}. From then on,
various variants of classical Hardy spaces have been
introduced and their real-variable theories have been well developed
(see, for instance, \cite{at07,ho15,ho17,hyy,lyy,ly95,ns12,ylk17}).
Recently, Sawano et al. \cite{shyy17} originally introduced
the ball quasi-Banach function space $X$ on $\rn$,
which further generalizes the Banach function space
in \cite{bs88} in order to include weighted Lebesgue spaces,
Morrey spaces, mixed-norm Lebesgue spaces, Orlicz-slice spaces,
and Musielak--Orlicz spaces.
Observe that the aforementioned several function spaces
are not quasi-Banach function spaces (see, for instance,
\cite{shyy17,st15,wyyz,zwyy}).
In the same article \cite{shyy17}, Sawano et al. also
introduced the Hardy space $H_X(\rn)$ associated with $X$,
and established its various maximal function characterizations
by assuming that the Hardy--Littlewood maximal operator is
bounded on the $p$-convexification of $X$,
and several other real-variable characterizations, respectively,
in terms of atoms, molecules, and Lusin-area functions by assuming that
the Hardy--Littlewood maximal operator satisfies a
Fefferman--Stein vector-valued inequality on $X$,
and is bounded on the associate space of $X$.
Later, Wang et al. \cite{wyy} further established the
Littlewood--Paley $g$-function and $g_{{\lambda}}^\ast$-function
characterizations of both $H_X(\rn)$ and its localized version $h_X(\rn)$,
and obtained the boundedness of Calder\'on--Zygmund operators
and pseudo-differential operators, respectively,
on $H_X(\rn)$ and $h_X(\rn)$.
Under much weaker assumptions on the Littlewood--Paley functions
than the corresponding ones in \cite{shyy17,wyy}, Chang et al.
\cite{cwyz19} obtained various Littlewood--Paley function
characterizations of $H_X(\rn)$.
Recently, Zhang et al. \cite{zwyy} and Wang et al. \cite{wyyz}
introduced weak Hardy-type spaces $W\!H_X(\rn)$ associated
with $X$ and developed a complete real-variable theory of these spaces;
Yan et al. \cite{yyy20} established the dual theory and
obtained the intrinsic square function characterizations of $\hx$;
Zhang et al. \cite{zhyy21} introduced some new ball Campanato-type
function space which proves the dual space of $\hx$,
and established its Carleson measure characterization.
For more studies about ball quasi-Banach function spaces on $\rn$,
we refer the reader to \cite{its19,is17,s18,yyy20b}.

On the other hand, Coifman and Weiss \cite{cw71,cw77} originally
introduced the notion of spaces of homogeneous type in order to
establish a uniform framework for both the theory of
Calder\'on--Zygmund operators and the real-variable theory
of Hardy spaces over different underlying spaces.
From then on, spaces of homogeneous type have become the most
natural and the most general underlying space to study the real-variable
theory of function spaces and the boundedness of operators
(see, for instance, \cite{dh09,k01,li98,ms79ii,n06}).
However, many existing results of both function spaces
and boundedness of operators on spaces of homogeneous
type require some additional geometrical assumptions
on the underlying spaces
such as the reverse doubling property of the equipped
measure $\mu$ (see, for instance, \cite{gly08,sst18,sgn18,zsy}).
Recently, a breakthrough on the analysis over a space ${\mathcal X}$
of homogeneous type without any additional geometrical assumptions
was made by Auscher and Hyt\"onen \cite{ah13} who constructed
an orthonormal wavelet basis, with exponential decay,
of $L^2(\cx)$ by using the system of random dyadic cubes.
From then on, the real-variable theory of function spaces
on a given space $\cx$ of homogeneous type has rapidly been developed.
To be precise, Han et al. \cite{hhl16} introduced the Hardy space
via wavelets on $\cx$ and established the criteria of the boundedness
of Calder\'on--Zygmund operators on it and its dual space.
Fu and Yang \cite{fy18} obtained an unconditional basis and
several equivalent characterizations of the atomic Hardy
space on $\cx$ in terms of wavelets. He et al. \cite{hlyy19}
introduced a kind of approximation of the identity with
exponential decay and obtained new Calder\'on reproducing formulae
on $\cx$, which proved necessary to
establish various real-variable characterizations of Hardy spaces.
Motivated by this, He et al. \cite{hhllyy} developed a complete
real-variable theory of Hardy spaces on $\cx$
including various real-variable equivalent characterizations
and the boundedness of sublinear operators.
Fu et al. \cite{fmy19} and Zhou et al. \cite{zhy}
generalized the corresponding results in \cite{hhllyy},
respectively, to Musielak--Orlicz Hardy spaces and
Hardy--Lorentz spaces.
We refer the reader to \cite{bdl18,bdl20,chen14,fy17jfaa,hlw18,hwyy21,hyy19,lj10,whyy21}
for more studies on the real-variable theory of
function spaces over $\cx$, and
\cite{dgklwy,lj13,liu18,whhy21} for their applications in
the boundedness of operators;
particularly, we refer the reader to \cite{bddm19,bdn20,bd20}
for some recent progress on the real-variable
theory of function spaces on $\cx$ associated with operators.

Throughout this article, we always let $(\cx,\rho,\mu)$ be a
space of homogeneous type in the sense of Coifman and Weiss,
and $Y(\cx)$ a ball quasi-Banach function space on $\cx$,
which supports a Fefferman--Stein vector-valued maximal inequality,
and the boundedness of the powered Hardy--Littlewood maximal
operator on its associate space. Motivated by \cite{hhllyy}
and \cite{shyy17}, in this article, using the grand maximal function,
we first introduce the Hardy space $H_{Y}^*(\cx)$
associated with $Y(\cx)$, and then establish
its various real-variable characterizations, respectively,
in terms of radial or non-tangential maximal functions,
atoms or finite atoms, and molecules. As an application,
we also give the dual space of $H_{Y}^*(\cx)$, which proves to be
a ball Campanato-type function space associated with $Y(\cx)$.
These results can further be used to establish
the Littlewood--Paley function characterizations of $H_{Y}^*(\cx)$
and the boundedness of Calder\'on--Zygmund operators on $H_{Y}^*(\cx)$;
to limit the length of this article, they are presented in \cite{yhyy21-2}.
The major novelties of this article exist in that, to escape the reverse
doubling condition of $\mu$ and the triangle inequality of $\rho$,
we cleverly construct admissible sequences of balls,
and fully use the geometrical properties of ${\mathcal X}$ expressed
by dyadic reference points and dyadic cubes and,
to overcome the difficulty caused by the lack of the good dense
subset of $Y({\mathcal X})$, we further prove that $Y({\mathcal X})$
can be embedded into the weighted Lebesgue space with certain special
weight, and then can fully use the known results of the
weighted Lebesgue space.

We point out that the results obtained in this article and
\cite{yhyy21-2} have a wide range of generality because
$H_{Y}^*(\cx)$ includes various known Hardy spaces on $\cx$,
for instance, Hardy spaces in \cite{hlyy19},
Hardy--Lorentz spaces in \cite{zhy},
weighted Hardy spaces in \cite{fmy19},
Orlicz--Hardy spaces in \cite{fmy19},
and variable Hardy spaces on RD-spaces in \cite{zsy}.
Moreover, even when these results are applied to these
concrete cases, some results are also new;
obviously, more applications are expectable.

The remainder of this article is organized as follows.

In Section \ref{s-pre}, we recall some notation and
notions used throughout this article. More precisely,
in Subsection \ref{hts}, we recall the definition of a space
$\cx$ of homogeneous type, and some basic properties of $\cx$.
In Subsection \ref{sbqbs}, we introduce the ball quasi-Banach
function space $Y(\cx)$ on $\cx$ and give some examples.
Some mild assumptions on the boundedness of the Hardy--Littlewood
maximal operator on $Y(\cx)$ are stated in Subsection \ref{bshl},
which are needed throughout this article.
In Subsection \ref{emd}, we first creatively introduce the notion
of admissible sequences of balls to get rid of the dependence
on the reverse doubling property of $\mu$ because an admissible
sequence of balls are either doubling and also reverse doubling
(see Remark \ref{r2.25} below for more details). Also,
it is well known that a ball quasi-Banach function space
may not have a good dense subset and hence the usual
dense argument does not work anymore in this setting.
To overcome this difficulty, using admissible sequences of balls,
we further show that a ball quasi-Banach function space can
be embedded to the weighted Lebesgue space with certain
special $A_1(\cx)$ weight (see Theorem \ref{embed} below),
which plays a crucial role in clarifying the relations between $Y(\cx)$
and the corresponding Hardy space (see the proof of Theorem
\ref{relation} below) and also in establishing the atomic characterization
of the corresponding Hardy space
(see the proof of Proposition \ref{atde} below).

The aim of Section \ref{s-max} is to introduce the Hardy space
$H_Y^*(\cx)$ associated with $Y(\cx)$, discuss the relations
between $Y(\cx)$ and $H_Y^*(\cx)$, and establish various
maximal function characterizations of $H_Y^*(\cx)$.
To be precise, in Subsection \ref{s-max1}, we first introduce
three Hardy spaces $H_{Y}^+(\cx)$, $H_{Y}^a(\cx)$ with $a\in(0,\fz)$,
and $H_{Y}^*(\cx)$ associated with $Y(\cx)$, via using, respectively,
the radial maximal function, the non-tangential maximal function,
and the grand maximal function in Definition \ref{bh},
and then clarify the relations between $Y(\cx)$ and $H_Y^*(\cx)$
in Theorem \ref{relation} below, whose proof strongly depends on
the embedding of $Y(\cx)$ into the weighted Lebesgue space
with certain special $A_1(\cx)$ weight obtained in Theorem \ref{embed}.
In Subsection \ref{s-max2}, we first prove that,
as subspaces of the space of distributions,
the Hardy spaces mentioned above coincide with
each other in the sense of equivalent quasi-norms
(see Theorem \ref{maxch} below).
Indeed, by \cite[(3.1)]{hhllyy}, we know that
these maximal functions are pointwisely comparable
with each other. From this, the boundedness of the
Hardy--Littlewood maximal operator on $Y(\cx)$,
and an argument similar to that used in the
proof of \cite[(3.4)]{hhllyy}, we deduce Theorem \ref{maxch}.
Second, by borrowing some ideas from the proof of
\cite[Proposition 3.8]{hhllyy},
we prove that $H_Y^*(\cx)$ are independent of the choices
of distributions (see Theorem \ref{indep} below).
Finally, combining Theorems \ref{maxch} and \ref{indep},
we further obtain various maximal function characterizations
of $H_Y^*(\cx)$.

Section \ref{s-atom} is devoted to establishing the atomic and
the finite atomic characterizations of $H_{Y}^*(\cx)$.
In Subsection \ref{satom1}, using a variant on $\cx$ of the
key lemma obtained in \cite[Lemma 4.8]{zwyy},
the boundedness of the powered Hardy--Littlewood maximal
operator on the associate space of $Y(\cx)$,
and the Fefferman--Stein vector-valued
maximal inequality on $Y(\cx)$, we prove a ``reconstruction"
result of $H_{Y}^*(\cx)$ (see Proposition \ref{atre} below).
On the other hand, $Y(\cx)$ does not have an absolutely
continuous quasi-norm, which makes the good dense subspace of
$H_{Y}^*(\cx)$ be still unknown.
Thus, the proof of the ``decomposition" result
(see Proposition \ref{atde} below) is quite different from
the methods used in the proofs of \cite[Theorem 5.4]{fmy19}
and \cite[Theorem 4.2]{hhllyy}.
Indeed, we borrow some ideas from the proof of
\cite[Theorem 4.6]{zhy} to directly obtain the atomic
decomposition of distributions in $H_{Y}^*(\cx)$
instead of using some good dense function subspace. However,
duo to the lack of an explicit norm expression of $Y(\cx)$,
the methods used in the proof of the convergence of atomic
decompositions (see, for instance, \cite[(4.27) and (4.41)]{zhy})
are inapplicable anymore. To overcome this difficulty,
in Lemma \ref{embedding} below, we establish a continuous
embedding of $Y(\cx)$ into the weighted Lebesgue space
$L^p_{w}(\cx)$ with certain special $A_1(\cx)$ weight $w$,
which is an application of Theorem \ref{embed}.
From this and the Fefferman--Stein vector-valued maximal
inequality for $L^p_{w}(\cx)$, we prove the convergence
of Calder\'on--Zygmund decompositions for a distribution
in $H_{Y}^*(\cx)$ (see Lemma \ref{prop4.13}).
By using Lemma \ref{embedding} again,
and the known atomic characterization of the weighted
Hardy space which is a special case of the known
atomic characterization of the Musielak--Orlicz
Hardy space on $\cx$ obtained in \cite[Theorem 5.4]{fmy19},
we further prove the convergence of atomic
decompositions in \eqref{3.26.x1} and hence
complete the proof of Proposition \ref{atde}.
In Subsection \ref{sfinatom}, we further obtain the
finite atomic characterization of $H_{Y}^*(\cx)$.

The aim of Section \ref{s-mole} is to characterize $H_{Y}^*(\cx)$
by molecules (see Theorem \ref{mol} below).
We first introduce the notions of both the molecule
and the molecular Hardy space
$H_{\mathrm{mol}}^{Y,q,d,\ez}(\cx)$ in Definition \ref{mold}.
By the fact that each $(Y(\cx),q)$-atom is also a
$(Y(\cx),q,\ez)$-molecule, and Proposition \ref{atde}, we prove
$H_{Y}^*(\cx)\st H_{\mathrm{mol}}^{Y,q,d,\ez}(\cx)$.
On the other hand, we show that any $(Y(\cx),q,\ez)$-molecule
can be divided into an infinite linear combination of $(Y(\cx),q)$-atoms
in Lemma \ref{moldec}, which, combined with an argument
similar to that used in the proof of Proposition \ref{atre},
further implies that
$H_{\mathrm{mol}}^{Y,q,d,\ez}(\cx)\st H_{Y}^*(\cx)$
and hence completes the proof of the molecular characterization
of $H_{Y}^*(\cx)$.

Throughout Section \ref{s-lipa}, we always assume that
$Y(\cx)$ has an absolutely continuous quasi-norm.
Under this additional assumption, we first establish a
dominated convergence theorem on $Y(\cx)$, and then
prove that the finite atomic Hardy is dense in $H_{Y}^*(\cx)$
(see Lemma \ref{dense} below). From this, Theorem \ref{atthm},
and Theorem \ref{finatomeq}, we deduce that the dual space of
$H_{Y}^*(\cx)$ is the ball Campanato-type function space
${\cl}_{Y,q,d}(\cx)$ associated with $Y(\cx)$.
Remarkably, similarly to \cite{hyy21,zhyy21},
to overcome the difficulty caused by the fact that
the quasi-norm of $Y(\cx)$ may not be concave,
we consider the role as a whole of finite linear
combinations of atoms instead of the role of a single atom
in the construction of the ball Campanato-type function space;
this makes perfect sense because the finite atomic Hardy space
is dense in the corresponding Hardy space, and the most
important thing is that both of them have equivalent quasi-norms.

At the end of this section,
we make some conventions on notation.
Let $\nn:=\{1,2,\ldots\}$ and $\zz_+:=\nn\cup\{0\}$.
We denote by $C$ a \emph{positive constant} which is independent
of the main parameters, but may vary from line to line.
We use $C_{(\az,\dots)}$ to denote a positive constant depending
on the indicated parameters $\az,\, \dots$.
The symbol $f\ls g$ means $f\le Cg$
and, if $f\ls g\ls f$, then we write $f\sim g$.
If $f\le Cg$ and $g=h$ or $g\le h$,
we then write $f\ls g\sim h$ or $f\ls g\ls h$,
rather than $f\ls g=h$ or $f\ls g\le h$.
If $E$ is a subset of $\cx$, we denote by ${\mathbf{1}}_E$ its
\emph{characteristic function} and by $E^\complement$
the set $\cx\setminus E$.
For any $r\in(0,\fz)$ and $x\in\cx$, we denote by $B(x,r)$ the ball
centered at $x$ with the radius $r$, namely,
$B(x,r):=\{y\in\cx:\ \rho(x,y)<r\}.$
For any ball $B$, we use $x_B$ to denote its center and $r_B$ its radius,
and denote by $\lz B$ for any $\lz\in(0,\fz)$ the ball concentric with
$B$ having the radius $\lz r_B$. For any index $q\in[1,\fz]$,
we denote by $q'$ its \emph{conjugate index}, namely, $1/q+1/q'=1$.
For any $x$, $x_0$, $y\in\cx$ and $r,\ \gz\in(0,\infty)$,
let $V_r(x):=\mu(B(x,r))$,
\begin{align*}
V(x,y):=
\begin{cases}
\mu(B(x,\rho(x,y)))
\ \ &\text{if}\ x\neq y,\\
0 \ \ &\text{if}\ x=y,
\end{cases}
\end{align*}
and
\begin{align}\label{21.7.15.x1}
P_{\gz}(x_0,x;r):=\frac{1}{V_r(x_0)+V(x_0,x)}
\left[\frac{r}{r+\rho(x_0,x)}\right]^\gz.
\end{align}

\section{Ball quasi-Banach function spaces on
spaces of homogeneous type\label{s-pre}}

In this section, we first recall some basic notions about
spaces of homogeneous type and ball quasi-Banach function
spaces, respectively, in Subsections \ref{hts} and \ref{sbqbs},
which are used throughout this article.
Then we make some assumptions about the boundedness of
the Hardy--Littlewood maximal
operator in Subsection \ref{bshl}, and establish an embedding
theorem from the ball quasi-Banach function space to
the weighted Lebesgue space
on space of homogeneous type in Subsection \ref{emd},
which plays an important role in the remainder of this article.

\subsection{Spaces of homogeneous type\label{hts}}

In this subsection, we recall the notion of spaces of homogeneous type,
and some related basic estimates.

\begin{definition}\label{metric}
A \emph{quasi-metric space} $(\cx, \rho)$ is a non-empty set
$\cx$ equipped with a \emph{quasi-metric} $\rho$,
namely, a non-negative function defined
on $\cx\times\cx$ satisfying that, for any $x,\ y,\ z \in \cx$,
\begin{enumerate}
\item[{\rm(i)}] $\rho(x,y)=0$ if and only if $x=y$;
\item[{\rm(ii)}] $\rho(x,y)=\rho(y,x)$;
\item[{\rm(iii)}] there exists a constant $A_0 \in [1, \infty)$,
independent of $x$, $y$, and $z$, such that
\begin{align}\label{2.1x}
\rho(x,z)\leq A_0[\rho(x,y)+\rho(y,z)].
\end{align}
\end{enumerate}
\end{definition}

The \emph{ball} $B$ of $\cx$, centered at $x_0 \in \cx$
with radius $r\in(0, \infty)$, is defined by setting
$$B:=B(x_0,r):=\lf\{x\in\cx:\ \rho(x,x_0)<r\r\}.$$
For any ball $B$ and any $\tau\in(0,\infty)$,
we denote $B(x_0,\tau r)$ by $\tau B$ if
$B:=B(x_0,r)$ for some $x_0\in\cx$ and $r\in(0,\fz)$.
\begin{definition}\label{homo}
Let $(\cx,\rho)$ be a quasi-metric space and $\mu$
a non-negative measure on $\cx$.
The triple $(\cx,\rho,\mu)$ is called a
\emph{space of homogeneous type}
if $\mu$ satisfies the following doubling condition:
there exists a constant $C_{(\mu)}\in[1,\infty)$
such that, for any ball $B \subset \cx$,
\begin{equation}\label{eq-db1}
\mu(2B)\leq C_{(\mu)}\mu(B).
\end{equation}
\end{definition}
The above doubling condition implies that,
for any ball $B\st\cx$ and any $\lambda\in[1,\infty)$,
\begin{equation}\label{eq-doub}
\mu(\lambda B)\leq C_{(\mu)}\lambda^\omega\mu(B),
\end{equation}
where $\omega:= \log_2 C_{(\mu)}$ is called the
\emph{upper dimension} of $\cx$.
If $A_0 = 1$, then $(\cx,\rho,\mu)$ is called a
\emph{metric measure space of homogeneous type} or,
simply, a \emph{doubling metric measure space}.

Both spaces of homogeneous type, with some additional assumptions,
and function spaces on them have been extensively investigated in
many articles. One special case of spaces of homogeneous type
is the RD-\emph{space}, originally introduced in \cite{han08}
(see also \cite{han06,yz11}), which is a doubling
metric measure space satisfying the following additional
\emph{reverse doubling condition}: there exist constants
$\widetilde{C}_{(\mu)}\in(0,1]$ and $\kappa\in(0,\omega]$ such that,
for any ball $B(x, r)$ with $x\in\cx$ and $r\in(0,\diam\cx/2)$,
and any $\lambda\in[1,\diam\cx/(2r))$,
\begin{equation*}
\widetilde{C}_{(\mu)}\lambda^\kappa\mu(B(x,r))\leq\mu(B(x,\lambda r)),
\end{equation*}
here and thereafter, $\diam\cx:=\sup_{x,\,y\in \cx}\rho(x,y)$.

Throughout this article, according to \cite[pp.\,587-588]{cw77},
we always make the following assumptions on $(\cx,\rho,\mu)$:
\begin{enumerate}
\item[{\rm(i)}] for any point $x\in\cx$, the balls
$\{B(x,r)\}_{r\in(0,\fz)}$ form a basis of
open neighborhoods of $x$;
\item[{\rm(ii)}] $\mu$ is \emph{Borel regular}
which means that all open sets are $\mu$-measurable and every set
$A\st\cx$ is contained in a Borel set $E$ such that $\mu(A)=\mu(E)$;
\item[{\rm(iii)}] for any $x\in\cx$ and $r\in(0,\fz)$,
$\mu(B(x,r))\in(0,\fz)$;
\item[{\rm(iv)}] $\diam \cx=\fz$,
and $(\cx,\rho,\mu)$ is \emph{non-atomic},
which means $\mu(\{x\})=0$ for any $x\in\cx$.
\end{enumerate}
Notice that $\diam \cx=\fz$ implies that $\mu(\cx)=\fz$
(see \cite[p.\,284]{ah13} or \cite[Lemma 5.1]{ny97}).
From this, it follows that, under the above assumptions,
$\mu(\cx)=\fz$ if and only if $\diam \cx=\fz$.

The following basic estimates are from \cite[Lemma 2.1]{han06},
which can be proved by using \eqref{eq-doub}.
\begin{lemma}\label{equlem}
Let $x,\ y\in\cx$ and $r\in(0,\infty)$.
Then $V(x, y)\sim V (y, x)$ and
$$V_r(x)+V_r(y)+V (x,y)\sim V_r(x)+V (x,y)\sim V_r(y)+V (x,y)
\sim \mu(B(x,r+\rho(x,y))).$$
Moreover, if $\rho(x,y)\leq r$, then $V_r(x)\sim V_r(y)$.
Here the positive equivalence  constants are independent of
$x$, $y$, and $r$.
\end{lemma}

\subsection{Ball quasi-Banach function spaces\label{sbqbs}}

Throughout this article, let $(\cx,\rho,\mu)$ be a
space of homogeneous type with $\mu(\cx)=\fz$.
In this subsection, we introduce the notion of (ball)
quasi-Banach function spaces on $\cx$ (see \cite{shyy17}
for the corresponding Euclidean case). Moreover,
we give some examples of (ball) quasi-Banach function spaces.

\begin{definition}\label{qbs}
Let $Y(\cx)$ be a quasi-Banach space consisting of
$\mu$-measurable functions on $\cx$.
Then $Y(\cx)$ is called a \emph{quasi-Banach function space}
if it satisfies:

\begin{enumerate}
\item[(i)] $\|f\|_{Y(\cx)}=0$ implies that $f=0$
almost everywhere in $\cx$;

\item[(ii)] $|g|\le|f|$ almost everywhere in $\cx$
implies that $\|g\|_{Y(\cx)}\le\|f\|_{Y(\cx)}$;

\item[(iii)] $0\le f_m\uparrow f$ almost everywhere in $\cx$
implies that $\|f_m\|_{Y(\cx)}\uparrow\|f\|_{Y(\cx)}$;

\item[(iv)] $\ch1_{E}\in Y(\cx)$ for any $\mu$-measurable set $E\subset\cx$
with finite measure.
\end{enumerate}

Moreover, a quasi-Banach function space $Y(\cx)$
is called a \emph{Banach function space} if
it is a Banach space and
\begin{enumerate}
\item[(v)] for any given $\mu$-measurable set $E\subset\cx$
with finite measure, there exists a positive constant $C_{(E)}$
such that, for any $f\in Y(\cx)$,
$$\int_E|f(x)|\,d\mu(x)\le C_{(E)}\|f\|_{Y(\cx)}.$$
\end{enumerate}
\end{definition}

\begin{definition}\label{bqbs}
Let $Y(\cx)$ be a quasi-Banach space consisting of
$\mu$-measurable functions on $\cx$.
Then $Y(\cx)$ is called a \emph{ball quasi-Banach function space}
if it satisfies (i), (ii), and (iii) of Definition \ref{qbs},
and
\begin{enumerate}
\item[(iv)$'$] $\ch1_B\in Y(\cx)$ for any ball $B\subset\cx$.
\end{enumerate}

Moreover, a ball quasi-Banach function space $Y(\cx)$
is called a \emph{ball Banach function space} if it
is a Banach space and Definition \ref{qbs}(v) holds true
with any $\mu$-measurable set $E$ of finite measure replaced
by any ball $B$.
\end{definition}

Obviously, a quasi-Banach function space is
also a ball quasi-Banach function space and
the converse is not necessary to be true.

\begin{remark}\label{bqbsrem}
\begin{enumerate}
\item[{\rm (i)}] Let $Y(\cx)$ be a ball quasi-Banach
function space on $\cx$. By an elementary calculation,
we know that $\|f\|_{Y(\cx)}=0$ if and only if
$f=0$ almost everywhere in $\cx$.

\item[{\rm (ii)}] Observe that, in Definition \ref{bqbs},
if we replace any ball $B$ by any bounded
$\mu$-measurable set $E$,
we obtain its another equivalent formulation.
\end{enumerate}
\end{remark}

Next, we present several concrete examples of
ball quasi-Banach function spaces on $\cx$ as follows.
\begin{remark}\label{qbsdefrem}
\begin{enumerate}
\item[{\rm (i)}]
For any given $\mu$-measurable set $E\st\cx$ and
any given $p\in(0,\fz]$,
the \emph{Lebesgue space} $L^p(E)$ is defined by setting,
when $p\in(0,\infty)$,
$$L^p(E):=\left\{f\ \text{is $\mu$-measurable on}\ E:\
\|f\|_{L^p(E)}:=\left[\int_{E}|f(x)|^p\,d\mu(x)\right]^{1/p}
<\infty\right\},$$
and
$$L^\infty(E):=\left\{f\ \text{is $\mu$-measurable on}\ E:\ \|f\|_{L^\infty(E)}:=\displaystyle{\esssup_{x\in E}|f(x)|<\infty}\right\}.$$

It is easy to show that, for any given $\mu$-measurable set $E$
and any given $p\in(0,\fz]$,
$L^p(E)$ is a quasi-Banach function space and hence
a ball quasi-Banach function space.

\item[{\rm (ii)}]
Let $p\in(0,\fz]$ and $r\in(0,\fz]$.
Recall that the \emph{Lorentz space $L^{p,r}(\cx)$}
is defined to be the set of all $\mu$-measurable
functions $f$ on $\cx$ such that
\begin{align*}
\|f\|_{L^{p,r}(\cx)}:=
\begin{cases}\displaystyle\lf\{\int_0^{\fz}
\lf[t^{1/p}R(f)(t)\r]^r\frac{dt}{t}\r\}^{1/r}<\fz
\ \ &\text{when}\ r\in(0,\fz),\\
\displaystyle{\sup_{t\in(0,\fz)}}\lf[t^{1/p}R(f)(t)\r]<\fz \ \ &\text{when}\ r=\fz,
\end{cases}
\end{align*}
with the usual modification made when $p=\fz$,
where $R(f)$ denotes the \emph{decreasing rearrangement function}
of $f$ defined by setting, for any $t\in(0,\fz)$,
$$R(f)(t):=\inf\lf\{\az\in(0,\fz):\,d_{f}(\az)\le t\r\}$$
with $d_{f}(\az):=\mu(\{x\in\cx:\,|f(x)|>\az\})$
for any $\alpha\in(0,\infty)$.
By \cite[Theorem 1.4.11]{Gra1449}, we know that,
for any given $p\in(0,\fz]$ and $r\in(0,\fz]$,
$L^{p,r}(\cx)$ is a quasi-Banach function space
and hence a ball quasi-Banach function space.

\item[{\rm (iii)}]
Recall that a function $\Phi:\ [0,\fz)\rightarrow[0,\fz)$ is
called an \emph{Orlicz function} if it is non-decreasing,
$\Phi(0) = 0$, $\Phi(t)>0$ for any $t\in(0,\fz)$, and $\lim_{t\rightarrow\fz}\Phi(t)=\fz$. The function $\Phi$ is
said to be of \emph{upper} (resp., \emph{lower}) \emph{type $p$}
for some $p\in[0,\fz)$ if there exists a positive constant $C$
such that, for any $s\in[1,\fz)$ (resp., $s\in[0,1]$) and
$t\in[0,\fz)$, $\Phi(st)\le C s^p\Phi(t)$.

For a given function $\vz:\ \cx\times[0,\fz)\rightarrow[0,\fz)$
such that, for almost every $x\in\cx$,
$\vz(x,\cdot)$ is an Orlicz function,
$\vz$ is said to be of \emph{uniformly upper} (resp., \emph{lower})
type $p$ for some $p\in[0,\fz)$ if there exists a positive constant
$C$ such that, for any $x\in\cx$, $s\in[1,\fz)$ (resp., $s\in[0,1]$),
and $t\in[0,\fz)$, $\vz(x,st)\le C s^p\vz(x,t)$.

Let $\vz:\ \cx\times[0,\fz)\rightarrow[0,\fz)$ satisfy that
$x\mapsto\vz(x,t)$ is $\mu$-measurable for any $t\in[0,\fz)$.
The function $\vz(\cdot,t)$ is said to satisfy the
\emph{uniformly Muckenhoupt condition} for some $q\in[1,\fz)$,
denoted by $\vz\in\aa_q(\cx)$, if, when $q\in(1,\fz)$,
\begin{align*}
[\vz]_{\aa_q(\cx)}:=\sup_{t\in(0,\fz)}\sup_{B\subset\cx}
\frac{1}{[\mu(B)]^q}\int_B\vz(x,t)\,d\mu(x)
\lf\{\int_B\lf[\vz(y,t)\r]^{-q'/q}\,d\mu(y)\r\}^{q/q'}<\fz,
\end{align*}
or, when $q=1$,
\begin{align*}
[\vz]_{\aa_1(\cx)}:=\sup_{t\in(0,\fz)}\sup_{B\subset\cx}
\frac{1}{\mu(B)}\int_B\vz(x,t)\,d\mu(x)\lf\{\mathop\mathrm{ess\,sup}
_{y\in B}\lf[\vz(y,t)\r]^{-1}\r\}<\fz,
\end{align*}
where the second suprema are taken over all balls $B\subset\cx$
(see, for instance, \cite{hyy14}). Let
$$\aa_{\fz}(\cx):=\bigcup_{q\in[1,\fz)}\aa_{q}(\cx).$$
The \emph{critical weight index} $q(\vz)$ of $\vz\in\aa_{\fz}(\cx)$
is defined by setting
\begin{align}\label{qvz}
q(\vz):=\inf\lf\{q\in[1,\fz):\ \vz\in\aa_q(\cx)\r\}.
\end{align}
Recall that a function $\vz: \cx\times[0,\fz)\rightarrow[0,\fz)$
is called a \emph{growth function} if $\vz$ has the
following properties:
\begin{enumerate}
\item[(a)] $\vz$ is a \emph{Musielak--Orlicz function},
namely,
\begin{enumerate}
\item[$\mbox{(a)}_1$] the function $\vz(x,\cdot)$
is an Orlicz function for almost every given $x\in\cx$;

\item[$\mbox{(a)}_2$] the function $\vz(\cdot,t)$ is
a $\mu$-measurable function on $\cx$ for any given $t\in[0,\fz)$.
\end{enumerate}

\item[(b)] $\vz\in\aa_{\fz}(\cx)$.

\item[(c)] $\vz$ is of uniformly lower type
$p$ for some $p\in(0,1]$, and of uniformly upper type 1.
\end{enumerate}

Let $\vz$ be a growth function. The \emph{Musielak--Orlicz space}
$L^{\vz}(\cx)$ is defined to be the set of all $\mu$-measurable
functions $f$ on $\cx$ such that
$\int_{\cx}\vz(x,|f(x)|)\,d\mu(x)<\fz$
equipped with the \emph{quasi-norm}
\begin{align*}
\|f\|_{L^{\vz}(\cx)}:=\inf\lf\{\lz\in(0,\fz):\
\int_{\cx}\vz\lf(x,\frac{|f(x)|}{\lz}\r)\,d\mu(x)\le1\r\}.
\end{align*}
By the property of $\aa_{\fz}(\cx)$, we know that,
for any ball $B\subset\cx$, $\|\ch1_B\|_{L^{\vz}(\cx)}<\fz$,
which further implies that $\ch1_B\in L^{\vz}(\cx)$ and hence
$L^{\vz}(\cx)$ is a ball quasi-Banach function space.
However, by an argument similar to that used in \cite[p.\,86]{shyy17},
we know that, if we define the growth function $\vz_0$ by setting,
for any $x\in\rr$ and $t\in[0,\fz)$,
$\vz_0(x,t):=(1+|x|)t$, then $L^{\vz_0}(\rr)$
is \emph{not} a quasi-Banach function space.

\item[{\rm (iv)}] Let $p(\cdot):\ \cx\to(0,\fz)$ be a
$\mu$-measurable function satisfying
\begin{align}\label{2.4x}
0<\widetilde{p_-}:=\mathop\mathrm{ess\,inf}_{x\in\cx}p(x)\le
\mathop\mathrm{ess\,sup}_{x\in\cx}p(x)=:\widetilde{p_+}<\fz.
\end{align}
The \emph{variable Lebesgue space $L^{p(\cdot)}(\cx)$}
is defined to be the set of all $\mu$-measurable functions $f$
on $\cx$ such that $\int_\cx|f(x)|^{p(x)}\,d\mu(x)<\fz$
and equipped with the \emph{quasi-norm}
$$
\|f\|_{L^{p(\cdot)}(\cx)}:=\inf\lf\{\lambda\in(0,\infty):\
\int_\cx\lf[\frac{|f(x)|}{\lambda}\r]^{p(x)}\,d\mu(x)\le1\r\}.
$$
If $\widetilde{p_-}\in[1,\fz)$,
then $L^{p(\cdot)}(\cx)$ is a Banach function space
(see, for instance, \cite{dhr11}).
From this, we deduce that $L^{p(\cdot)/\widetilde{p_-}}(\cx)$
is a Banach function space and hence, for any given
$p(\cdot):\ \cx\to(0,\fz)$ and any given $\mu$-measurable set
$E\st\cx$ with finite measure,
$$\lf\|\ch1_E\r\|_{L^{p(\cdot)}(\cx)}
=\lf\|\ch1_E\r\|_{L^{p(\cdot)/\widetilde{p_-}}(\cx)}
^{1/\widetilde{p_-}}<\fz.$$
Thus, for any given $p(\cdot):\ \cx\to(0,\fz)$ satisfying \eqref{2.4x},
$L^{p(\cdot)}(\cx)$ is a quasi-Banach function space
and hence a ball quasi-Banach function space.
\end{enumerate}
\end{remark}

\begin{remark}\label{r-ar}
Let $Y(\cx)$ be a ball quasi-Banach function space on $\cx$.
Then, by the Aoki--Rolewicz theorem
(see \cite[Exercise 1.4.6]{Gra1449} and also \cite{ta42, sr57}),
we find that there exists a constant $v\in(0,1)$
such that, for any $N\in\nn$ and $\{f_k\}_{k=1}^N\st Y(\cx)$,
$$\lf\|\sum_{k=1}^N |f_k|\r\|_{Y(\cx)}^{v}
\le 4\lf[\sum_{k=1}^N \lf\|f_k\r\|_{Y(\cx)}^{v}\r].$$
\end{remark}

The following Fatou lemma of $Y(\cx)$ when $\cx:=\rn$
is just \cite[Chapter 1, Lemma 1.5(ii)]{bs88},
whose proof is a slight modification of
\cite[Chapter 1, Lemma 1.5(ii)]{bs88}; we omit the details here.
\begin{lemma}\label{fatou}
Let $Y(\cx)$ be a ball quasi-Banach function space
on $\cx$ and $\{f_k\}_{k\in\nn}\st Y(\cx)$.
If $f_k\to f$ almost everywhere in $\cx$ as $k\to\fz$,
and $\liminf_{k\to\fz}\|f_k\|_{Y(\cx)}$ is finite,
then $f\in Y(\cx)$ and
$$\|f\|_{Y(\cx)}\le\liminf_{k\to\fz}\|f_k\|_{Y(\cx)}.$$
\end{lemma}

Next, we introduce the notion of the convexification of $Y(\cx)$.
\begin{definition}\label{cvex}
Let $Y(\cx)$ be a ball quasi-Banach function space on $\cx$,
and $p\in(0,\fz)$. The \emph{$p$-convexification} $Y^p(\cx)$
of $Y(\cx)$ is defined by setting
$$Y^p(\cx):=\lf\{f\,\mbox{is $\mu$-measurable on}
\,\cx:\ |f|^p\in Y(\cx)\r\}$$
equipped with the \emph{quasi-norm}
$\|f\|_{Y^p(\cx)}:=\||f|^p\|_{Y(\cx)}^{\frac{1}{p}}$.
\end{definition}

\begin{remark}\label{covrem}
Let $p\in(0,\fz)$ and $Y(\cx)$ be a ball quasi-Banach
function space on $\cx$. By Definitions \ref{bqbs}
and \ref{cvex}, it is easy to show that $Y^p(\cx)$
is also a ball quasi-Banach function space on $\cx$;
see \cite[Lemma 2.6]{cwyz19} for the corresponding
Euclidean case.
\end{remark}

\subsection{Boundedness of the Hardy--Littlewood
maximal operator\label{bshl}}

Throughout this article, we always let $Y(\cx)$ be
a ball quasi-Banach function space on $\cx$.
In this subsection, we make two mild assumptions on the boundedness
of the Hardy--Littlewood maximal operator on $Y(\cx)$ and its associate
space. We first recall that the \emph{Hardy--Littlewood maximal operator}
$\cm$ is defined by setting, for any $\mu$-measurable
function $f$ and any $x\in\cx$,
\begin{align}\label{hlmax}
\cm(f)(x):=\sup_{B\ni x}\frac1{\mu(B)}\int_B |f(y)|\,d\mu(y),
\end{align}
where the supremum is taken over all balls $B$ of $\cx$ containing $x$.
For any given $\theta\in(0,\fz)$, the \emph{powered Hardy--Littlewood
maximal operator} $\cm^{(\tz)}$ is defined by setting, for any
$\mu$-measurable function $f$ and any $x\in\cx$,
\begin{align*}
\cm^{(\tz)}(f)(x):=\lf\{\cm\lf(|f|^{\tz}\r)(x)\r\}^{\frac{1}{\tz}}.
\end{align*}

Throughout this article, we denote by $M(\cx)$ the
\emph{set of all $\mu$-measurable functions} on $\cx$.
For any ball Banach function space $Y(\cx)$,
its \emph{associate space (K\"othe dual)} $Y'(\cx)$
is defined by setting
\begin{align}\label{1.7.x2}
Y'(\cx)&:=\lf\{f\in M(\cx):\ \|f\|_{Y'(\cx)}<\fz\r\},
\end{align}
where, for any $f\in M(\cx)$,
$$\|f\|_{Y'(\cx)}
:=\sup\lf\{\|fg\|_{L^1(\cx)}:\ g\in Y(\cx),\|g\|_{Y(\cx)}=1\r\}.$$

In this article, we need the following mild assumptions about
the boundedness of $\cm$ on the ball quasi-Banach function
space under consideration, and its associate space.

\begin{assumption}\label{assump1}
Let $Y(\cx)$ be a ball quasi-Banach function space on $\cx$.
Assume that there exists a $p_-\in(0,\fz)$ such that,
for any given $p\in(0,p_-)$ and $r\in(1,\fz)$,
there exists a positive constant C such that, for any
$\{f_k\}_{k=1}^{\fz}\subset M(\cx)$,
\begin{align*}
\lf\|\lf\{\sum_{k=1}^{\fz}\lf[\cm(f_k)\r]^r\r\}
^{\frac1r}\r\|_{{Y^{\frac{1}{p}}(\cx)}}\le
C\lf\|\lf\{\sum_{k=1}^{\fz}|f_k|^r\r\}
^{\frac1r}\r\|_{{Y^{\frac{1}{p}}(\cx)}}.
\end{align*}
\end{assumption}

In what follows, for any given $p_-\in(0,\fz)$, let
\begin{align}\label{2.1y}
\underline{p}:=\min\{p_-,1\}.
\end{align}

The following lemma is used throughout this article,
which is a trivial corollary of \eqref{hlmax};
we omit the details here.

\begin{lemma}\label{lem6.2}
For any $h\in(0,\fz)$, any $\az\in[1,\fz)$,
any ball $B\st\cx$, and any $x\in\cx$, it holds true that
\begin{align*}
\ch1_{\az B}(x)\le\lf[\frac{\mu(\az B)}{\mu(B)}\r]^{h}
\lf[\cm\lf(\ch1_B\r)(x)\r]^{h}.
\end{align*}
\end{lemma}

\begin{remark}\label{rek3.19}
Let $Y(\cx)$ be a ball quasi-Banach function space
on $\cx$ satisfying Assumption \ref{assump1},
$\az\in[1,\fz)$, and $h\in(0,\underline{p})$
with $\underline{p}$ as in \eqref{2.1y}.
Then, by Lemma \ref{lem6.2}, \eqref{eq-doub},
and Assumption \ref{assump1},
we conclude that there exists a positive
constant $C$, independent of $\az$, such that,
for any sequence $\{B_k\}_{k\in\nn}$ of balls of $\cx$,
\begin{align*}
\lf\|\sum_{k\in\nn}\ch1_{\az B_k}\r\|_{Y(\cx)}
&\le C\az^{\frac{\oz}{h}}\lf\|\sum_{k\in\nn}
\lf[\cm\lf(\ch1_{B_k}\r)\r]^{\frac{1}{h}}\r\|_{Y(\cx)}\\
&=C\az^{\frac{\oz}{h}}\lf\|\lf\{\sum_{k\in\nn}
\lf[\cm\lf(\ch1_{B_k}\r)\r]^{\frac{1}{h}}\r\}^h\r\|
^{\frac{1}{h}}_{Y^{\frac{1}{h}}(\cx)}\\
&\le C\az^{\frac{\oz}{h}}\lf\|\sum_{k\in\nn}\ch1_{B_k}\r\|_{Y(\cx)}.
\end{align*}
\end{remark}

\begin{assumption}\label{assump2}
Let $p_-\in(0,\fz)$ and $Y(\cx)$ be a ball
quasi-Banach function space on $\cx$.
Assume that there exist a $\tz_0\in(0,\underline{p})$
with $\underline{p}$ as in \eqref{2.1y}, and a $p_0\in(\tz_0,\fz)$
such that $Y^{1/{\tz_0}}(\cx)$ is a ball Banach function space and,
for any $f\in(Y^{1/{\tz_0}})'(\cx)$,
\begin{align}\label{5.14.y1}
\lf\|\cm^{((p_0/\tz_0)')}(f)\r\|_{(Y^{1/{\tz_0}})'(\cx)}
\le C\|f\|_{(Y^{1/{\tz_0}})'(\cx)},
\end{align}
where $C$ is a positive constant independent of $f$,
and $\frac{1}{p_0/\tz_0}+\frac{1}{(p_0/\tz_0)'}=1$.
\end{assumption}

\begin{remark}\label{21.1.7.x1}
Let $\tz_1,\ \tz_2\in(0,\fz)$ satisfy $\tz_1<\tz_2$.
By the H\"older inequality, we know that, for any
$\mu$-measurable function $f$ and any $x\in\cx$,
\begin{align}\label{1.7.x1}
\cm^{(\tz_1)}(f)(x)
&=\sup_{B\ni x}\lf\{\frac{1}{\mu(B)}\int_{B}
|f(y)|^{\tz_1}\,d\mu(y)\r\}^{\frac{1}{\tz_1}}\noz\\
&\le\sup_{B\ni x}\lf\{\frac{1}{\mu(B)}
\lf[\int_{B}|f(y)|^{\tz_2}\,d\mu(y)\r]^{\frac{\tz_1}{\tz_2}}
[\mu(B)]^{1-\frac{\tz_1}{\tz_2}}\r\}^{\frac{1}{\tz_1}}\noz\\
&=\sup_{B\ni x}\lf\{\frac{1}{\mu(B)}
\int_{B}|f(y)|^{\tz_2}\,d\mu(y)\r\}^{\frac{1}{\tz_2}}
=\cm^{(\tz_2)}(f)(x).
\end{align}
Let $Y(\cx)$ be a ball quasi-Banach function space on $\cx$
satisfying Assumption \ref{assump2} with some
$\tz_0\in(0,\underline{p})$ and $p_0\in(\tz_0,\fz)$.
Then, by \eqref{1.7.x1} and Definition \ref{qbs}(ii),
we conclude that there exists a positive constant $C$ such that,
for any given $p_1\in(p_0,\fz]$ and any $f\in(Y^{1/{\tz_0}})'(\cx)$,
\begin{align*}
\lf\|\cm^{((p_1/\tz_0)')}(f)\r\|_{(Y^{1/{\tz_0}})'(\cx)}
\le \lf\|\cm^{((p_0/\tz_0)')}(f)\r\|_{(Y^{1/{\tz_0}})'(\cx)}
\le C\|f\|_{(Y^{1/{\tz_0}})'(\cx)}.
\end{align*}
\end{remark}

\subsection{Embedding theorem for ball
quasi-Banach function spaces \label{emd}}
In this subsection, we aim to embed $Y(\cx)$
continuously into the weighted Lebesgue space on $\cx$.
To this end, we first recall some notions on
weighted Lebesgue spaces, and then state some basic properties
of $Y(\cx)$.

Throughout this article, for any given $p\in(0,\fz)$,
a function $f$ is said to be \emph{locally $p$-integrable}
if, for any $x\in\cx$, there exists an $r\in(0,\fz)$ such that
$$\int_{B(x,r)} |f(y)|^p\,d\mu(y)<\fz.$$
Denote by $L_{\rm loc}^p(\cx)$ the set of all
\emph{locally $p$-integrable functions} on $\cx$.

Let $q\in[1,\fz)$ and $w\in L^1_{\rm loc}(\cx)$ be a nonnegative
function. Then $w$ is said to be a \emph{Muckenhoupt $A_q$-weight},
denoted by $w\in A_q(\cx)$, if
\begin{align}\label{21.7.14.x1}
[w]_{A_q(\cx)}:=\sup_{B\subset\cx}\frac{1}{[\mu(B)]^q}
\int_Bw(x)\,d\mu(x)
\lf\{\int_B\lf[w(y)\r]^{-q'/q}\,d\mu(y)\r\}^{q/q'}<\fz
\end{align}
when $q\in(1,\fz)$, or
\begin{align*}
[w]_{A_1(\cx)}:=\sup_{B\subset\cx}\frac{1}{\mu(B)}\int_Bw(x)\,d\mu(x)
\lf(\mathop\mathrm{ess\,sup}_{y\in B}\lf[w(y)\r]^{-1}\r)<\fz,
\end{align*}
where the suprema are taken over all balls $B\subset\cx$
(see, for instance, \cite{st89}). Let
$$A_{\fz}(\cx):=\bigcup_{q\in[1,\fz)}A_{q}(\cx).$$
Recall that, for any given $p\in(0,\fz)$ and $w\in A_{\fz}(\cx)$,
the \emph{weighted Lebesgue space $L^{p}_w(\cx)$} is defined to be
the set of all $\mu$-measurable functions $f$ on $\cx$ such that
$$\|f\|_{L^{p}_w(\cx)}:=
\lf\{\int_{\cx}|f(x)|^pw(x)\,d\mu(x)\r\}^{1/p}<\infty.$$
Obviously, $L^{p}_w(\cx)$ is a ball quasi-Banach function space,
which even when $\cx:=\rn$ may not be a quasi-Banach function
space (see, for instance, \cite[p.\,86]{shyy17}).

Next, we state an embedding theorem from $Y(\cx)$ to
the weighted Lebesgue space (see \cite[Lemma 4.7]{cwyz19}
for the corresponding Euclidean case).
\begin{theorem}\label{embed}
Let $Y(\cx)$ be a ball quasi-Banach function space on $\cx$.
Assume that there exists an $r\in(0,\fz)$ such that
$Y^{1/{r}}(\cx)$ is a ball Banach function space and
that $\cm$ is bounded on $(Y^{1/{r}})'(\cx)$. Fix $x_0\in\cx$.
Then there exists an $\ez\in(0,1)$ such that $Y(\cx)$ is
continuously embedded into $L^{r}_{w}(\cx)$ with
$w:=[\cm(\ch1_{B(x_0,1)})]^{\ez}\in A_1(\cx)$. Moreover,
there exists a positive constant $C$, depending on $x_0$,
such that, for any $f\in Y(\cx)$,
\begin{align}\label{embd}
\lf\|f\r\|_{L^{r}_{w}(\cx)}\le C\|f\|_{Y(\cx)}.
\end{align}
\end{theorem}

Theorem \ref{embed} is the main result of this subsection.
To prove it, we borrow some ideas from the proof of
\cite[Lemma 4.7]{cwyz19} and establish several technical
lemmas, some of which are used again in later sections.
By Definition \ref{qbs}(i) and \eqref{1.7.x2},
we easily obtain the following H\"older inequality
on $Y(\cx)$ (see \cite[Chapter 1, Theorem 2.4]{bs88} for the
corresponding Euclidean case),
whose proof is a slight modification of
\cite[Chapter 1, Theorem 2.4]{bs88};
we omit the details here.
\begin{lemma}\label{holder}
Let $Y(\cx)$ be a ball Banach function space on $\cx$,
and $Y'(\cx)$ the associate space of $Y(\cx)$.
If $f\in Y(\cx)$ and $g\in Y'(\cx)$,
then $fg$ is integrable and
\begin{align*}
\int_{\cx}|f(x)g(x)|\,d\mu(x)\le\|f\|_{Y(\cx)}\|g\|_{Y'(\cx)}.
\end{align*}
\end{lemma}

For the corresponding Euclidean case, the following lemma is just
\cite[Proposition 2.3]{shyy17}, whose proof is a slight modification
of \cite[Proposition 2.3]{shyy17}; we omit the details here.
\begin{lemma}\label{assoalso}
Let $Y(\cx)$ be a ball Banach function space on $\cx$.
Then its associate space $Y'(\cx)$ is also a ball Banach function space.
\end{lemma}

By a method similar to that used in \cite[Chapter 1, Theorem 2.7]{bs88},
and Lemmas \ref{holder} and \ref{assoalso},
we have the following conclusion, which is a generalization
of \cite[Lemma 2.6]{zwyy} on $\rn$ to $\cx$;
we omit the details here.
\begin{lemma}\label{second}
Every ball Banach function space $Y(\cx)$ coincides with its
second associate space $Y''(\cx)$. More precisely,
a function $f$ belongs to $Y(\cx)$ if and only if it
belongs to $Y''(\cx)$. Moreover, for any $f\in Y(\cx)$
or $f\in Y''(\cx)$,
$$\|f\|_{Y(\cx)}=\|f\|_{Y''(\cx)}.$$
\end{lemma}

The following lemma is a generalization of
\cite[Lemma 2.15(ii)]{shyy17} on $\rn$ to $\cx$,
whose proof is also valid in this setting;
we omit the details here.
\begin{lemma}\label{embedlem1}
Assume that $Y(\cx)$ is a ball quasi-Banach function
space on $\cx$, and $\cm$ bounded on $Y(\cx)$.
Then there exists an $\eta\in(1,\fz)$ such that
$\cm^{(\eta)}$ is bounded on $Y(\cx)$.
\end{lemma}

The following lemma establishes the pointwise estimates of
$\cm(\mathbf 1_{B(x_0,1)})$ with $x_0\in\cx$ and $\cm$
as in \eqref{hlmax}, which can be regarded as a generalization
of \cite[(2.1.6)]{Gra1449} on $\rn$ to $\cx$.

\begin{lemma}\label{lem-phl}
Let $x_0\in\cx$ and $\cm$ be as in \eqref{hlmax}.
Then there exists a constant $C\in[1,\fz)$,
independent of $x_0$, such that, for any $x\in\cx$,
\begin{equation}\label{eq-phl}
C^{-1}\min\lf\{1,\frac{V_1(x_0)}{V(x_0,x)}\r\}
\le\cm(\mathbf 1_{B(x_0,1)})(x)
\le C\min\lf\{1,\frac{V_1(x_0)}{V(x_0,x)}\r\}.
\end{equation}
\end{lemma}

\begin{proof}
Let all the symbols be as in the present lemma.
To obtain \eqref{eq-phl},
we consider two cases on $\rho(x_0,x)$.

{\it Case 1)} $\rho(x_0,x)\in[0,2A_0)$. In this case,
it suffices to show that $\cm(\mathbf 1_{B(x_0,1)})(x)\sim 1$.
On one hand, by \eqref{hlmax}, we have
$$
\cm\lf(\mathbf 1_{B(x_0,1)}\r)(x)\le 1.
$$
On the other hand, by $\rho(x_0,x)<2A_0$, \eqref{hlmax},
and \eqref{eq-doub}, we conclude that
\begin{equation*}
\cm(\ch 1_{B(x_0,1)})(x)
\ge\frac{1}{\mu(B(x_0,2A_0))}\int_{B(x_0,2A_0)}
\ch 1_{B(x_0,1)}(y)\,d\mu(y)
=\frac{\mu(B(x_0,1))}{\mu(B(x_0,2A_0))}\gtrsim 1.
\end{equation*}
Combining the above two inequalities,
we obtain the desired estimates in this case.

{\it Case 2)} $\rho(x_0,x)\in[2A_0,\fz)$.
In this case, we only need to show that
\begin{align}\label{5.7.x1}
\cm\lf(\mathbf 1_{B(x_0,1)}\r)(x)\sim\frac{V_1(x_0)}{V(x_0,x)}.
\end{align}
To this end, we first claim that
\begin{equation}\label{eq-cl1}
B(x_0,1)\cap B(x,\rho(x_0,x)/[2A_0])=\emptyset.
\end{equation}
Indeed, if there exists a $y\in B(x_0,1)\cap B(x,\rho(x_0,x)/[2A_0])$,
then, by the quasi-triangle inequality and $\rho(x_0,x)\ge2A_0$,
we conclude that
$$
\rho(x_0,x)\le A_0[\rho(x_0,y)+\rho(y,x)]
<A_0\lf[1+\frac{\rho(x_0,x)}{2A_0}\r]\le\rho(x_0,x),
$$
which makes a contradiction. Thus, \eqref{eq-cl1} holds true.
On one hand, by \eqref{eq-cl1}, \eqref{eq-doub},
and Lemma \ref{equlem}, we obtain
\begin{align}\label{5.7.x2}
\cm(\ch 1_{B(x_0,1)})(x)
&\sim\sup_{r\in(0,\fz)}\frac{1}{\mu(B(x,r))}\int_{B(x,r)}
\ch 1_{B(x_0,1)}(y)\,d\mu(y)\noz\\
&\sim\sup_{r\in(\rho(x_0,x)/[2A_0],\fz)}\frac{\mu(B(x,r)\cap B(x_0,1))}{\mu(B(x,r))}\noz\\
&\ls\frac{\mu(B(x_0,1))}{\mu(B(x,\rho(x_0,x)/[2A_0]))}
\sim\frac{\mu(B(x_0,1))}{\mu(B(x_0,\rho(x_0,x)))}
\sim\frac{V_1(x_0)}{V(x_0,x)}.
\end{align}

On the other hand, observing that $x\in B(x_0,2\rho(x_0,x))$,
form this, \eqref{hlmax}, and \eqref{eq-doub},
it follows that
\begin{align*}
\cm(\ch 1_{B(x_0,1)})(x)
&\ge\frac{1}{\mu(B(x_0,2\rho(x_0,x)))}
\int_{B(x_0,2\rho(x_0,x))}\ch 1_{B(x_0,1)}(y)\,d\mu(y)\\
&=\frac{\mu(B(x_0,1))}{\mu(B(x_0,2\rho(x_0,x)))}
\sim\frac{V_1(x_0)}{V(x_0,x)},
\end{align*}
which, combined with \eqref{5.7.x2}, further implies \eqref{5.7.x1}.
Thus, we find that \eqref{eq-phl} also holds true in this case.
This finishes the proof of Lemma \ref{lem-phl}.
\end{proof}

By borrowing some ideas from \cite{glpv20}, we establish the following
lemma, which is regarded as a substitute of the reverse doubling
condition and can be used to deal with integrals only in terms
of measures of balls.

\begin{lemma}\label{lem-rdc}
Let $\cx$ be a space of homogeneous type with $\mu(\cx)=\fz$,
$c\in(1,\fz)$, $x_0\in\cx$, and $r_0\in(0,\fz)$.
Then there exists a sequence $\{r_k\}_{k=1}^\fz\subset(0,\fz)$ such that,
for any $k\in\nn$,
\begin{enumerate}
\item[\rm (i)] $r_k\ge 2r_{k-1}$;
\item[\rm (ii)]
$c\mu(B(x_0,r_{k-1}))<\mu(B(x_0,r_k))
\le cC_{(\mu)}\mu(B(x_0,r_{k-1}))$,
where $C_{(\mu)}\in(1,\fz)$ is as in \eqref{eq-db1}.
\end{enumerate}
\end{lemma}

\begin{proof}
Let all the symbols be as in the present lemma. We use an inductive
method to construct $\{r_k\}_{k=1}^\fz$.
For any fixed $k\in\nn$, if $r_{k-1}\in(0,\fz)$ has been constructed,
we then let

$$
A:=\lf\{r\in(0,\fz):\ \mu(B(x_0,r))
\le cC_{(\mu)}\mu(B(x_0,r_{k-1}))\r\}
\quad\mbox{and}\quad r_k:=\sup A.
$$
First, by $r_{k-1}\in A$, we know that
$r_k\ge r_{k-1}$ and hence $r_k>0$.
Second, if $r_k=\fz$, then, from the definition of $r_k$,
$\mu(\cx)=\fz$, and $r_{k-1}<\fz$, we infer that
there exists an increasing sequence
$\{\tau_n\}_{n=1}^\fz\subset A$ such that
$\lim_{n\to\fz}\tau_n=\fz$ and
$$\fz=\lim_{n\to\fz}\mu(B(x_0,\tau_n))
\le cC_{(\mu)}\mu(B(x_0,r_{k-1}))<\fz,$$
which makes a contradiction. Thus, $r_k<\fz$.
Finally, by \eqref{eq-db1}, we find that $2r_{k-1}\in A$
which implies that $r_{k}\ge 2r_{k-1}$. This proves (i).

Now, we prove (ii). Indeed, by the definition of $r_k$,
we know that there exists an increasing sequence
$\{\tau_n\}_{n=1}^\fz\subset A$
such that $\lim_{n\to\fz}\tau_n=r_k$,
which further implies that
$$
\mu(B(x_0,r_k))=\lim_{n\to\fz}\mu(B(x_0,\tau_n))
\le cC_{(\mu)}\mu(B(x_0,r_{k-1})).
$$
On the other hand, note that $2r_k>r_k$. Thus, $2r_k\notin A$.
From this and \eqref{eq-db1}, we deduce that
$$
\mu(B(x_0,r_k))\ge C_{(\mu)}^{-1}\mu(B(x_0,2r_k))
>c\mu(B(x_0,r_{k-1})).
$$
These two inequalities then complete the proof of (ii)
and hence of Lemma \ref{lem-rdc}.
\end{proof}

\begin{definition}\label{admbdef}
Let $\cx$ be a space of homogeneous type with $\mu(\cx)=\fz$,
$c\in(1,\fz)$, $x_0\in\cx$, and $r_0\in(0,\fz)$.
A sequence $\{B_k\}_{k\in\zz_+}:=\{B(x_0,r_k)\}_{k\in\zz_+}$ of balls
centered at $x_0$, associated with $\mu$,
$c$, and $r_0$, is said to be \emph{admissible} if $\{r_k\}_{k\in\zz_+}$
satisfy (i) and (ii) of Lemma \ref{lem-rdc}.
\end{definition}

\begin{remark}\label{r2.25}
\begin{itemize}
\item[(i)] Obviously, for any given $\mu$ and $c$ as
in Definition \ref{admbdef}, and any $x_0\in\cx$
and $r_0\in (0,\fz)$, Lemma \ref{lem-rdc}
guarantees that there exists an admissible sequence
$\{B_k\}_{k\in\zz_+}$ of balls centered at $x_0$,
associated with $\mu$, $c$, and $r_0$.

\item[(ii)] Observe that any admissible sequence $\{B_k\}_{k\in\zz_+}$
of balls satisfies both the doubling and the reverse doubling conditions,
which is the key tool used in the proof of Theorem \ref{embed}
to escape the dependence on the reverse doubling assumption of
the measure $\mu$ under consideration.
We believe that more applications are expectable.
\end{itemize}
\end{remark}

Now, we give the proof of Theorem \ref{embed}.

\begin{proof}[Proof of Theorem \ref{embed}]
Let all the symbols be as in the present theorem.
By Lemma \ref{assoalso}, we know that $(Y^{1/{r}})'(\cx)$
is a ball Banach function space. From this,
the fact that $\cm$ is bounded on $(Y^{1/{r}})'(\cx)$,
and Lemma \ref{embedlem1}, it follows that there exists an
$\eta\in(1,\fz)$ such that $\cm^{(\eta)}$ is bounded
on $(Y^{1/{r}})'(\cx)$. Let $\ez\in(1/\eta,1)$
and $w:=[\cm(\ch1_{B(x_0,1)})]^{\ez}$.
We first prove that $w\in A_1(\cx)$.
To this end, from the definition of $A_1(\cx)$,
it suffices to show that there exists a positive constant
$C$ such that, for any $B:=B(x_B,r_B)\st\cx$ with $x_B\in\cx$
and $r_B\in(0,\fz)$, and almost every $x\in B$,
\begin{align}\label{6.11.x1}
\frac{1}{\mu(B)}\int_Bw(y)\,d\mu(y)\le Cw(x).
\end{align}
Decompose $\ch1_{B(x_0,1)}$ as
\begin{align}\label{6.11.x2}
\ch1_{B(x_0,1)}&=\ch1_{B(x_0,1)\cap 2A_0B}
+\ch1_{B(x_0,1)\cap(2A_0B)^{\complement}}
=:f_1+f_2.
\end{align}

On one hand, by the boundedness of $\cm$ from $L^1(\cx)$
to $L^{1,\fz}(\cx)$ (see, for instance, \cite[Theorem (3.5)]{cw77}),
\eqref{eq-db1}, $\ez<1$, and \eqref{hlmax}, we know that,
for any $x\in B$,
\begin{align}\label{6.11.x3}
&\frac{1}{\mu(B)}\int_B\lf[\cm\lf(f_1\r)(y)\r]^{\ez}\,d\mu(y)\noz\\
&\quad=\frac{\ez}{\mu(B)}\int_0^{\fz}\lz^{\ez-1}
\mu\lf(\lf\{y\in B:\ \cm\lf(f_1\r)(y)>\lz\r\}\r)\,d\lz\noz\\
&\quad\ls\frac{\ez}{\mu(B)}\int_0^{\fz}\lz^{\ez-1}
\min\lf\{\mu(B),\frac{\|f_1\|_{L^1(\cx)}}
{\lz}\r\}\,d\lz\noz\\
&\quad\ls\frac{\ez}{\mu(B)}
\lf[\int_0^{\|f_1\|_{L^1(\cx)}/{\mu(B)}}\lz^{\ez-1}\mu(B)\,d\lz
+\int_{\|f_1\|_{L^1(\cx)}/{\mu(B)}}^{\fz}\lz^{\ez-2}
\|f_1\|_{L^1(\cx)}\,d\lz\r]\noz\\
&\quad\ls\|f_1\|_{L^1(\cx)}^{\ez}[\mu(B)]^{-\ez}
\ls\lf[\frac{1}{\mu(2A_0B)}\int_{2A_0B}
\ch1_{B(x_0,1)}(y)\,d\mu(y)\r]^{\ez}\ls w(x),
\end{align}
where the implicit positive constants are independent of $B$ and $x$.

On the other hand, let $z\in B$ and
$\widetilde{B}:=B(x_{\widetilde{B}},r_{\widetilde{B}})\st\cx$
be such that $z\in\widetilde{B}$ and
$\int_{\widetilde{B}}f_2(x)\,d\mu(x)>0$. Then
\begin{align}\label{6.14.x1}
r_B<2A_0r_{\widetilde{B}}.
\end{align}
Otherwise,
by the fact that $\int_{\widetilde{B}}f_2(x)\,d\mu(x)>0$, we know that
there exists a $u\in \widetilde{B}\cap(2A_0B)^{\complement}$,
and hence
\begin{align*}
2A_0r_B\le\rho(u,x_B)\le A_0[\rho(u,z)+\rho(z,x_B)]
< A_0\rho(u,z)+A_0r_B,
\end{align*}
which further implies that
\begin{align*}
r_B<\rho(u,z)\le A_0[\rho(u,x_{\widetilde{B}})+\rho(x_{\widetilde{B}},z)]
<2A_0r_{\widetilde{B}}\le r_B.
\end{align*}
This is a contradiction and hence \eqref{6.14.x1} holds true.
By \eqref{6.14.x1}, we conclude that, for any $v\in B$,
\begin{align*}
\rho(v,x_{\widetilde{B}})
&\le A_0[\rho(v,x_B)+\rho(x_B,x_{\widetilde{B}})]
\le A_0\{\rho(v,x_B)+
A_0[\rho(x_B,z)+\rho(z,x_{\widetilde{B}})]\}\\
&<(A_0+A_0^2)r_B+A_0^2r_{\widetilde{B}}
<(2A_0+3)A_0^2r_{\widetilde{B}}.
\end{align*}
This further implies that $v\in(2A_0+3)A_0^2\widetilde{B}$,
and hence $B\st(2A_0+3)A_0^2\widetilde{B}$ by the arbitrariness
of $v\in B$. From this and \eqref{eq-doub},
we deduce that, for any $x\in B$,
\begin{align*}
\frac{1}{\mu(\widetilde{B})}\int_{\widetilde{B}}f_2(y)\,d\mu(y)
&\ls\frac{1}{\mu((2A_0+3)A_0^2\widetilde{B})}
\int_{(2A_0+3)A_0^2\widetilde{B}}f_2(y)\,d\mu(y)
\ls\cm\lf(\ch1_{B(x_0,1)}\r)(x).
\end{align*}
Thus, by the arbitrariness of $\widetilde{B}$,
we find that, for any $x\in B$,
\begin{align*}
\cm\lf(f_2\r)(z)\ls\cm\lf(\ch1_{B(x_0,1)}\r)(x),
\end{align*}
which further implies that
\begin{align}\label{6.11.x6}
\frac{1}{\mu(B)}\int_B\lf[\cm\lf(f_2\r)(z)\r]^{\ez}\,d\mu(z)
\ls w(x),
\end{align}
where the implicit positive constant is independent of $B$ and $x$.

Combining \eqref{6.11.x2}, \eqref{6.11.x3}, and \eqref{6.11.x6},
we conclude that, for any $B\st\cx$ and $x\in B$,
\begin{align*}
\frac{1}{\mu(B)}\int_Bw(y)\,d\mu(y)
\ls\frac{1}{\mu(B)}\int_B\lf\{\lf[\cm\lf(f_1\r)(y)\r]^{\ez}
+\lf[\cm\lf(f_2\r)(y)\r]^{\ez}\r\}\,d\mu(y)\ls w(x).
\end{align*}
This proves \eqref{6.11.x1} and hence $w\in A_1(\cx)$.

Next, we prove \eqref{embd}.
Let $r_0=1$, $c\in(1,\fz)$, $x_0$ be as in the present theorem,
and $\{B_k\}_{k\in\zz_+}$ an admissible sequence of balls
as in Definition \ref{admbdef}. From Lemma \ref{lem-rdc}(i),
we deduce that, for any $f\in Y(\cx)$,
\begin{align}\label{embddec}
\lf\|f\r\|_{L^{r}_{w}(\cx)}^r&=\int_{\cx}|f(x)|^rw(x)\,d\mu(x)\noz\\
&=\int_{B_0}|f(x)|^rw(x)\,d\mu(x)
+\sum_{k=1}^{\fz}\int_{B_k\setminus B_{k-1}}\cdots\noz\\
&=:{\rm I}_1+{\rm I}_2.
\end{align}

We first estimate ${\rm I}_1$. Indeed, by Lemmas \ref{lem-phl},
\ref{holder}, and \ref{assoalso}, and Definitions \ref{cvex}
and \ref{qbs}(iv),
we know that
\begin{align}\label{embdi1}
{\rm I}_1&\ls\int_{B_0}|f(x)|^r\,d\mu(x)
\ls\lf\||f|^r\r\|_{Y^{\frac1{r}}(\cx)}
\lf\|\ch1_{B_0}\r\|_{(Y^{\frac1{r}})'(\cx)}
\sim\|f\|_{Y(\cx)}^r.
\end{align}

Next, we estimate ${\rm I}_2$. By Lemmas \ref{lem-phl} and
\ref{lem-rdc}(ii), we know that,
for any $k\in\nn$ and $x\in B_k\setminus B_{k-1}$,
\begin{align}\label{1.8.x1}
w(x)\ls\lf[\frac{\mu(B_0)}{\mu(B_{k-1})}\r]^\ez
\sim\lf[\frac{\mu(B_0)}{\mu(B_{k})}\r]^\ez.
\end{align}
On the other hand, from the definition of $\cm^{(\eta)}$,
we deduce that, for any $k\in\nn$ and $x\in B_k$,
$$
\cm^{(\eta)}(\ch 1_{B_0})(x)
\ge\lf[\frac{\mu(B_0)}{\mu(B_k)}\r]^{1/\eta},
$$
which further implies that, for any $x\in\cx$,
\begin{equation*}
\ch 1_{B_k}(x)
\le\lf[\frac{\mu(B_0)}{\mu(B_k)}\r]^{-1/\eta}
\cm^{(\eta)}(\ch 1_{B_0})(x).
\end{equation*}
By this, \eqref{1.8.x1}, Lemma \ref{holder}, Definition \ref{cvex},
the boundedness of $\cm^{(\eta)}$ on $(Y^{1/{r}})'(\cx)$,
Lemmas \ref{assoalso} and \ref{lem-rdc}(ii),
Definition \ref{qbs}(iv), and $\ez>1/\eta$,
we conclude that
\begin{align*}
{\rm I}_2&\ls\sum_{k=1}^{\fz}
\lf[\frac{\mu(B_0)}{\mu(B_k)}\r]^{\ez}
\int_{B_k}|f(x)|^r\,d\mu(x)
\ls\sum_{k=1}^{\fz}
\lf[\frac{\mu(B_0)}{\mu(B_k)}\r]^{\ez}
\lf\|\ch1_{B_k}\r\|_{(Y^{\frac1{r}})'(\cx)}
\|f\|_{Y(\cx)}^r\\
&\ls\sum_{k=1}^{\fz}
\lf[\frac{\mu(B_0)}{\mu(B_k)}\r]^{\ez-1/\eta}
\lf\|\ch1_{B_0}\r\|_{(Y^{\frac1{r}})'(\cx)}
\|f\|_{Y(\cx)}^r
\ls\|f\|_{Y(\cx)}^r\sum_{k=1}^{\fz}c^{-(k-1)(\ez-1/\eta)}
\ls\|f\|_{Y(\cx)}^r.
\end{align*}
This, combined with \eqref{embddec} and \eqref{embdi1},
implies \eqref{embd} and hence finishes
the proof of Theorem \ref{embed}.
\end{proof}

\begin{remark}\label{weight}
We point out that the proof of $w\in A_1(\cx)$ in
Theorem \ref{embed} is a slight modification of the
proof of \cite[Lemma 3.3]{am97}.
We still give some details here for the sake of completeness.
\end{remark}

\section{Maximal function characterizations of $H_Y^*(\cx)$\label{s-max}}
%\hskip\parindent
In this section, we first introduce the Hardy space
$H_Y^*(\cx)$ associated with ball quasi-Banach function
space $Y(\cx)$ via using the grand maximal function,
then discuss the relations between $Y(\cx)$ and $H_Y^*(\cx)$,
and finally establish the radial and the non-tangential
maximal function characterizations of $H_Y^*(\cx)$.

\subsection{Relations between $Y(\cx)$ and $H_Y^*(\cx)$\label{s-max1}}

The aim of this subsection is to clarify the relations
between $Y(\cx)$ and $H_Y^*(\cx)$. To this end,
we first recall the notions of the spaces of test
functions and distributions on $\cx$,
which were originally introduced by Han et al.
\cite[Definition 2.2]{han06} (see also \cite[Definition 2.8]{han08}).
\begin{definition}[test functions]\label{test}
Let $x_0\in\cx$, $r\in(0,\infty)$, $\beta \in (0,1]$,
and $\gamma \in (0,\infty)$.
A function $f$ on $\cx$ is called a \emph{test function of type}
$(x_0, r, \beta, \gamma)$, denoted by $f \in \cg(x_0,r,\beta,\gamma)$,
if there exists a positive constant $C$ such that
\begin{enumerate}
\item[{\rm(i)}] (the \emph{size condition}) for any $x\in\cx$,
\begin{equation}\label{3.31.x1}
|f(x)|\leq CP_{\gz}(x_0,x;r),
\end{equation}
here and thereafter, $P_{\gz}(x_0,x;r)$
is as in \eqref{21.7.15.x1};
\item[{\rm(ii)}] (the \emph{regularity condition})
for any $x,\ y \in \cx$ satisfying
$\rho(x, y)\leq(2A_0)^{-1}[r + \rho(x_0, x)]$,
\begin{equation}\label{3.31.x2}
|f(x)-f(y)|\leq C\left[\frac{\rho(x,y)}{r+\rho(x_0,x)}\right]^\beta
P_{\gz}(x_0,x;r).
\end{equation}
\end{enumerate}
For any  $f\in \cg(x_0,r,\beta,\gamma)$,
its \emph{norm} $\|f\|_{\cg(x_0,r,\beta,\gamma)}$ in
$ \cg(x_0,r,\beta,\gamma)$ is defined by setting
$$\|f\|_{\cg(x_0, r, \beta, \gamma)}:=\inf\lf\{C\in(0,\infty):\
C\ \text{satisfies \eqref{3.31.x1} and \eqref{3.31.x2}}\r\}.$$
\end{definition}

Fix $x_0\in\cx$.
We denote $\cg(x_0, 1, \beta, \gamma)$ simply by $\cg(\beta,\gamma)$.
Obviously, $\cg(\beta,\gamma)$ is a Banach space.
Note that, for any fixed $x\in\cx$ and $r \in(0,\infty)$,
$\cg(x,r,\beta,\gamma)=\cg(\beta,\gamma)$
with equivalent norms, but the positive equivalence
constants may depend on $x$ and $r$.

Fix $\epsilon\in (0, 1]$ and $\beta,\ \gamma \in (0, \epsilon]$.
Let $\cg^\epsilon_0(\beta,\gamma)$ be the completion
of the set $\cg(\epsilon, \epsilon)$ in $\cg(\beta,\gamma)$.
Furthermore,
the norm of $\cg^\epsilon_0(\beta,\gamma)$ is defined by setting
$\|\cdot\|_{\cg^\epsilon_0(\beta,\gamma)}:=\|\cdot\|_{\cg(\beta,\gamma)}$.
The space $\cg^\epsilon_0(\beta,\gamma)$
is called the \emph{space of test functions}.
The \emph{dual space} $(\cg^\epsilon_0(\beta,\gamma))'$
is defined to be the set of all continuous linear functionals from
$\cg^\epsilon_0(\beta,\gamma)$
to $\mathbb{C}$, equipped with the weak-$\ast$ topology.
The space $(\cg^\epsilon_0(\beta,\gamma))'$
is called the \emph{space of distributions}.

The following system of dyadic cubes of $(\cx,\rho,\mu)$
was established by Hyt\"onen and Kairema in \cite[Theorem 2.2]{hk12}.
\begin{lemma}\label{daydic}
Suppose that constants $0 < c_0\leq C_0 <\infty$ and $\delta\in(0, 1)$
satisfy $12A_0^3C_0\delta\leq c_0$. Assume that a set of points,
$\{ z_\alpha^k : k \in\zz , \alpha \in \ca_k \} \subset \cx$ with $\ca_k$,
for any $k \in\zz$, being a set of indices, has the following properties:
for any $k \in\zz$,
$$\rho(z^k_\alpha, z^k_\beta)\geq c_0\delta^k\quad\text{if}\quad\alpha\neq\beta,\quad
\text{and}\quad\min_{\alpha\in\ca_k}\rho(x,z_\alpha^k)<C_0\delta^k
\quad\text{for any}\quad x\in\cx.$$
Then there exists a family of sets,
$\{ Q_\alpha^k : k \in\zz , \alpha \in \ca_k \}$,
satisfying
\begin{enumerate}
\item[{\rm(i)}] for any $k\in\zz$, $\bigcup_{\alpha\in\ca_k}
Q_\alpha^k=\cx$ and $\{ Q_\alpha^k : \alpha \in \ca_k \}$
is disjoint for any given $k\in\zz$;
\item[{\rm(ii)}] if $k,\ l\in\zz$ and $k\leq l$, then,
for any $\alpha\in\ca_k$ and $\beta\in\ca_l$,
either $Q_\beta^l\subset Q_\alpha^k$ or
$Q_\beta^l\cap Q_\alpha^k=\emptyset$;
\item[{\rm(iii)}] for any $k\in\zz$ and $\alpha\in\ca_k$,
$B(z^k_\alpha, (3A_0^2)^{-1}c_0\delta^k)\subset Q_\alpha^k\subset
B(z^k_\alpha, 2A_0C_0\delta^k)$.
\end{enumerate}
\end{lemma}

Throughout this article, for any $k\in\zz$, define
\begin{align}\label{21.7.16.x1}
\cg_k:=\ca_{k+1}\setminus\ca_k\quad\text{and}\quad
\cy^k:=\left\{z^{k+1}_\alpha\right\}_{\alpha\in\cg_k}
=: \left\{y_\alpha^k\right\}_{\alpha\in\cg_k}
\end{align}
and, for any $x \in \cx$, define
\begin{align}\label{21.7.16.x2}
\rho\lf(x,\cy^k\r):=\inf_{y\in\cy^k}\rho(x,y).
\end{align}

Next, we recall the notion of approximations of
the identity with exponential decay which was
introduced in \cite[Definition 2.7]{hlyy19}.
\begin{definition}\label{expati}
Let $\delta$ be as in Lemma \ref{daydic}.
A sequence $\{Q_k\}_{k\in\zz}$ of bounded linear integral operators on
$L^2(\cx)$ is called an \emph{approximation of the identity with
exponential decay} (for short, exp-ATI) if there exist constants $C$,
$\nu\in(0, \infty)$, $a\in(0, 1]$, and $\eta\in(0,1)$ such that,
for any $k \in\zz$, the kernel of the operator $Q_k$, a
function on $\cx \times \cx$, which is still denoted by $Q_k$,
has the following properties:
\begin{enumerate}
\item[{\rm(i)}] (the \emph{identity condition})
$\sum_{k\in\zz} Q_k=I$ in $L^2(\cx)$,
where $I$ denotes the identity operator on $L^2(\cx)$;
\item[{\rm(ii)}] (the \emph{size condition}) for any $x,\ y\in\cx$,
\begin{align*}
|Q_k(x,y)|&\leq CE_k(x,y),
\end{align*}
here and thereafter,
$$E_k(x,y):=\frac{1}{\sqrt{V_{\delta^k}(x)V_{\delta^k}(y)}}
\exp\left\{-\nu\left[\frac{\rho(x,y)}{\delta^k}\right]^a\right\}
\exp\left\{-\nu\left[\frac{\max\{\rho(x,\cy^k),\rho(y,\cy^k)\}}
{\delta^k}\right]^a\right\};$$
\item[{\rm(iii)}] (the \emph{regularity condition}) for any
$x,\ x',\ y\in\cx$ with $\rho(x, x')\leq\delta^k$,
\begin{align*}
|Q_k(x,y)-Q_k(x',y)|+|Q_k(y,x)-Q_k(y,x')|
\leq C\left[\frac{\rho(x,x')}{\delta^k}\right]^\eta
E_k(x,y);
\end{align*}
\item[{\rm(iv)}] (the \emph{second difference regularity condition})
for any $x,\ x',\ y,\ y'\in\cx$ with $\rho(x, x')\leq\delta^k$
and $\rho(y, y')\leq\delta^k$,
\begin{align*}
|[Q_k(x,y)-Q_k(x',y)]-[Q_k(x,y')-Q_k(x',y')]|
\leq C\left[\frac{\rho(x,x')}{\delta^k}\right]^\eta
\left[\frac{\rho(y,y')}{\delta^k}\right]^\eta E_k(x,y);
\end{align*}
\item[{\rm(v)}] (the \emph{cancellation condition}) for any $x,\ y \in \cx$,
$$\int_\cx Q_k(x,y')\,d\mu(y')=0=\int_\cx Q_k(x',y)\,d\mu(x').$$
\end{enumerate}
\end{definition}

The following notion of 1-exp-ATIs is just \cite[Definition 2.8]{hhllyy}.
\begin{definition}\label{1expati}
Let $\eta\in(0,1)$ be as in Definition \ref{expati}.
A sequence $\{P_k\}_{k\in\zz}$ of bounded
linear integral operators on $L^2(\cx)$ is called an
\emph{approximation of the identity with exponential
decay and integration 1} (for short, 1-exp-ATI)
if $\{P_k\}_{k\in\zz}$ has the following properties:
\begin{enumerate}
\item[{\rm(i)}] for any $k\in\zz$, $P_k$ satisfies (ii), (iii),
and (iv) of Definition \ref{expati} but without the term
$E_k(x,y)$;
\item[{\rm(ii)}] for any $k\in\zz$ and $x\in\cx$,
$$\int_\cx P_k(x,y)\,d\mu(y)=1=\int_\cx P_k(y,x)\,d\mu(y);$$
\item[{\rm(iii)}] for any $k\in\zz$, letting $Q_k:=P_k-P_{k-1}$,
then $\{Q_k\}_{k\in\zz}$ is an exp-ATI.
\end{enumerate}
\end{definition}

The following lemma is just \cite[Lemma 3.3]{zhy}.
\begin{lemma}\label{1eptilem}
Let $\{P_k\}_{k\in\zz}$ be a {\rm{1-exp-ATI}} and $\eta$
as in Definition \ref{expati}.
Then the following statements hold true:
\begin{enumerate}
\item[{\rm(i)}] for any given $\bz$, $\gz\in(0,\eta]$,
there exists a positive constant $C$ such that,
for any $k\in \zz_+$ and $\vz\in\cg(\beta,\gamma)$,
$$\lf\|P_k\vz\r\|_{\cg(\bz,\gz)}\le C\|\vz\|_{\cg(\bz,\gz)};$$
\item[{\rm(ii)}] for any $\bz$, $\gz\in(0,\eta]$,
$\widetilde{\bz}\in(0,\bz)$, and $f\in\cg(\beta,\gamma)$,
$$\lim_{k\rightarrow\fz}P_kf=f$$
in $\cg(\widetilde{\bz},\gz)$;
\item[{\rm(iii)}] for any $\bz\in(0,\eta)$ and $\gz\in(0,\eta]$,
if $f\in\cg_0^{\eta}(\bz,\gz)$
[resp., $f\in(\cg_0^{\eta}(\bz,\gz))'$], then
$$\lim_{k\rightarrow\fz}P_kf=f$$
in $\cg_0^{\eta}(\bz,\gz)$ [resp., $(\cg_0^{\eta}(\bz,\gz))'$].
\end{enumerate}
\end{lemma}

The following properties for \rm{1-exp-ATIs} as
in Definition \ref{1expati} are just
\cite[Lemmas 4.16 and 4.17]{zhy}, respectively.

\begin{lemma}\label{zhoulem1}
Let $\bz$, $\gz\in(0,\eta)$ with $\eta$ as in Definition \ref{expati},
$\{P_k\}_{k\in\zz}$ be a {\rm{1-exp-ATI}} as in Definition \ref{1expati},
and $f\in \cg(x_0,r_0,\bz,\gz)$ for some $x_0\in\cx$ and $r_0\in(0,\fz)$.
Then there exists a positive constant $C$, independent of $x_0$, $r_0$,
and $f$, such that, for any $k\in\zz$ satisfying $r_0\ge\dz^k$,
$$\lf\|P_kf\r\|_{\cg(x_0,r_0,\bz,\gz)}
\le C\|f\|_{\cg(x_0,r_0,\bz,\gz)}.$$
\end{lemma}

\begin{lemma}\label{zhoulem2}
Let $\bz$, $\gz\in(0,\eta)$ with $\eta$ as in Definition \ref{expati},
$\{P_k\}_{k\in\zz}$ be a {\rm{1-exp-ATI}} as in Definition \ref{1expati},
and $f\in \cg(x_0,r_0,\bz,\gz)$ for some $x_0\in\cx$ and $r_0\in(0,\fz)$.
Then there exists a positive constant $C$, independent of $x_0$, $r_0$,
and $f$, such that, for any $k\in\zz$ satisfying $r_0<\dz^k$,
$$\lf\|P_kf\r\|_{\cg(x_0,\dz^k,\bz,\gz)}
\le C\|f\|_{\cg(x_0,r_0,\bz,\gz)}.$$
\end{lemma}

\begin{remark}\label{remhe1}
Compared with the corresponding ones on RD-spaces (see \cite{han08}),
these exp-ATIs on $\cx$ have some essential differences presented
via some terms such as
$$\exp\left\{-\nu\left[\frac{\max\{\rho(x,\cy^k),\rho(y,\cy^k)\}}
{\delta^k}\right]^a\right\}.$$
Observe that, here $x$, $y\in\cx$, $\cy^k$ is the set
of dyadic reference points as in \eqref{21.7.16.x1},
and $\rho(x,\cy^k)$ [resp. $\rho(y,\cy^k)$] the quasi-distance
between $x$ (resp. $y$) and $\cy^k$ as in \eqref{21.7.16.x2}.
Thus, such terms are closely connecting with the geometrical
property of $\cx$ and, therefore, help us to get rid of the
dependence on the reverse doubling assumption of the
equipped measure $\mu$ of $\cx$.
\end{remark}

Next, we introduce several Hardy spaces, associated with $Y(\cx)$,
via different maximal functions.

\begin{definition}\label{bh}
Let $\beta,\ \gamma \in (0, \eta)$ and $f\in(\icgg)'$
with $\eta$ as in Definition \ref{expati}.
Let $\{P_k\}_{k\in\zz}$ be a 1-exp-ATI as
in Definition \ref{1expati}.
The \emph{radial maximal function} $\cm^+(f)$ of $f$
is defined by setting, for any $x\in\cx$,
\begin{align*}
\cm^+(f)(x):=\sup_{k\in\zz}\lf|P_kf(x)\r|
\end{align*}
and, for any given $a\in(0,\fz)$,
the \emph{non-tangential maximal function}
$\cm_a(f)$, with aperture $a$, of $f$ is defined by setting,
for any $x\in\cx$,
$$\cm_a(f)(x):=\sup_{k\in\zz}\sup_{y\in B(x,\,a\delta^k)}\lf|P_kf(y)\r|,$$
where $\delta$ is as in Lemma \ref{daydic}. Moreover,
the \emph{grand maximal function} $f^*$ of $f$ is defined by setting,
for any $x\in\cx$,
$$f^*(x):=\sup\lf\{|\langle f,h\rangle|:\
h\in\icgg,\ \|h\|_{\cg(x,r,\beta,\gamma)}\le1\ \mbox{for some}\ r\in(0,\fz)\r\}.$$
\end{definition}

Let $Y(\cx)$ be a ball quasi-Banach function space on $\cx$.
Then the \emph{Hardy spaces} $H_Y^+(\cx)$, $H_Y^a(\cx)$
[with $a\in(0,\fz)$], and $H_Y^*(\cx)$, associated with $Y(\cx)$,
are defined, respectively, by setting
\begin{align}\label{3.5x}
H_Y^+(\cx):=\lf\{f\in\lf(\icgg\r)':\
\|f\|_{H_Y^+(\cx)}:=\lf\|\cm^+(f)\r\|_{Y(\cx)}<\fz\r\},
\end{align}
\begin{align}\label{3.5y}
H_Y^a(\cx):=\lf\{f\in\lf(\icgg\r)':\
\|f\|_{H_Y^a(\cx)}:=\lf\|\cm_a(f)\r\|_{Y(\cx)}<\fz\r\},
\end{align}
and
\begin{align}\label{3.5z}
H_Y^*(\cx):=\lf\{f\in\lf(\icgg\r)':\
\|f\|_{H_Y^*(\cx)}:=\lf\|f^*\r\|_{Y(\cx)}<\fz\r\}.
\end{align}

The following lemma shows that $H_Y^*(\cx)$ can be embedded
into $(\icgg)'$.

\begin{lemma}\label{7.9.x1}
Assume that $Y(\cx)$ is a ball quasi-Banach function space on $\cx$.
Let $\bz$, $\gz$, and $\eta$ be as in Definition \ref{bh}.
Then $H_Y^*(\cx)$ is continuously embedded into $(\icgg)'$, namely,
there exists a positive constant $C$ such that, for any $f\in H_Y^*(\cx)$
and $h\in\icgg$,
\begin{align}\label{1.9.x1}
\lf|\langle f,h\rangle\r|\le C\|f\|_{H_Y^*(\cx)}\|h\|_{\icgg}.
\end{align}
\end{lemma}
\begin{proof}
Let all the symbols be as in the present lemma.
Without loss of generality, we may assume
that $h\in\icgg$ with $\|h\|_{\icgg}\le1$.
Let $x_0\in\cx$ be as in Definition \ref{test}.
By Definition \ref{test}, we know that, for any $x\in B(x_0,1)$,
\begin{align*}
\|h\|_{\cg(x,1,\bz,\gz)}\sim\|h\|_{\cg(\bz,\gz)}\sim\|h\|_{\icgg}\ls1,
\end{align*}
where the implicit positive constants are independent of $x$ and $h$.
From this and Definition \ref{bh}, we deduce that,
for any $f\in H_Y^*(\cx)$ and $x\in B(x_0,1)$,
$|\langle f,h\rangle|\ls f^*(x)$,
which, combined with (ii) and (iv) of Definition \ref{qbs},
further implies that
\begin{align}\label{3.19.x1}
|\langle f,h\rangle|\ls\inf_{x\in B(x_0,1)}f^*(x)
\ls\lf\|f^*\r\|_{Y(\cx)}\lf\|\ch1_{B(x_0,1)}\r\|_{Y(\cx)}^{-1}
\sim\lf\|f\r\|_{H_Y^*(\cx)}.
\end{align}
This shows \eqref{1.9.x1} and hence
finishes the proof of Lemma \ref{7.9.x1}.
\end{proof}

\begin{remark}\label{maxeq}
\begin{enumerate}
\item[{\rm (i)}]
Let $a\in(0,\fz)$ and $\eta$ be as in Definition \ref{expati}.
By \cite[Remark 2.9(ii)]{gly08}, we know that
there exists a positive constant $C$, depending on $a$,
such that, for any $f\in(\icgg)'$ with $\bz,\ \gz\in(0, \eta)$,
and any $x\in\cx$,
$$\cm^+(f)(x)\le\cm_a(f)(x)\le Cf^*(x).$$

\item[{\rm (ii)}]
Let $Y(\cx)$ be a ball quasi-Banach function space on $\cx$.
By Lemmas \ref{fatou} and \ref{7.9.x1}, together with an
argument similar to that used in the proof of
\cite[Proposition 3.1]{hhllyy} (see also \cite[Proposition 3.9]{zhy}),
we conclude that $H_Y^*(\cx)$ is complete.
\end{enumerate}
\end{remark}

Next, we state the main theorem of this subsection as follows.

\begin{theorem}\label{relation}
Let $\bz,\ \gz\in(0,\eta)$ with $\eta$ as in Definition \ref{expati},
and $Y(\cx)$ be a ball quasi-Banach function space on $\cx$.
Assume that there exist an $r_1\in(1,\fz)$ and an $r_2\in[1,\fz)$
such that $Y^{1/{r_1}}(\cx)$ is a ball Banach function space and
that $\cm$ is bounded, respectively, on both $(Y^{1/{r_1}})'(\cx)$
and $Y^{1/{r_2}}(\cx)$.
Then the following statements hold true:
\begin{enumerate}
\item[{\rm (i)}]
$Y(\cx)\hookrightarrow(\icgg)'$;

\item[{\rm (ii)}]
if $f\in Y(\cx)$, then $f\in H_Y^*(\cx)$ and there exists a
positive constant $C$, independent of $f$, such that
$\|f\|_{H_Y^*(\cx)}\le C\|f\|_{Y(\cx)}$;

\item[{\rm (iii)}]
if $f\in H_Y^+(\cx)$, then there exists an $\widetilde{f}\in Y(\cx)$
such that $\widetilde{f}=f$ in $(\icgg)'$ and
$\|\widetilde{f}\|_{Y(\cx)}\le \|f\|_{H_Y^+(\cx)}$.
\end{enumerate}

Consequently, for any fixed $a\in(0,\fz)$,
$H_Y^+(\cx)=H_Y^a(\cx)=H_Y^*(\cx)=Y(\cx)$
in the sense of both representing the same
distributions and having the equivalent norms.
\end{theorem}

\begin{proof}
Let all the symbols be as in the present theorem.
First, we prove (i).
Let $f\in Y(\cx)$, $\vz\in\icgg$, and $r_2$ be as in the present theorem.
By \eqref{1.7.x1}, we know that $\cm(f)\le\cm^{(r_2)}(f)$,
which further implies that $[\cm(f)]^{r_2}\le\cm(|f|^{r_2})$.
From this, Definition \ref{cvex},
and the boundedness of $\cm$ on $Y^{1/{r_2}}(\cx)$,
we deduce that
\begin{align}\label{21.7.5.x1}
\lf\|\cm(f)\r\|_{Y(\cx)}
&=\lf\|\lf[\cm(f)\r]^{r_2}\r\|_{Y^{\frac1{r_2}}(\cx)}^{\frac1{r_2}}
\le\lf\|\cm(|f|^{r_2})\r\|_{Y^{\frac1{r_2}}(\cx)}^{\frac1{r_2}}
\ls\lf\||f|^{r_2}\r\|_{Y^{\frac1{r_2}}(\cx)}^{\frac1{r_2}}
\sim\|f\|_{Y(\cx)}.
\end{align}
By this, Definition \ref{test}(i), \eqref{eq-doub},
and Definition \ref{qbs}(iv), we conclude that
\begin{align*}
&\int_{\cx}|f(x)\vz(x)|\,d\mu(x)\\
&\quad\le\|\vz\|_{\icgg}
\int_{\cx}|f(x)|P_{\gz}(x_0,x;1)\,d\mu(x)\\
&\quad=\|\vz\|_{\icgg}\lf\{\sum_{k\in\nn}
\int_{B(x_0,2^{k+1})\backslash B(x_0,2^{k})}
+\int_{B(x_0,2)}\r\}|f(x)|P_{\gz}(x_0,x;1)\,d\mu(x)\\
&\quad\ls\|\vz\|_{\icgg}\sum_{k\in\zz_+}2^{-k\gamma}
\frac{1}{\mu(B(x_0,2^{k}))}
\int_{B(x_0,2^{k+1})}|f(x)|\,d\mu(x)\\
&\quad\ls\|\vz\|_{\icgg}\sum_{k\in\zz_+}2^{-k\gamma}
\inf_{y\in B(x_0,1)}\cm(f)(y)\\
&\quad\ls\|\vz\|_{\icgg}\lf\|\cm(f)\r\|_{Y(\cx)}
\lf\|\ch1_{B(x_0,1)}\r\|_{Y(\cx)}^{-1}
\ls\|\vz\|_{\icgg}\|f\|_{Y(\cx)},
\end{align*}
where $P_{\gz}(x_0,x;1)$ is as in \eqref{21.7.15.x1}
with $r$ replaced by $1$.
This further implies that $Y(\cx)\hookrightarrow(\icgg)'$
and hence (i) holds true.

Next, we prove (ii). Let $f\in Y(\cx)$.
By (i) and \cite[Proposition 3.9]{gly08}, we know that
$f^*\ls\cm(f)$. From this and \eqref{21.7.5.x1}, we deduce that
\begin{align*}
\|f\|_{H_Y^*(\cx)}\ls\lf\|\cm(f)\r\|_{Y(\cx)}\ls\|f\|_{Y(\cx)},
\end{align*}
which completes the proof of (ii).

Finally, we prove (iii). Let $f\in H_Y^+(\cx)$,
$\{P_k\}_{k\in\zz}$ be a 1-exp-ATI as in Definition \ref{1expati},
and $r_1$ as in the present theorem.
By Definition \ref{bh} and Theorem \ref{embed}, we know that
$$\lf\|\sup_{k\in\zz}\lf|P_kf\r|\r\|_{L^{r_1}_w(\cx)}
\ls\lf\|\sup_{k\in\zz}\lf|P_kf\r|\r\|_{Y(\cx)}
\sim\lf\|f\r\|_{H_Y^+(\cx)}<\fz,$$
where $w$ is as in Theorem \ref{embed}.
From this and the Banach--Alaoglu theorem (see, for instance,
\cite[Theorem 3.17]{rudin91}), we deduce that there exist an
$\widetilde{f}\in L^{r_1}_w(\cx)$ and a sequence
$\{k_j\}_{j\in\nn}\st\zz$ such that,
for any $g\in(L^{r_1}_w(\cx))^*$
[the dual space of $L^{r_1}_w(\cx)$],
\begin{align}\label{21.7.5.x2}
\int_{\cx}\widetilde{f}(x)g(x)\,d\mu(x)
&=\lim_{j\rightarrow\fz}\int_{\cx}P_{k_j}f(x)g(x)\,d\mu(x).
\end{align}
Using this and $\icgg\st(L^{r_1}_w(\cx))^*$,
we conclude that, for any $\vz\in\icgg$,
\begin{align*}
\lf|\int_{\cx}\widetilde{f}(x)\vz(x)\,d\mu(x)\r|
&=\lim_{j\rightarrow\fz}
\lf|\int_{\cx}P_{k_j}f(x)\vz(x)\,d\mu(x)\r|\\
&\le\int_{\cx}\sup_{k\in\zz}\lf|P_{k}f(x)\r||\vz(x)|\,d\mu(x),
\end{align*}
which, combined with the arbitrariness of $\vz\in\icgg$ and
the Lebesgue differentiation theorem
(see, for instance, \cite[Theorem 1.8]{he01}),
further implies that, for almost every $x\in\cx$,
$$\lf|\widetilde{f}(x)\r|\le\sup_{k\in\zz}\lf|P_{k}f(x)\r|.$$
Thus, $\|\widetilde{f}\|_{Y(\cx)}\le \|f\|_{H_Y^+(\cx)}$.
Moreover, letting $w_1:=w^{1-r_1'}$, then, by the H\"older inequality
and $\icgg\st L^{r_1'}_{w_1}(\cx)$, we know that,
for any $j\in\nn$ and $\vz\in\icgg$,
\begin{align*}
&\int_{\cx}\lf|P_{k_j}f(x)\vz(x)\r|\,d\mu(x)\\
&\qquad\le\lf[\int_{\cx}\lf|P_{k_j}f(x)\r|^{r_1}w(x)\,d\mu(x)\r]^{1/{r_1}}
\lf[\int_{\cx}\lf|\vz(x)\r|^{r_1'}
\lf[w(x)\r]^{-r_1'/{r_1}}\,d\mu(x)\r]^{1/{r_1'}}\\
&\qquad=\lf\|P_{k_j}f\r\|_{L^{r_1}_w(\cx)}
\|\vz\|_{L^{r_1'}_{w_1}(\cx)}<\fz.
\end{align*}
From this, Lemma \ref{1eptilem}(iii), $L^{r_1}_w(\cx)\st(\icgg)'$,
\cite[(3.98)]{han08}, and \eqref{21.7.5.x2},
we deduce that, for any $\vz\in\icgg$,
\begin{align*}
\langle f,\vz\rangle
&=\lim_{j\rightarrow\fz}
\lf\langle P_{k_j}f,\vz\r\rangle
=\lim_{j\rightarrow\fz}\int_{\cx}P_{k_j}f(x)\vz(x)\,d\mu(x)
=\int_{\cx}\widetilde{f}(x)\vz(x)\,d\mu(x).
\end{align*}
This finishes the proof of (iii).

By (ii), (iii), and Remark \ref{maxeq}(i), we obtain
$H_Y^+(\cx)=H_Y^a(\cx)=H_Y^*(\cx)=Y(\cx)$,
which completes the proof of Theorem \ref{relation}.
\end{proof}

\begin{remark}\label{relarem}
\begin{enumerate}
\item[{\rm (i)}]
Let $a\in(0,\fz)$, $p\in(1,\fz]$, and $Y(\cx):=L^{p}(\cx)$.
Then $H_Y^+(\cx)$, $H_Y^a(\cx)$, and $H_Y^*(\cx)$ are,
respectively, the classical \emph{Hardy spaces}
$H^{+,p}(\cx)$, $H^{p}_a(\cx)$, and $H^{*,p}(\cx)$,
which were introduced by He et al. in \cite{hhllyy}.
Choose $r_1$, $r_2\in(1,p)$. By this, Remark \ref{qbsdefrem}(i),
and \cite[(3.6)]{cw77}, we conclude that $L^{p}(\cx)$
satisfies all the assumptions of Theorem \ref{relation}.
In this case, Theorem \ref{relation} covers
\cite[Theorem 3.4]{hhllyy}.

\item[{\rm (ii)}]
We point out that, in the proof of Theorem \ref{relation}(iii),
we can not find a desired $\widetilde{f}$ directly since $Y(\cx)$
does not have an explicit norm expression.
To overcome the difficulty caused by this,
via Theorem \ref{embed}, we embed $Y(\cx)$
to the weighted Lebesgue space $L^{r_1}_w(\cx)$ with $w$
as in Theorem \ref{embed}, and then apply the completeness
of $L^{r_1}_w(\cx)$ together with the Banach--Alaoglu theorem
to complete the proof of \ref{relation}(iii).
\end{enumerate}
\end{remark}

\subsection{Maximal function characterizations\label{s-max2}}

In this subsection, we characterize $H_Y^*(\cx)$ by
the radial and the non-tangential maximal functions, respectively.
To this end, we first prove the following conclusion.

\begin{theorem}\label{maxch}
Let $Y(\cx)$ be a ball quasi-Banach function space on $\cx$,
$a\in(0,\fz)$, and $p_-\in({\omega}/(\omega+\eta),\fz)$ with
$\omega$ and $\eta$, respectively,
as in \eqref{eq-doub} and Definition \ref{expati}.
Assume that, for any given $\tz\in(0,\underline{p})$ with $\underline{p}$
as in \eqref{2.1y}, the Hardy--Littlewood maximal operator $\cm$
in \eqref{hlmax} is bounded on $Y^{1/\tz}(\cx)$.
Let $\bz$, $\gz\in(\omega(1/\underline{p}-1),\eta)$.
Then
$$\lf[H_Y^+(\cx)\cap(\icgg)'\r]=\lf[H_Y^a(\cx)\cap(\icgg)'\r]
=\lf[H_Y^*(\cx)\cap(\icgg)'\r]$$
with equivalent quasi-norms.
\end{theorem}

\begin{proof}
Let all the symbols be as in the present theorem.
Assume that $f\in(\icgg)'$.
By Remark \ref{maxeq}(i) and Definition \ref{qbs}(ii),
we know that
$$\lf\|\cm^+(f)\r\|_{Y(\cx)}\le\lf\|\cm_a(f)\r\|_{Y(\cx)}
\ls\lf\|f^*\r\|_{Y(\cx)}$$
and hence $H_Y^*(\cx)\st H_Y^a(\cx)\st H_Y^+(\cx)$.
Thus, to complete the proof of this theorem,
it suffices to show that $H_Y^+(\cx)\st H_Y^*(\cx)$ and
\begin{align}\label{4.2.x1}
\lf\|f^*\r\|_{Y(\cx)}\ls\lf\|\cm^+(f)\r\|_{Y(\cx)}.
\end{align}
Choose $\widetilde{\bz}\in(0,\min\{\bz,\gz\})$ such that $\frac{\oz}{\oz+\widetilde{\bz}}<\underline{p}$.
Let $\tz\in(\frac{\oz}{\oz+\widetilde{\bz}},\underline{p})$.
Then, similarly to the proof of \cite[(3.5)]{hhllyy},
we conclude that, for any $x\in\cx$,
$$f^*(x)\ls\cm^+(f)(x)
+\lf\{\cm\lf(\lf[\cm^+(f)\r]^{\tz}\r)(x)\r\}^{\frac1{\tz}},$$
which, combined with Definitions \ref{qbs}(ii) and \ref{cvex},
and the boundedness of $\cm$ on $Y^{1/\tz}(\cx)$,
further implies that
\begin{align*}
\lf\|f^*\r\|_{Y(\cx)}&\ls\lf\|\cm^+(f)\r\|_{Y(\cx)}
+\lf\|\lf\{\cm\lf(\lf[\cm^+(f)\r]^{\tz}\r)\r\}
^{\frac1{\tz}}\r\|_{Y(\cx)}\\
&\sim\lf\|\cm^+(f)\r\|_{Y(\cx)}
+\lf\|\cm\lf(\lf[\cm^+(f)\r]^{\tz}\r)\r\|_{Y^{1/{\tz}}(\cx)}^{\frac1{\tz}}\\
&\ls\lf\|\cm^+(f)\r\|_{Y(\cx)}.
\end{align*}
This finishes the proof of \eqref{4.2.x1}
and hence of Theorem \ref{maxch}.
\end{proof}

Next, we give an important property of $H_Y^*(\cx)$ as follows.

\begin{theorem}\label{indep}
Let $Y(\cx)$ be a ball quasi-Banach function space on $\cx$,
and $p_-\in({\omega}/(\omega+\eta),\fz)$ with $\omega$ and $\eta$,
respectively, as in \eqref{eq-doub} and Definition \ref{expati}.
Assume that, for any given $\tz\in(0,\underline{p})$ with $\underline{p}$
as in \eqref{2.1y}, the Hardy--Littlewood maximal operator $\cm$
in \eqref{hlmax} is bounded on $Y^{1/\tz}(\cx)$.
Let $\bz_1$, $\bz_2$, $\gz_1$, $\gz_2\in(\omega(1/\underline{p}-1),\eta)$
and $H_Y^*(\cx)$ be as in Definition \ref{bh} with $\bz$ and $\gz$
replaced, respectively, by $\bz_1$ and $\gz_1$.
If $f\in H_Y^*(\cx)$, then $f\in (\mathcal{G}_0^\eta(\bz_2,\gz_2))'$.
\end{theorem}

\begin{proof}
Let $\bz_1$ and $\gz_1$ be as in the present theorem,
$f\in H_Y^*(\cx)$,
and $x_0$ be as in Definition \ref{test}.
Then, by an argument similar to that used in the proof of
\cite[(3.9)]{hhllyy}, we know that there exist an
$a\in(0,\fz)$ and a $\tz\in({\omega}/(\omega+\eta),\underline{p})$
such that, for any $\vz\in\mathcal{G}(\eta,\eta)$
with $\|\vz\|_{\mathcal{G}(\bz_2,\gz_2)}\le 1$, and any $x\in B(x_0,1)$,
$$\lf|\lf\langle f,\vz\r\rangle\r|
\ls\lf\{\cm\lf(\lf[\cm_a(f)\r]^{\tz}\r)(x)\r\}^{\frac1{\tz}},$$
which, combined with (ii) and (iv) of Definition \ref{qbs},
Definition \ref{cvex}, the boundedness of $\cm$ on $Y^{1/\tz}(\cx)$,
and Theorem \ref{maxch}, further implies that
\begin{align*}
\lf|\lf\langle f,\vz\r\rangle\r|&\ls\inf_{x\in B(x_0,1)}
\lf\{\cm\lf(\lf[\cm_a(f)\r]^{\tz}\r)(x)\r\}^{\frac1{\tz}}
\ls\lf\|\lf\{\cm\lf(\lf[\cm_a(f)\r]^{\tz}\r)\r\}^{\frac1{\tz}}\r\|_{Y(\cx)}
\lf\|\ch1_{B(x_0,1)}\r\|_{Y(\cx)}^{-1}\\
&\ls\lf\|\cm\lf(\lf[\cm_a(f)\r]^{\tz}\r)\r\|
_{Y^{1/{\tz}}(\cx)}^{\frac1{\tz}}
\ls\lf\|\cm_a(f)\r\|_{Y(\cx)}
\sim\|f\|_{H_Y^{*}(\cx)}.
\end{align*}
From this and the homogeneity of
$\|\cdot\|_{\mathcal{G}(\bz_2,\gz_2)}$,
we deduce that, for any $\vz\in\mathcal{G}(\eta,\eta)$,
\begin{align}\label{21.7.17.x1}
\lf|\lf\langle f,\vz\r\rangle\r|&
\ls\|f\|_{H_Y^{*}(\cx)}\|\vz\|_{\mathcal{G}(\bz_2,\gz_2)}.
\end{align}
Now, let $g\in\mathcal{G}_0^\eta(\bz_2,\gz_2)$.
By the definition of $\mathcal{G}_0^\eta(\bz_2,\gz_2)$,
we know that there exist $\{g_j\}_{j\in\nn}\st\mathcal{G}(\eta,\eta)$
such that
\begin{align}\label{21.7.18.x1}
\lf\|g-g_j\r\|_{\mathcal{G}(\bz_2,\gz_2)}\rightarrow0
\end{align}
as $j\rightarrow\fz$. Moreover, from \eqref{21.7.17.x1},
we deduce that, for any $j$, $k\in\nn$,
$$\lf|\lf\langle f,g_j\r\rangle-\lf\langle f,g_k\r\rangle\r|
=\lf|\lf\langle f,g_j-g_k\r\rangle\r|\ls\|f\|_{H_Y^{*}(\cx)}
\lf\|g_j-g_k\r\|_{\mathcal{G}(\bz_2,\gz_2)}.$$
This, combined with \eqref{21.7.18.x1}, implies that
$\lim_{j\rightarrow\fz}\langle f,g_j\rangle$ exists and
the limit is independent of the choice of $\{g_j\}_{j\in\nn}$.
Define
$\langle f,g\rangle:=\lim_{j\rightarrow\fz}\langle f,g_j\rangle$.
By this, \eqref{21.7.17.x1}, and \eqref{21.7.18.x1}, we conclude that
\begin{align*}
\lf|\lf\langle f,g\r\rangle\r|
&=\lim_{j\rightarrow\fz}\lf|\lf\langle f,g_j\r\rangle\r|
\ls\|f\|_{H_Y^{*}(\cx)}
\lim_{j\rightarrow\fz}\lf\|g_j\r\|_{\mathcal{G}(\bz_2,\gz_2)}
\sim\|f\|_{H_Y^{*}(\cx)}\|g\|_{\mathcal{G}_0^\eta(\bz_2,\gz_2)},
\end{align*}
which further implies that $f\in(\mathcal{G}_0^\eta(\bz_2,\gz_2))'$.
This finishes the proof of Theorem \ref{indep}.
\end{proof}

Combining Theorems \ref{maxch} and \ref{indep},
we obtain the following main theorem of this subsection,
which shows that the Hardy spaces in Definition \ref{bh}
are all independent of the choices of $(\icgg)'$
and coincide with equivalent quasi-norms.

\begin{theorem}\label{maxchprop}
Let $Y(\cx)$ be a ball quasi-Banach function space on $\cx$, $a\in(0,\fz)$,
and $p_-\in({\omega}/(\omega+\eta),\fz)$ with $\omega$ and $\eta$, respectively,
as in \eqref{eq-doub} and Definition \ref{expati}.
Assume that, for any given $\tz\in(0,\underline{p})$ with
$\underline{p}$ as in \eqref{2.1y}, the Hardy--Littlewood
maximal operator $\cm$ in \eqref{hlmax} is bounded on $Y^{1/\tz}(\cx)$.
Then $H_Y^+(\cx)$, $H_Y^a(\cx)$, and $H_Y^*(\cx)$ are all independent
of the choices of $(\icgg)'$ whenever $\bz$, $\gz\in(\omega(1/\underline{p}-1),\eta)$,
and coincide with equivalent quasi-norms.
\end{theorem}

\begin{proof}
Let $\bz_1$, $\bz_2$, $\gz_1$,
$\gz_2\in(\omega(1/\underline{p}-1),\eta)$.
For any $i\in\{1,2\}$, define $H_Y^{+,i}(\cx)$, $H_Y^{a,i}(\cx)$,
and $H_Y^{*,i}(\cx)$, respectively, as in \eqref{3.5x},
\eqref{3.5y}, and \eqref{3.5z} with $(\icgg)'$ replaced by
$(\mathcal{G}_0^\eta(\bz_i,\gz_i))'$.
To prove this theorem, it suffices to show that
\begin{align}\label{21.7.18.x2}
H_Y^{+,1}(\cx)&=H_Y^{+,2}(\cx).
\end{align}
Indeed, using \eqref{21.7.18.x2} and Theorem \ref{maxch} twice,
we obtain
$$H_Y^{a,1}(\cx)=H_Y^{+,1}(\cx)=H_Y^{+,2}(\cx)=H_Y^{a,2}(\cx)$$
and
$$H_Y^{*,1}(\cx)=H_Y^{+,1}(\cx)=H_Y^{+,2}(\cx)=H_Y^{*,2}(\cx).$$
Next, we prove \eqref{21.7.18.x2}.
By symmetry, we only need to show
$H_Y^{+,1}(\cx)\st H_Y^{+,2}(\cx)$. To this end,
let $f\in H_Y^{+,1}(\cx)\st(\mathcal{G}_0^\eta(\bz_1,\gz_1))'$.
Then, by Theorems \ref{maxch} and \ref{indep}, we know that
$f\in H_Y^{*,1}(\cx)\st(\mathcal{G}_0^\eta(\bz_2,\gz_2))'$,
which, combined with $\cm^+(f)\in Y(\cx)$,
further implies that $f\in H_Y^{+,2}(\cx)$.
Thus, $H_Y^{+,1}(\cx)\st H_Y^{+,2}(\cx)$.
This finishes the proof of \eqref{21.7.18.x2}
and hence of Theorem \ref{maxchprop}.
\end{proof}

\begin{remark}\label{maxcharem}
Next, we apply Theorem \ref{maxchprop} to several concrete
examples of ball quasi-Banach function spaces on $\cx$,
namely, classical Lebesgue spaces, Lorentz spaces,
Musielak--Orlicz spaces, and variable Lebesgue spaces.
\begin{enumerate}
\item[{\rm (i)}]
Assume that $a\in(0,\fz)$, $p\in({\omega}/(\omega+\eta),1]$,
and $\bz$, $\gz\in(\omega(1/p-1),\eta)$ with $\omega$ as in \eqref{eq-doub}
and $\eta$ as in Definition \ref{expati}.
Let $Y(\cx):=L^{p}(\cx)$. Then $H_Y^+(\cx)$,
$H_Y^a(\cx)$, and $H_Y^*(\cx)$ are, respectively,
the classical \emph{Hardy spaces} $H^{+,p}(\cx)$,
$H^{p}_a(\cx)$, and $H^{*,p}(\cx)$,
which were introduced by He et al. in \cite{hhllyy}.
Choose
$p_-\in(\max\{{\omega}/(\omega+\bz),{\omega}/(\omega+\gz)\},p]$.
By this, Remark \ref{qbsdefrem}(i),
the fact that $Y^{1/\tz}(\cx)=L^{p/\tz}(\cx)$ for any $\tz\in(0,\fz)$,
and \cite[(3.6)]{cw77}, we conclude that $L^{p}(\cx)$
satisfies all the assumptions of Theorem \ref{maxchprop}.
In this case, Theorem \ref{maxchprop} improves the corresponding
results in \cite[Theorem 3.5]{hhllyy} by removing the assumption
that Hardy spaces are restricted to the same space of distributions.

\item[{\rm (ii)}]
Assume that $a\in(0,\fz)$, $r\in(0,\fz)$,
$p\in({\omega}/(\omega+\eta),\fz)$,
and $\bz$, $\gz\in(\omega[1/p-1]_+,\eta)$ with $\omega$ as in
\eqref{eq-doub}, $\eta$ as in Definition \ref{expati},
and $[1/p-1]_+:=\max\{1/p-1,0\}$.
Let $Y(\cx):=L^{p,r}(\cx)$.
Then $H_Y^+(\cx)$, $H_Y^a(\cx)$, and $H_Y^*(\cx)$
are, respectively, the \emph{Hardy--Lorentz spaces}
$H^{p,r}_+(\cx)$, $H^{p,r}_a(\cx)$, and $H^{p,r}_{\star}(\cx)$,
which were introduced by Zhou et al. in \cite{zhy}.
Choose
\begin{align*}
p_-\in
\begin{cases}
(\max\{{\omega}/(\omega+\bz),{\omega}/(\omega+\gz)\},p]
\ \ &\text{if}\ p\in({\omega}/(\omega+\eta),1),\\
[1,p]
\ \ &\text{if}\ p\in[1,\fz).
\end{cases}
\end{align*}
By this, Remark \ref{qbsdefrem}(ii),
the fact that
$Y^{1/\tz}(\cx)=L^{p/\tz,r/\tz}(\cx)$ for any $\tz\in(0,\fz)$,
and \cite[Lemma 3.5]{zhy}, we conclude that $L^{p,r}(\cx)$
satisfies all the assumptions of Theorem \ref{maxchprop}.
In this case, Theorem \ref{maxchprop} improves the corresponding
results in \cite[Theorem 3.6]{zhy} by removing the assumption
that Hardy spaces are restricted to the same space of distributions.

\item[{\rm (iii)}]
Assume that $a\in(0,\fz)$, $\vz$ is a growth function
as in Remark \ref{qbsdefrem}(iii) satisfying $p/q(\vz)\in(\oz/(\oz+\eta),1]$,
and $\bz$, $\gz\in(\omega[q(\vz)/p-1],\eta)$
with $\omega$ as in \eqref{eq-doub} and
$\eta$ as in Definition \ref{expati}.
Let $Y(\cx):=L^{\vz}(\cx)$. Then $H_Y^+(\cx)$,
$H_Y^a(\cx)$, and $H_Y^*(\cx)$ are, respectively,
the \emph{Musielak--Orlicz Hardy spaces} $H^{+,\vz}(\cx)$, $H^{\vz}_a(\cx)$, and $H^{*,\vz}(\cx)$, which were
introduced by Fu et al. in \cite{fmy19}.
Choose $p_-\in(\max\{{\omega}/(\omega+\bz),{\omega}/(\omega+\gz)\},p/q(\vz)]$.
By this, Remark \ref{qbsdefrem}(iii),
the fact that $Y^{1/\tz}(\cx)=L^{\widetilde{\vz}}(\cx)$
with $\widetilde{\vz}(\cdot,t)=\vz(\cdot,t^{1/\tz})$
for any $t\in(0,\fz)$ and $\tz\in(0,\fz)$,
and \cite[Theorem 4.11]{fmy19}, we conclude that $L^{\vz}(\cx)$
satisfies all the assumptions of Theorem \ref{maxchprop}.
In this setting, Theorem \ref{maxchprop} improves the
corresponding results in \cite[Theorem 4.12]{fmy19}
by removing the assumption that Hardy spaces are restricted
to the same space of distributions.

\item[{\rm (iv)}]
Recall that a $\mu$-measurable function $p(\cdot):\ \cx\to(0,\infty)$
is said to be \emph{globally log-H\"older continuous},
denoted by $p(\cdot)\in C^{\log}(\cx)$,
if there exist positive constants $C_{\log}(p)$
and $C_\fz$, $x_p\in\cx$, and $p_\fz\in\rr$ such that,
for any $x,\ y\in\cx$,
\begin{equation*}
|p(x)-p(y)|\le \frac{C_{\log}(p)}{\log(e+1/\rho(x,y))}
\end{equation*}
and
\begin{equation*}
|p(x)-p_\fz|\le \frac{C_\fz}{\log(e+\rho(x,x_p))}.
\end{equation*}
Assume that $a\in(0,\fz)$, $p(\cdot)\in C^{\log}(\cx)$
satisfies $\widetilde{p_-}\in({\omega}/(\omega+\eta),\fz)$,
and $\bz$, $\gz\in(\omega[1/\widetilde{p_-}-1]_+,\eta)$
with $\omega$ as in \eqref{eq-doub}, $\eta$ as in Definition \ref{expati}, and $[1/\widetilde{p_-}-1]_+:=\max\{1/\widetilde{p_-}-1,0\}$.
Let $Y(\cx):=L^{p(\cdot)}(\cx)$.
Then $H_Y^*(\cx)$ is the \emph{variable Hardy space}
$H^{*,p(\cdot)}(\cx)$ which was introduced
by Zhuo et al. in \cite{zsy}.
Choose
\begin{align*}
p_-\in
\begin{cases}
(\max\{{\omega}/(\omega+\bz),{\omega}/(\omega+\gz)\},\widetilde{p_-}]
\ \ &\text{if}\ \widetilde{p_-}\in({\omega}/(\omega+\eta),1),\\
[1,\widetilde{p_-}]
\ \ &\text{if}\ \widetilde{p_-}\in[1,\fz).
\end{cases}
\end{align*}
By this, Remark \ref{qbsdefrem}(iv),
the fact that $Y^{1/\tz}(\cx)=L^{p(\cdot)/\tz}(\cx)$
for any $\tz\in(0,\fz)$, and \cite[Lemma 2.5]{zsy},
we conclude that $L^{p(\cdot)}(\cx)$ satisfies
all the assumptions of Theorem \ref{maxchprop}.
In this case, a corresponding conclusion of
Theorem \ref{maxchprop} was obtained in
\cite[Theorem 3.11]{zsy}. However, we point out that,
in \cite[Theorem 3.11]{zsy}, $(\cx,\rho,\mu)$ is assumed to be an RD-space.
Thus, Theorem \ref{maxchprop} improves
the corresponding results in \cite{zsy} by removing
the assumptions that $\mu$ satisfies the reverse doubling
condition and that Hardy spaces are restricted to
the same space of distributions.
\end{enumerate}
\end{remark}

\section{Atomic and finite atomic characterizations of
$H_Y^*(\cx)$\label{s-atom}}

In this section, we characterize
$H_Y^*(\cx)$ by atoms and finite atoms. Indeed,
using the boundedness of the powered Hardy--Littlewood
maximal operator on the associate space of $Y(\cx)$,
a key lemma (see Lemma \ref{claatlem} below),
the continuous embedding of $Y(\cx)$ to $L^p_{w}(\cx)$,
and the Fefferman--Stein vector-valued maximal inequality on $Y(\cx)$,
we first establish the atomic characterization of
$H_Y^*(\cx)$. By this, we further obtain the
finite atomic characterization of $H_{Y}^*(\cx)$.

\subsection{Atomic characterizations\label{satom1}}

In this subsection, we establish the atomic
characterization of $H_Y^*(\cx)$. To this end,
we first introduce the notion of atoms associated with
$Y(\cx)$.
\begin{definition}\label{atom}
Let $Y(\cx)$ be a ball quasi-Banach function space
on $\cx$, and $q\in[1,\fz]$.
A $\mu$-measurable function $a$ on $\cx$ is called a
\emph{$(Y(\cx),q)$-atom} if there exists a ball $B\subset\cx$ such that
\begin{enumerate}
\item[{\rm (i)}] $\supp a:=\lf\{x\in\cx:\ a(x)\neq0\r\} \st B$;
\item[{\rm (ii)}] $\|a\|_{L^q(\cx)}
\le\lf[\mu(B)\r]^{1/q}\lf\|\ch1_{B}\r\|^{-1}_{Y(\cx)}$;
\item[{\rm (iii)}] $\int_{\cx}a(x)\,d\mu(x)=0$.
\end{enumerate}
\end{definition}

We first give a reconstruction result of $H_Y^*(\cx)$
as follows (see \cite[Theorem 3.6]{shyy17} for the
corresponding Euclidean case).
\begin{proposition}\label{atre}
Let $Y(\cx)$ be a ball quasi-Banach function space on $\cx$ satisfying Assumption \ref{assump1} with $p_-\in({\omega}/(\omega+\eta),\fz)$,
where $\omega$ is as in \eqref{eq-doub} and $\eta$ as in
Definition \ref{expati}. Further assume that $Y(\cx)$
satisfies Assumption \ref{assump2} with the same $p_-$
as in Assumption \ref{assump1},
$\tz_0\in({\omega}/(\omega+\eta),\underline{p})$,
and $p_0\in(\tz_0,\fz)$, where $\underline{p}$ is as in \eqref{2.1y}.
Let $q\in(\max\{p_0,1\},\fz]$ and $d\in(0,\tz_0]$.
Suppose that $\{a_j\}_{j\in\nn}$ is a sequence of
$(Y(\cx),q)$-atoms supported, respectively, in balls
$\{B_j\}_{j\in\nn}$ of $\cx$, $\{\lz_j\}_{j\in\nn}\subset[0,\fz)$,
and
$$\lf\|\lf\{\sum_{j\in\nn}\lf[\frac{\lz_j}{\|\ch1_{B_j}\|_{Y(\cx)}}\r]
^{d}\ch1_{B_j}\r\}^{\frac{1}{d}}\r\|_{Y(\cx)}<\fz.$$
Then $f:=\sum_{j\in\nn}\lz_j a_j$ converges in $(\icgg)'$ for any
given $\bz$, $\gz\in(\omega(1/\tz_0-1),\eta)$. Moreover,
$f\in H_Y^*(\cx)$ and there exists a positive constant $C$,
independent of $\{a_j\}_{j\in\nn}$, $\{B_j\}_{j\in\nn}$,
and $\{\lz_j\}_{j\in\nn}$, such that
\begin{align}\label{8.1.x1}
\|f\|_{H_Y^*(\cx)}\le C\lf\|\lf\{\sum_{j\in\nn}
\lf[\frac{\lz_j}{\|\ch1_{B_j}\|_{Y(\cx)}}\r]^{d}
\ch1_{B_j}\r\}^{\frac{1}{d}}\r\|_{Y(\cx)}.
\end{align}
\end{proposition}

The following technical lemma,
which is a generalization of \cite[Lemma 4.8]{zwyy}
on $\rn$ to $\cx$, plays a key role in the proof of
Proposition \ref{atre}. It can be proved by an argument
similar to that used in the proofs of \cite[Lemma 4.8]{zwyy}
and Lemmas \ref{holder} and \ref{second};
we omit the details here.

\begin{lemma}\label{claatlem}
Let $r\in(0,\fz)$, $s\in(r,\fz]$, and $Y(\cx)$ be a ball
quasi-Banach function space on $\cx$. Assume that $Y^{1/{r}}(\cx)$
is a ball Banach function space and that there exists a
positive constant $C_0$ such that, for any $f\in (Y^{1/{r}})'(\cx)$ ,
\begin{align*}
\lf\|\cm^{((s/r)')}(f)\r\|_{(Y^{1/{r}})'(\cx)}
\le C_0\|f\|_{(Y^{1/{r}})'(\cx)}.
\end{align*}
Then there exists a positive constant $C$ such that,
for any sequence $\{B_j\}_{j\in\nn}$ of balls,
$\{\lz_j\}_{j\in\nn}\st\mathbb{C}$, and $\mu$-measurable functions
$\{a_j\}_{j\in\nn}$ satisfying that, for any $j\in\nn$,
$\supp a_j\subset B_j$ and $\|a_j\|_{L^s(\cx)}\le[\mu(B_j)]^{1/s}$,
it holds true that
\begin{align*}
\lf\|\lf(\sum_{j\in\nn}\lf|\lz_ja_j\r|^{r}\r)^{\frac{1}{r}}\r\|_{Y(\cx)}
\le C\lf\|\lf(\sum_{j\in\nn}\lf|\lz_j\ch1_{B_j}\r|^{r}\r)
^{\frac{1}{r}}\r\|_{Y(\cx)}.
\end{align*}
\end{lemma}

We also need the following basic inequality.
\begin{lemma}\label{jessen}
Let $\tz\in(0,1]$. Then, for any $\{a_j\}_{j\in\nn}\st[0,\fz)$,
$$\lf(\sum_{j\in\nn}a_j\r)^{\tz}\le\sum_{j\in\nn}a_j^{\tz}.$$
\end{lemma}

Now, we show Proposition \ref{atre}.
\begin{proof}[Proof of Proposition \ref{atre}]
Let all the symbols be as in the present proposition,
$a$ a $(Y(\cx),q)$-atom supported in a ball
$B:=B(x_B, r_B)$ with $x_B\in\cx$ and $r_B\in(0,\infty)$,
and $\bz$, $\gz\in(\omega(1/\tz_0-1),\eta)$.
Let $\widetilde{B}:=2A_0B$.
From the proof of \cite[(4.1)]{hhllyy}, we deduce that,
for any $x\in\widetilde{B}$,
\begin{align}\label{5.9.x3}
a^*(x)&\ls\cm(a)(x).
\end{align}
Moreover, by the cancellation of $a$, \eqref{3.31.x2},
Lemma \ref{equlem}, the H\"older inequality, \eqref{eq-doub},
and the definition of $\cm$, we know that,
for any $x\in{\widetilde{B}}^{\com}$ and $h\in\icgg$
satisfying that $\|h\|_{\cg(x,r,\beta,\gamma)}\le1$
for some $r\in(0,\fz)$,
\begin{align*}
|\la a,h\ra|&=\lf|\int_{B} a(y)\lf[h(y)-h(x_B)\r]\,d\mu(y)\r|
\le\int_{B} |a(y)|\lf|h(y)-h(x_B)\r|\,d\mu(y)\\
&\ls\int_{B} |a(y)|
\lf[\frac{\rho(x_B,y)}{r+\rho(x,x_B)}\r]^{\bz}\frac{1}
{V_{r}(x)+V(x,x_B)}
\lf[\frac{r}{r+\rho(x,x_B)}\r]^{\gz}\,d\mu(y)\noz\\
&\ls\int_{B}|a(y)|\,d\mu(y)
\lf[\frac{r_B}{\rho(x_B,x)}\r]^{\bz}\frac{1}{V(x_B,x)}
\ls\frac{\mu(B)}{\|\ch1_{B}\|_{Y(\cx)}}
\lf[\frac{r_B}{\rho(x_B,x)}\r]^{\bz}
\frac{1}{V(x_B,x)}\noz\\
&\ls\frac{1}{\|\ch1_{B}\|_{Y(\cx)}}
\lf[\frac{\mu(B)}{V(x_B,x)}\r]^{\frac{\bz+\oz}{\oz}}
\ls\frac{1}{\|\ch1_{B}\|_{Y(\cx)}}\lf[\cm(\ch1_{B})(x)\r]
^{\frac{\bz+\oz}{\oz}},\noz
\end{align*}
which further implies that, for any $x\in{\widetilde{B}}^{\com}$,
\begin{align}\label{5.9.x2}
a^*(x)&\ls\frac{1}{\|\ch1_{B}\|_{Y(\cx)}}\lf[\cm(\ch1_{B})(x)\r]
^{\frac{\bz+\oz}{\oz}}.
\end{align}

Let $\{a_j\}_{j\in\nn}$ be a sequence of $(Y(\cx),q)$-atoms supported,
respectively, in balls $\{B_j \}_{j\in\nn}$ of $\cx$,
and $\{\lz_j\}_{j\in\nn}\subset[0,\fz)$.
We first prove that $\sum_{j\in\nn}\lz_j a_j$ converges in $(\icgg)'$.
Without loss of generality, we may assume that
$\vz\in\icgg$ with $\|\vz\|_{\icgg}\le1$.
Let $x_0$ be as in Definition \ref{test}.
Then, by an argument similar to that used in the
proof of \eqref{3.19.x1}, (ii) and (iv) of Definition \ref{qbs},
\eqref{5.9.x3}, and \eqref{5.9.x2}, we know that
\begin{align}\label{5.9.x4}
\sum_{j\in\nn}\lz_j\lf|\lf\langle a_j,\vz\r\rangle\r|
&\ls\sum_{j\in\nn}\lz_j\inf_{x\in B(x_0,1)}a_j^*(x)
\ls\inf_{x\in B(x_0,1)}\sum_{j\in\nn}\lz_ja_j^*(x)\noz\\
&\ls\lf\|\sum_{j\in\nn}\lz_j a_j^*\r\|_{Y(\cx)}
\lf\|\ch1_{B(x_0,1)}\r\|_{Y(\cx)}^{-1}
\ls\lf\|\sum_{j\in\nn}\lz_j a_j^*
\ch1_{\widetilde{B_j}}\r\|_{Y(\cx)}
+\lf\|\sum_{j\in\nn}\lz_j a_j^*
\ch1_{{\widetilde{B_j}}^{\com}}\r\|_{Y(\cx)}\noz\\
&\ls{\rm I}_1+{\rm I}_2,
\end{align}
where, for any $j\in\nn$, $\widetilde{B_j}:=2A_0B_j$,
the implicit positive constants depend on $x_0$,
$${\rm I}_1:=\lf\|\sum_{j\in\nn}\lz_j\cm(a_j)
\ch1_{\widetilde{B_j}}\r\|_{Y(\cx)},$$
and
$${\rm I}_2:=\lf\|\sum_{j\in\nn}\frac{\lz_j }{\|\ch1_{B_j}\|_{Y(\cx)}}
\lf[\cm(\ch1_{B_j})\r]^{\frac{\bz+\oz}{\oz}}\r\|_{Y(\cx)}.$$

For ${\rm I}_1$, by $q\in(\max\{p_0,1\},\fz]$,
the H\"older inequality, and the boundedness of $\cm$
on $L^{q}(\cx)$ (see, for instance, \cite[(3.6)]{cw77}), we know that,
for any $j\in\nn$,
$$\supp\lf(\cm(a_j)\ch1_{\widetilde{B_j}}
\lf\|\ch1_{B_j}\r\|_{Y(\cx)}\r)\subset\widetilde{B_j}$$
and
$$\lf\|\cm(a_j)\ch1_{\widetilde{B_j}}
\lf\|\ch1_{B_j}\r\|_{Y(\cx)}\r\|_{L^{p_0}(\cx)}
\ls\lf\|\ch1_{B_j}\r\|_{Y(\cx)}\lf\|\cm(a_j)\r\|_{L^{q}(\cx)}
\lf[\mu\lf(\widetilde{B_j}\r)\r]^{\frac{1}{p_0}-\frac{1}{q}}
\ls\lf[\mu\lf(\widetilde{B_j}\r)\r]^{\frac{1}{p_0}}.$$
From this, Lemma \ref{jessen}, Definition \ref{qbs}(ii),
$p_0>\tz_0$, Assumption \ref{assump2}, Lemma \ref{claatlem},
Remark \ref{rek3.19}, and $d\in(0,\tz_0]$,
we deduce that
\begin{align}\label{5.9.x5}
{\rm I}_1&\le\lf\|\lf\{\sum_{j\in\nn}
\lf[\lz_j\cm(a_j)\ch1_{\widetilde{B_j}}\r]^{\tz_0}\r\}
^{\frac1{\tz_0}}\r\|_{Y(\cx)}
\ls\lf\|\lf\{\sum_{j\in\nn}
\lf[\frac{\lz_j}{\|\ch1_{B_j}\|_{Y(\cx)}}\r]
^{\tz_0}\ch1_{\widetilde{B_j}}\r\}^{\frac1{\tz_0}}\r\|_{Y(\cx)}\noz\\
&\ls\lf\|\lf\{\sum_{j\in\nn}
\lf[\frac{\lz_j}{\|\ch1_{B_j}\|_{Y(\cx)}}\r]
^{\tz_0}\ch1_{B_j}\r\}^{\frac1{\tz_0}}\r\|_{Y(\cx)}
\ls\lf\|\lf\{\sum_{j\in\nn}
\lf[\frac{\lz_j}{\|\ch1_{B_j}\|_{Y(\cx)}}\r]
^{d}\ch1_{B_j}\r\}^{\frac1{d}}\r\|_{Y(\cx)}<\fz.
\end{align}

For ${\rm I}_2$, by Definition \ref{cvex},
$\bz>\oz(1/{\tz_0}-1)$, $\tz_0<\underline{p}$,
Assumption \ref{assump1}, Lemma \ref{jessen},
and $d\in(0,\tz_0]\st(0,1)$, we conclude that
\begin{align}\label{6.15.x3}
{\rm I}_2&=\lf\|\lf\{\sum_{j\in\nn}
\frac{\lz_j }{\|\ch1_{B_j}\|_{Y(\cx)}}
\lf[\cm(\ch1_{B_j})\r]^{\frac{\bz+\oz}{\oz}}\r\}
^{\frac{\oz}{\bz+\oz}}\r\|
_{Y^{\frac{\bz+\oz}{\oz}}(\cx)}^{\frac{\bz+\oz}{\oz}}\noz\\
&\ls\lf\|\sum_{j\in\nn}\frac{\lz_j }{\|\ch1_{B_j}\|_{Y(\cx)}}
\ch1_{B_j}\r\|_{Y(\cx)}
\ls\lf\|\lf\{\sum_{j\in\nn}\lf[\frac{\lz_j}{\|\ch1_{B_j}\|_{Y(\cx)}}\r]
^{d}\ch1_{B_j}\r\}^{\frac1{d}}\r\|_{Y(\cx)}<\fz,
\end{align}
which, combined with \eqref{5.9.x4} and \eqref{5.9.x5},
further implies that $\sum_{j\in\nn}\lz_j a_j$ converges in $(\icgg)'$.

Next, we prove \eqref{8.1.x1}. Define $f:=\sum_{j\in\nn}\lz_j a_j$
in $(\icgg)'$. Then, from \eqref{5.9.x4}, \eqref{5.9.x5},
and \eqref{6.15.x3}, it follows that
\begin{align*}
\|f\|_{H_Y^*(\cx)}&=\lf\|f^*\r\|_{Y(\cx)}
\ls\lf\|\sum_{j\in\nn}\lz_j a_j^*
\ch1_{\widetilde{B_j}}\r\|_{Y(\cx)}
+\lf\|\sum_{j\in\nn}\lz_j a_j^*
\ch1_{{\widetilde{B_j}}^{\com}}\r\|_{Y(\cx)}\\
&\ls\lf\|\lf\{\sum_{j\in\nn}\lf[\frac{\lz_j}{\|\ch1_{B_j}\|_{Y(\cx)}}\r]
^{d}\ch1_{B_j}\r\}^{\frac1{d}}\r\|_{Y(\cx)}.\noz
\end{align*}
This shows \eqref{8.1.x1} and hence finishes
the proof of  Proposition \ref{atre}.
\end{proof}

Now, we formulate a decomposition result of $H_Y^*(\cx)$
(see \cite[Theorem 3.7]{shyy17} for the corresponding Euclidean case).
\begin{proposition}\label{atde}
Let $Y(\cx)$ be a ball quasi-Banach function space on $\cx$ satisfying Assumption \ref{assump1} with $p_-\in({\omega}/(\omega+\eta),\fz)$,
where $\omega$ is as in \eqref{eq-doub} and $\eta$ as in
Definition \ref{expati}. Further assume that $Y(\cx)$
satisfies Assumption \ref{assump2} with the same $p_-$
as in Assumption \ref{assump1},
$\tz_0\in({\omega}/(\omega+\eta),\underline{p})$,
and $p_0\in(\tz_0,\fz)$, where $\underline{p}$ is as in \eqref{2.1y}.
Let $d\in(0,\tz_0]$, $\bz$, $\gz\in(\omega(1/\tz_0-1),\eta)$,
and $f\in H_Y^*(\cx)$. Then there exist a sequence
$\{a_j\}_{j\in\nn}$ of $(Y(\cx),\fz)$-atoms supported,
respectively, in the balls $\{B_j \}_{j\in\nn}$,
and a sequence $\{\lz_j\}_{j\in\nn}\st[0,\fz)$ such that
\begin{align*}
f&=\sum_{j\in\nn}\lz_j a_j
\end{align*}
in $(\icgg)'$, and there exists a positive constant $C$,
independent of $f$, such that
$$\lf\|\lf\{\sum_{j\in\nn}
\lf[\frac{\lz_j}{\|\ch1_{B_j}\|_{Y(\cx)}}\r]^{d}
\ch1_{B_j}\r\}^{\frac{1}{d}}\r\|_{Y(\cx)}
\le C\|f\|_{H_Y^*(\cx)}.$$
\end{proposition}

To prove Proposition \ref{atde}, we need several technical lemmas.
The following two lemmas are, respectively, generalizations
of \cite[Propositions 4.4 and 4.5]{hhllyy} from the measure $\mu$
under consideration to any given doubling measure $\nu$,
whose proofs are slight modifications of
\cite[Propositions 4.4 and 4.5]{hhllyy};
we omit the details here.
\begin{lemma}\label{helem1}
Suppose that $\nu$ is a doubling measure,
$\Omega\subset\cx$ an open set with $\nu(\Omega)\in(0,\fz)$,
and $A\in[1,\fz)$. For any $x\in\Omega$, let
$$r(x):=\frac{\rho(x,\Omega^\complement)}{2AA_0}\in(0,\fz)$$
with $A_0$ as in \eqref{2.1x}.
Then there exist an $L_0\in\nn$ and a sequence
$\{x_k\}_{k\in I}\subset\Omega$,
where $I$ is a countable index set, such that
\begin{enumerate}
\item[{\rm (i)}] $\{B(x_k,r_k/(5A_0^3))\}_{k\in I}$ is disjoint.
Here and thereafter, $r_k:=r(x_k)$ for any $k\in I$;

\item[{\rm (ii)}] $\bigcup_{k\in I}B(x_k,r_k)=\Omega$ and
$B(x_k,Ar_k)\st\Omega$ for any $k\in I$;

\item[{\rm (iii)}] for any $k\in I$ and $x\in B(x_k,r_k)$,
$Ar_k\le \rho(x,\Omega^\complement)\le3AA_0^2r_k$;

\item[{\rm (iv)}] for any $k\in I$,
there exists a $y_k\notin\Omega$ such that $\rho(x_k,y_k)<3AA_0r_k$;

\item[{\rm (v)}] for any given $k\in I$, the number of balls
$\{B(x_j,Ar_j)\}_{j\in I}$ that intersect $B(x_k,Ar_k)$ is at most $L_0$;

\item[{\rm (vi)}] if, in addition, $\Omega$ is bounded,
then, for any $\sz\in(0,\fz)$, the set $\{k\in I:\ r_k>\sz\}$ is finite.
\end{enumerate}
\end{lemma}

\begin{lemma}\label{helem2}
Let $\nu$ be a doubling measure,
$\Omega\st\cx$ an open set with $\nu(\Omega)\in(0,\fz)$,
and $A_0$ as in \eqref{2.1x}.
Suppose that sequences $\{x_k\}_{k\in I}$ and $\{r_k\}_{k\in I}$
are as in Lemma \ref{helem1} with $A:=16A_0^4$ therein.
Then there exist non-negative functions $\{\phi_k\}_{k\in I}$ such that
\begin{enumerate}
\item[{\rm (i)}] for any $k\in I$, $0\le\phi_k\le1$ and
$\supp\phi_k\subset B(x_k,2A_0r_k)$;

\item[{\rm (ii)}] $\sum_{k\in I}\phi_k=\ch1_{\Omega}$;

\item[{\rm (iii)}] for any $k\in I$, $\phi_k\ge L_0^{-1}$
in $B(x_k,r_k)$, where $L_0$ is as in Lemma \ref{helem1};

\item[{\rm (iv)}] there exists a positive constant $C$ such that,
for any $k\in I$,
$\|\phi_k\|_{\cg(x_k,r_k,\eta,\eta)}\le CV_{r_k}(x_k)$
with $\eta$ as in Definition \ref{expati}.
\end{enumerate}
\end{lemma}

Let $Y(\cx)$, $\bz$, $\gz$, and $\eta$ be as in Proposition \ref{atde}.
Assume that $f\in H_Y^*(\cx)$. From Lemma \ref{7.9.x1},
it follows that $f\in (\icgg)'$.
Next, we aim to obtain the Calder\'on--Zygmund decomposition of $f$
by applying Lemmas \ref{helem1} and \ref{helem2} to
the level set $\{x\in\cx:\ f^\ast(x)>\lz\}$ with $\lz\in(0,\fz)$.
However, this level set may not be open even
in the case that $\rho$ is a metric. To solve this problem,
we borrow some ideas from  \cite[Section 2]{gly08}
and \cite[Theorem 2]{ms79} (see also \cite[Section 4.2]{hhllyy}).

For the quasi-metric $\rho$, by the proof of \cite[Theorem 2]{ms79},
we know that there exist a $\tz\in(0,1)$ and a metric $\widetilde{\rho}$
such that $\widetilde{\rho}\sim \rho^\tz$.
For any $x\in\cx$ and $r\in(0,\fz)$,
define the $\widetilde{\rho}$\emph{-ball} $\widetilde{B}(x,r)$ by setting
$$\widetilde{B}(x,r)
:=\lf\{y\in\cx:\ \widetilde{\rho}(x, y)<r\r\}.$$
Then $(\cx,\widetilde{\rho},\mu)$ is a doubling metric measure space
and, for any $x,\ y\in\cx$ and $r\in(0,\fz)$,
\begin{align}\label{3.22.x1}
\mu(B(x,r+\rho(x,y)))\sim\mu\lf(\widetilde{B}\lf(x,[r+\rho(x,y)]^{\tz}\r)\r)
\sim\mu\lf(\widetilde{B}\lf(x,r^{\tz}+\widetilde{\rho}(x,y)\r)\r),
\end{align}
where the implicit positive constants
are independent of $x$, $y$, and $r$.
Based on the metric $\widetilde{\rho}$, a variant of the space of
test functions was introduced in \cite[Definition 4.6]{hhllyy}.

\begin{definition}\label{hedef1}
For any $x_1\in\cx$, $R\in(0,\fz)$, and
$\widetilde{\bz},\ \widetilde{\gz}\in(0,\fz)$,
the \emph{space $G(x_1,R,\widetilde{\bz},\widetilde{\gz})$}
is defined to be the set of all functions $f$ on $\cx$
satisfying that there exists a positive constant $C$ such that
\begin{enumerate}
\item[{\rm (i)}] (the \emph{size condition}) for any $x\in\cx$,
$$|f(x)|\le C\frac{1}{\mu(\widetilde{B}(x,R+\widetilde{\rho}(x_1,x)))}
\lf[\frac{R}{R+\widetilde{\rho}(x_1,x)}\r]^{\widetilde{\gz}};$$

\item[{\rm (ii)}] (the \emph{regularity condition})
for any $x,\ y\in\cx$ satisfying $\widetilde{\rho}(x,y)\le[R+\widetilde{\rho}(x_1,x)]/2$,
$$|f(x)-f (y)|\le C\lf[\frac{\widetilde{\rho}(x,y)}
{R+\widetilde{\rho}(x_1,x)}\r]^{\widetilde{\bz}}
\frac{1}{\mu(\widetilde{B}(x,R+\widetilde{\rho}(x_1,x)))}
\lf[\frac{R}{R+\widetilde{\rho}(x_1,x)}\r]^{\widetilde{\gz}}.$$
\end{enumerate}
For any  $f\in G(x_1,R,\widetilde{\bz},\widetilde{\gz})$, its norm
$\|f\|_{G(x_1, R, \widetilde{\bz}, \widetilde{\gz})}$ in
$G(x_1,R,\widetilde{\bz},\widetilde{\gz})$ is defined by setting
\begin{align*}
\|f\|_{G(x_1,R,\widetilde{\bz},\widetilde{\gz})}:=
\inf\lf\{C\in(0,\fz):\ \mbox{(i) and (ii) hold true}\r\}.
\end{align*}
\end{definition}

By Lemma \ref{equlem} and \eqref{3.22.x1},
we know that there exists a $\tz\in(0,1)$
such that $$\cg(x_1,r,\bz,\gz)=G(x_1,r^{\tz},\bz/\tz,\gz/\tz)$$
with equivalent norms, and the positive equivalence
constants are independent of $x_1$ and $r$.
For any $\bz,\ \gz\in(0,\eta)$ and $f\in(\icgg)'$,
define the \emph{modified grand maximal function} $f^\star$ of $f$
by setting, for any $x\in\cx$,
$$f^\star(x):=\sup\lf\{|\langle f,\varphi\rangle|:\ \varphi\in\icgg\
\mbox{with}\ \|\varphi\|_{G(x,r^\tz,\bz/\tz,\gz/\tz)}\le 1\ \mbox{for some}
\ r\in(0,\fz)\r\}.$$
It is easy to see that $f^\star\sim f^\ast$ pointwisely on $\cx$.

Next, we establish an embedding lemma for $Y(\cx)$ as in
Proposition \ref{atde}, which is actually an application of
Theorem \ref{embed}.

\begin{lemma}\label{embedding}
Let $Y(\cx)$ be a ball quasi-Banach function space on $\cx$ satisfying Assumption \ref{assump1} with $p_-\in({\omega}/(\omega+\eta),\fz)$,
where $\omega$ is as in \eqref{eq-doub} and $\eta$ as in
Definition \ref{expati}. Further assume that $Y(\cx)$
satisfies Assumption \ref{assump2} with the same $p_-$
as in Assumption \ref{assump1},
$\tz_0\in({\omega}/(\omega+\eta),\underline{p})$,
and $p_0\in(\tz_0,\fz)$, where $\underline{p}$ is as in \eqref{2.1y}.
Let $x_0\in\cx$. Then there exists an $\ez\in(0,1)$ such that
$Y(\cx)$ is continuously embedded into $L^{\tz_0}_{w}(\cx)$
with $w:=[\cm(\ch1_{B(x_0,1)})]^{\ez}\in A_1(\cx)$.
\end{lemma}

\begin{proof}
By Assumption \ref{assump2} and \eqref{1.7.x1} with $\tz_1$
and $\tz_2$ replaced, respectively, by $1$ and $(p_0/\tz_0)'$,
we know that $Y^{1/{\tz_0}}(\cx)$ is a ball Banach function space
and $\cm$ bounded on $(Y^{1/{\tz_0}})'(\cx)$. From this
and Theorem \ref{embed}, we further deduce that
there exists an $\ez\in(0,1)$ such that $Y(\cx)$ is
continuously embedded into $L^{\tz_0}_{w}(\cx)$ with
$w:=[\cm(\ch1_{B(x_0,1)})]^{\ez}\in A_1(\cx)$.
This finishes the proof of Lemma \ref{embedding}.
\end{proof}

The following lemma shows that the set
$\{x\in\cx:\ f^\star(x)>\lz\}$ with $\lz\in(0,\fz)$ is an
open and $\mu_{w}$-measurable set with finite measure.

\begin{lemma}\label{omegafinite}
Let $f\in H^*_{Y}(\cx)$.
For any $\lz\in(0,\fz)$ and any $\mu$-measurable $E\st\cx$,
define
$$\Omega_{\lz}:=\lf\{x\in\cx:\ f^\star(x)>\lz\r\}$$
and
$$\mu_{w}(E):=w(E):=\int_{E}w(x)\,d\mu(x),$$
where $w$ is as in Lemma \ref{embedding}.
Then the following two statements hold true:
\begin{enumerate}
\item[{\rm (i)}] $\mu_{w}$ is a doubling measure;

\item[{\rm (ii)}] for any given $\lz$, $\Omega_{\lz}$ is open under
the topology induced by $\rho$, and $\mu_{w}(\Omega_{\lz})<\fz$.
\end{enumerate}
\end{lemma}

\begin{proof}
First, by \eqref{eq-doub}, the fact that $w\in A_1(\cx)$,
and \cite[Chapter 1, Lemma 12]{st89},
we know that $\mu_{w}$ is a doubling measure,
namely, (i) holds true. Next, we prove (ii).
By an argument similar to that used in \cite[Remark 2.9(iii)]{gly08},
we conclude that $\Omega_{\lz}$ is open
under the topology induced by $\rho$.
Moreover, from the fact that $f^\star\sim f^\ast$,
and Lemma \ref{embedding}, we deduce that,
for any given $\lz\in(0,\fz)$,
\begin{align*}
\mu_{w}(\Omega_{\lz})&=\int_{\Omega_{\lz}}w(x)\,d\mu(x)
\le\int_{\cx}\lf[\frac{f^\star(x)}{\lz}\r]^{\tz_0}w(x)\,d\mu(x)\\
&\sim\lz^{-\tz_0}\lf\|f^*\r\|^{\tz_0}_{L^{\tz_0}_{w}(\cx)}
\ls\lz^{-\tz_0}\lf\|f^*\r\|_{Y(\cx)}^{\tz_0}
\sim\lz^{-\tz_0}\|f\|^{\tz_0}_{H^*_{Y}(\cx)}<\fz,
\end{align*}
where $\tz_0$ is as in Lemma \ref{embedding} and the implicit
positive constants are independent of $\lz$ and $f$. This proves
(ii) and hence finishes the proof of Lemma \ref{omegafinite}.
\end{proof}

Let $Y(\cx)$, $\bz$, $\gz$, and $\eta$ be as in Proposition \ref{atde}.
Assume that $f\in H_Y^*(\cx)$. To decompose $f$ into a sum of $(Y(\cx),\fz)$-atoms,
we let $\{P_m\}_{m\in\zz}$ be a 1-exp-ATI and,
for any $m\in\zz$, let
\begin{align}\label{4.8x}
f^{(m)}:=P_mf.
\end{align}
In the remainder of this section,
we always assume that $m\in\nn$.
By the fact that $\bz$, $\gz\in(\omega(1/\tz_0-1),\eta)$,
Lemmas \ref{1eptilem}(iii), \ref{zhoulem1}, and \ref{zhoulem2},
and an argument similar to that used in \cite[p.\,210]{zhy},
we know that
\begin{align}\label{7.21.y1}
f^{(m)}\rightarrow f
\end{align}
in $(\icgg)'$ as $m\rightarrow\fz$ and, for any $m\in\nn$ and $x\in\cx$,
$(f^{(m)})^*(x)\ls f^*(x)$ with the implicit positive constant
independent of $m$, $f$, and $x$. Next, we deal with $f^{(m)}$.
To this end, for any $j\in\zz$, let
\begin{align}\label{4.10x}
\Omega_{j}:=\lf\{x\in\cx:\ f^\star(x)>2^j\r\}.
\end{align}
By Lemma \ref{omegafinite}(ii), we know that, for any $j\in\zz$,
$\Omega_{j}$ is open and $\mu_w(\Omega_{j})<\fz$.
From this, and Lemmas \ref{helem1} and \ref{helem2}
with $\nu:=\mu_w$ therein, we further deduce that,
for any $j\in\zz$, there exist $\{x_{j,k}\}_{k\in I_j}\subset\cx$
with $I_j$ being a countable index set,
$\{r_{j,k}\}_{k\in I_j}\subset(0,\fz)$,
$L_0\in\nn$, and a sequence $\{\phi_{j,k}\}_{k\in I_j}$ of
non-negative functions satisfying all the conclusions of
Lemmas \ref{helem1} and \ref{helem2}. For any $m\in\nn$,
$j\in\zz$, and $k\in I_j$, let
\begin{align}\label{5.28.x1}
P_{j,k}^{(m)}:=\frac{1}{\|\phi_{j,k}\|_{L^1(\cx)}}
\int_{\cx}f^{(m)}(\xi)\phi_{j,k}(\xi)\,d\mu(\xi)\quad\mbox{and}\quad
b_{j,k}^{(m)}:=\lf(f^{(m)}-P_{j,k}^{(m)}\r)\phi_{j,k}.
\end{align}

We then have the following technical lemma.

\begin{lemma}\label{prop4.13}
For any $m\in\nn$, $j\in\zz$, and $k\in I_j$,
let $b_{j,k}^{(m)}$ be as in \eqref{5.28.x1}.
Then the following statements hold true:
\begin{enumerate}
\item[{\rm (i)}] there exists a positive constant $C$ such that,
for any $m\in\nn$ and $j\in\zz$,
\begin{align}\label{5.24.x5m}
\lf\|\sum_{k\in I_j}\lf(b_{j,k}^{(m)}\r)^*\r\|_{Y(\cx)}
&\le C\lf\|f^\ast\ch1_{\Omega_j}\r\|_{Y(\cx)};
\end{align}
moreover, $b_j^{(m)}:=\sum_{k\in I_j}b_{j,k}^{(m)}$
converges in $(\icgg)'$ satisfying that $b_j^{(m)}\in H_Y^*(\cx)$
and, for any $x\in\cx$,
\begin{align}\label{5.24.x2m}
\lf(b_j^{(m)}\r)^\ast(x)&\le C2^j\sum_{k\in I_j}
\frac{{V_{r_{j,k}}(x_{j,k})}}{{V_{r_{j,k}}(x_{j,k})}+V(x_{j,k},x)}
\lf[\frac{r_{j,k}}{r_{j,k}
+\rho(x_{j,k},x)}\r]^\bz+Cf^\ast(x)\ch1_{\Omega_j}(x);
\end{align}

\item[{\rm (ii)}] for any $m\in\nn$, $b_j^{(m)}\rightarrow0$
in $(\icgg)'$ as $j\rightarrow\fz$;

\item[{\rm (iii)}] for any $m\in\nn$ and $j\in\zz$,
let $g_j^{(m)}:=f^{(m)}-b_j^{(m)}$. Then, for any $m\in\nn$,
$g_j^{(m)}\rightarrow0$ in $(\icgg)'$ as $j\rightarrow-\fz$.
\end{enumerate}
\end{lemma}

\begin{proof}
Let all the symbols be as in the present lemma and
$${\rm{J}}:=
\frac{{V_{r_{j,k}}(x_{j,k})}}{{V_{r_{j,k}}(x_{j,k})}+V(x_{j,k},x)}
\lf[\frac{r_{j,k}}{r_{j,k}+\rho(x_{j,k},x)}\r]^\bz.$$
We first prove (i). By \cite[(4.21)]{zhy} and
Definition \ref{qbs}(ii), we conclude that,
for any $m\in\nn$ and $j\in\zz$,
\begin{align}\label{5.24.y1}
\lf\|\sum_{k\in I_j}\lf(b_{j,k}^{(m)}\r)^\ast\r\|_{Y(\cx)}
&\ls\lf\|\sum_{k\in I_j}f^\ast\ch1_{B(x_{j,k},Ar_{j,k})}\r\|_{Y(\cx)}
+\lf\|\sum_{k\in I_j}2^jJ
\ch1_{[B(x_{j,k},Ar_{j,k})]^\complement}\r\|_{Y(\cx)}\noz\\
&=:{\rm I}_1+{\rm I}_2,
\end{align}
where $A:=16A_0^4$. By (ii) and (v) of Lemma \ref{helem1},
we obtain
\begin{align}\label{7.6.x1}
{\rm I}_1&\ls\lf\|f^\ast\ch1_{\Omega_j}\r\|_{Y(\cx)}.
\end{align}
To estimate ${\rm I}_2$, since $\bz\in(\omega(1/\tz_0-1),\eta)$,
it follows that we can choose $h\in(\frac{\oz}{\oz+\bz},\tz_0)$.
By Definition \ref{qbs}(ii), Lemma \ref{lem6.2}, \eqref{eq-doub},
Definition \ref{cvex}, $\oz(1/h-1)<\bz$, Assumption \ref{assump1},
\eqref{4.10x} with the fact that $f^\star\sim f^\ast$,
and (ii) and (v) of Lemma \ref{helem1}, we conclude that
\begin{align*}
{\rm I}_2&\sim\lf\|\sum_{k\in I_j}
\sum_{s\in\zz_+}2^jJ\ch1_{B(x_{j,k},2^{s+1}Ar_{j,k})
\setminus B(x_{j,k},2^{s}Ar_{j,k})}\r\|_{Y(\cx)}\\
&\ls\lf\|\sum_{k\in I_j}2^j\sum_{s\in\zz_+}
2^{-s\bz}\frac{{V_{r_{j,k}}(x_{j,k})}}{V_{2^{s}Ar_{j,k}}(x_{j,k})}
\ch1_{B(x_{j,k},2^{s+1}Ar_{j,k})}\r\|_{Y(\cx)}\noz\\
&\sim\lf\|\sum_{k\in I_j}2^j\sum_{s\in\zz_+}2^{-s\bz}
\lf[\frac{{V_{2^{s}Ar_{j,k}}(x_{j,k})}}{V_{Ar_{j,k}}(x_{j,k})}\r]
^{\frac{1}{h}-1}
\lf[\frac{V_{Ar_{j,k}}(x_{j,k})}{{V_{2^{s}Ar_{j,k}}(x_{j,k})}}\r]
^{\frac{1}{h}}\ch1_{B(x_{j,k},2^{s+1}Ar_{j,k})}\r\|_{Y(\cx)}\noz\\
&\ls\lf\|\sum_{k\in I_j}2^j
\sum_{s\in\zz_+}2^{[\oz(\frac1h-1)-\bz]s}
\lf[\cm\lf(\ch1_{B(x_{j,k},Ar_{j,k})}\r)\r]
^{\frac{1}{h}}\r\|_{Y(\cx)}
\ls\lf\|\lf\{\sum_{k\in I_j}2^j
\lf[\cm\lf(\ch1_{B(x_{j,k},Ar_{j,k})}\r)\r]
^{\frac{1}{h}}\r\}^h\r\|_{Y^{\frac1h}(\cx)}^{\frac1h}\noz\\
&\ls\lf\|\sum_{k\in I_j}
2^j\ch1_{B(x_{j,k},Ar_{j,k})}\r\|_{Y(\cx)}
\ls\lf\|\sum_{k\in I_j}f^\ast
\ch1_{B(x_{j,k},Ar_{j,k})}\r\|_{Y(\cx)}
\ls\lf\|f^\ast\ch1_{\Omega_j}\r\|_{Y(\cx)},\noz
\end{align*}
which, combined with \eqref{5.24.y1} and \eqref{7.6.x1},
further implies \eqref{5.24.x5m}.

Next, we show that, for any $m\in\nn$ and $j\in\zz$,
$\sum_{k\in I_j}b_{j,k}^{(m)}$ converges in $(\icgg)'$.
Without loss of generality,
we may assume that $\vz\in\icgg$ with $\|\vz\|_{\icgg}\le1$.
Let $x_0\in\cx$ be as in Definition \ref{test}.
From an argument similar to that used in the proof of \eqref{3.19.x1},
(ii) and (iv) of Definition \ref{qbs}, and \eqref{5.24.x5m},
we deduce that, for any $m\in\nn$ and $j\in\zz$,
\begin{align*}
\sum_{k\in I_j}\lf|\lf\langle b_{j,k}^{(m)},\varphi\r\rangle\r|
&\ls\sum_{k\in I_j}\inf_{x\in B(x_0,1)}\lf(b_{j,k}^{(m)}\r)^*(x)
\ls\inf_{x\in B(x_0,1)}\sum_{k\in I_j}\lf(b_{j,k}^{(m)}\r)^*(x)\\
&\ls\lf\|\sum_{k\in I_j}\lf(b_{j,k}^{(m)}\r)^*\r\|_{Y(\cx)}
\lf\|\ch1_{B(x_0,1)}\r\|_{Y(\cx)}^{-1}
\sim\lf\|\sum_{k\in I_j}\lf(b_{j,k}^{(m)}\r)^*\r\|_{Y(\cx)}\\
&\ls\lf\|f^\ast\ch1_{\Omega_j}\r\|_{Y(\cx)}
\ls\|f\|_{H_Y^*(\cx)}<\fz,
\end{align*}
which further implies that
$\sum_{k\in I_j}b_{j,k}^{(m)}$ in $(\icgg)'$,
where the implicit positive constants depend on $x_0$.
Let $b_j^{(m)}:=\sum_{k\in I_j}b_{j,k}^{(m)}$ in $(\icgg)'$.
Then, by \eqref{5.24.x5m} and \cite[(4.21)]{zhy},
we obtain $b_j^{(m)}\in H_Y^*(\cx)$ and \eqref{5.24.x2m}.
This finishes the proof of (i).

Next, we prove (ii). Without loss of generality,
we may assume that $\vz\in\icgg$ with $\|\vz\|_{\icgg}\le 1$.
Let $x_0\in\cx$ be as in Definition \ref{test}.
By an argument similar to that used in
the proof of \eqref{3.19.x1}, we conclude that,
for any $m\in\nn$ and $j\in\zz$,
\begin{align}\label{6.23.x1}
\lf|\lf\langle b_j^{(m)},\varphi\r\rangle\r|
&\ls\inf_{x\in B(x_0,1)}\lf(b_j^{(m)}\r)^*(x)
\ls\lf\|\lf(b_j^{(m)}\r)^*\r\|_{L^{\tz_0}_{w}(\cx)}
\lf\|\ch1_{B(x_0,1)}\r\|_{L^{\tz_0}_{w}(\cx)}^{-1}
\ls\lf\|\lf(b_j^{(m)}\r)^*\r\|_{L^{\tz_0}_{w}(\cx)},
\end{align}
where $w$ is as in Lemma \ref{embedding} and the implicit
positive constants depend on $x_0$.
Moreover, on one hand, by \eqref{hlmax},
we know that, for any $j\in\zz$, $k\in I_j$,
and $x\in B(x_{j,k},2A_0Ar_{j,k})$,
\begin{align*}
{\rm{J}}\le1\ls\lf[\cm\lf(\ch1_{B(x_{j,k},Ar_{j,k})}\r)(x)\r]
^{{\frac{\bz+\oz}{\oz}}};
\end{align*}
on the other hand, by \eqref{eq-doub}, we conclude that,
for any $j\in\zz$, $k\in I_j$, and $x\in(B(x_{j,k},2A_0Ar_{j,k}))^\complement$,
\begin{align*}
{\rm{J}}\le\frac{{V_{r_{j,k}}(x_{j,k})}}{V(x_{j,k},x)}
\lf[\frac{r_{j,k}}{\rho(x_{j,k},x)}\r]^\bz
\ls\lf[\cm\lf(\ch1_{B(x_{j,k},Ar_{j,k})}\r)(x)\r]
^{{\frac{\bz+\oz}{\oz}}}.
\end{align*}
To summarize, for any $j\in\zz$, $k\in I_j$, and $x\in\cx$,
it holds true that
\begin{align*}
{\rm{J}}\ls\lf[\cm\lf(\ch1_{B(x_{j,k},Ar_{j,k})}\r)(x)\r]
^{{\frac{\bz+\oz}{\oz}}}.
\end{align*}
From this, \eqref{6.23.x1}, \eqref{5.24.x2m}, the facts that $\bz>\omega(1/{\tz_0}-1)$ and $f^\star\sim f^\ast$,
the Fefferman--Stein vector-valued maximal inequality of $L^{[\tz_0(\bz+\oz)]/{\oz}}_{w}(\cx)$ (see \cite[Theorem 6.5(ii)]{fmy19}),
and (ii) and (v) of Lemma \ref{helem1}, we deduce that,
for any $m\in\nn$ and $j\in\zz$,
\begin{align*}
\lf|\lf\langle b_j^{(m)},\varphi\r\rangle\r|
&\ls\lf\|2^j\sum_{k\in I_j}
\lf[\cm\lf(\ch1_{B(x_{j,k},Ar_{j,k})}\r)\r]
^{{\frac{\bz+\oz}{\oz}}}\r\|_{L^{\tz_0}_{w}(\cx)}
+\lf\|f^\ast\ch1_{\Omega_j}\r\|_{L^{\tz_0}_{w}(\cx)}\\
&\sim\lf\|\lf\{2^j\sum_{k\in I_j}
\lf[\cm\lf(\ch1_{B(x_{j,k},Ar_{j,k})}\r)\r]
^{{\frac{\bz+\oz}{\oz}}}\r\}^{\frac{\oz}{\bz+\oz}}\r\|
^{\frac{\bz+\oz}{\oz}}_{L^{\frac{\tz_0(\bz+\oz)}{\oz}}_{w}(\cx)}
+\lf\|f^\ast\ch1_{\Omega_j}\r\|_{L^{\tz_0}_{w}(\cx)}\\
&\ls\lf\|2^j\sum_{k\in I_j}
\ch1_{B(x_{j,k},Ar_{j,k})}\r\|_{L^{\tz_0}_{w}(\cx)}
+\lf\|f^\ast\ch1_{\Omega_j}\r\|_{L^{\tz_0}_{w}(\cx)}
\ls\lf\|f^\ast\ch1_{\Omega_j}\r\|_{L^{\tz_0}_{w}(\cx)},
\end{align*}
which, combined with the dominated convergence theorem and the
fact that $\Omega_j\downarrow\emptyset$ as $j\rightarrow\fz$,
then completes the proof of (ii).

Finally, we prove (iii). By Lemmas \ref{zhoulem1}, \ref{zhoulem2},
\ref{helem1}, and \ref{helem2}, and an argument similar to that
used in the proof of \cite[(4.26)]{zhy}, we conclude that,
for any $m\in\nn$ and $j\in\zz$,
$\|g_j^{(m)}\|_{L^{\fz}(\cx)}\ls 2^j$,
which further implies (iii) and hence completes
the proof of Lemma \ref{prop4.13}.
\end{proof}

Next, we recall the atomic characterization of
Musielak--Orlicz Hardy spaces on $\cx$,
which plays a key role in the proof of Proposition \ref{atde}
and was originally established in \cite{fmy19}.
To this end, we first recall some notions on
Musielak--Orlicz Hardy spaces.

\begin{definition}\label{fuvzdefn}
Let $\beta,\ \gamma \in (0, \eta)$ with $\eta$ as
in Definition \ref{expati}, and $\varphi$ be a growth
function as in Remark \ref{qbsdefrem}(iii).
Then the \emph{Musielak--Orlicz Hardy space} $H^{*,\vz}(\cx)$
is defined by setting
$$H^{*,\vz}(\cx):=\lf\{f\in\lf(\icgg\r)':\
\|f\|_{H^{*,\vz}(\cx)}:=\lf\|f^*\r\|_{L^{\vz}(\cx)}<\fz\r\}.$$
\end{definition}

\begin{definition}\label{fudefn}
Let $E$ be a $\mu$-measurable subset of $\cx$ and $q\in[1,\fz]$,
the \emph{space} $L^q_{\vz}(E)$ is defined to be the set of all
$\mu$-measurable functions $f$ on $E$ such that
\begin{align*}
\|f\|_{L^q_\vz(E)}:=
\begin{cases}
\displaystyle\sup_{t\in(0,\fz)}\lf[\frac1{\vz(E,t)}\int_{E}
|f(x)|^q\vz(x,t)\,d\mu(x)\r]^{1/q}
\ \ &\text{when}\ q\in[1,\fz),\\
\|f\|_{L^{\fz}(E)} \ \ &\text{when}\ q=\fz
\end{cases}
\end{align*}
is finite, here and thereafter,
$\vz(E,t):=\int_{E}\vz(x,t)\,d\mu(x)$.
\end{definition}

\begin{definition}\label{fuatomdefn}
Let $\oz$ be as in \eqref{eq-doub}, $\varphi$ a growth function
as in Remark \ref{qbsdefrem}(iii), $q\in(q(\vz),\fz]$, and
\begin{align*}
m(\vz):=\lf\lfloor\oz\lf[\frac{q(\vz)}{i(\vz)}-1\r]\r\rfloor\le0,
\end{align*}
where $q(\vz)$ is as in \eqref{qvz} and
\begin{align*}
i(\vz):=\sup\lf\{p\in(0,\fz):\ \vz\ \mbox{is of uniformly lower type}\ p\r\}.
\end{align*}
A $\mu$-measurable function $a$ is called a \emph{$(\vz,q)$-atom}
supported in a ball $B\st\cx$ if the following three conditions
hold true:
\begin{enumerate}
\item[{\rm (i)}] $\supp a:=\{x\in\cx:\ a(x)\neq0\}\st B$;

\item[{\rm (ii)}] $a\in L^q_{\vz}(B)$ and
$\|a\|_{L^q_{\vz}(B)}\le\|\ch1_{B}\|^{-1}_{L^{\vz}(\cx)}$;

\item[{\rm (iii)}] $\int_{\cx}a(x)\,d\mu(x)=0$.
\end{enumerate}
\end{definition}

The following lemma is just \cite[Lemma 4.9]{fmy19}.

\begin{lemma}\label{fulem}
Let $\beta,\ \gamma \in (0, \eta)$ with $\eta$ as in Definition
\ref{expati}, and $\vz$ be as in Definition \ref{fuatomdefn}.
Then $H^{*,\vz}(\cx)$ continuously embeds into $(\icgg)'$.
\end{lemma}

The following lemma is just \cite[Lemma 2.8(ii)]{fmy19}
(see \cite[Lemma 1.1.10(i)]{ylk17} for the corresponding
Euclidean case).
\begin{lemma}\label{fulemma}
Let $\vz$ be as in Definition \ref{fuatomdefn}.
Then, for any $f\in L^{\vz}(\cx)\setminus\{0\}$,
\begin{align*}
\int_{\cx}\vz\lf(x,\frac{|f(x)|}{\|f\|_{L^{\vz}(\cx)}}\r)\,d\mu(x)=1.
\end{align*}
\end{lemma}

To prove Proposition \ref{atde}, we also need the following lemma.
\begin{lemma}\label{futheorem}
Let $\vz$ and $q$ be as in Definition \ref{fuatomdefn}.
Then there exists a positive constant $C$ such that,
for any $(\vz,q)$-atom $a$, $\|a\|_{H^{*,\vz}(\cx)}\le C$.
\end{lemma}

\begin{proof}
Let $a$ be a $(\vz,q)$-atom supported in $B$. Then,
by \cite[(5.2)]{fmy19}, the fact that $\varphi$ is a growth function,
and Lemma \ref{fulemma}, we know that
\begin{align*}
\int_{\cx}\vz\lf(x,a^*(x)\r)\,d\mu(x)\ls\vz\lf(B,\|a\|_{L^q_{\vz}(B)}\r)
\ls\vz\lf(B,\|\ch1_B\|^{-1}_{L^{\vz}(\cx)}\r)\sim1,
\end{align*}
which further implies that
\begin{align*}
\|a\|_{H^{*,\vz}(\cx)}=\lf\|a^*\r\|_{L^{\vz}(\cx)}\ls 1.
\end{align*}
This finishes the proof of Lemma \ref{futheorem}.
\end{proof}

Now, we give the proof of Proposition \ref{atde}.

\begin{proof}[Proof of Proposition \ref{atde}]
Let all the symbols be as in the present proposition.
For any $m\in\nn$ and $j\in\zz$, let $f^{(m)}$ be as in \eqref{4.8x},
and $b_{j}^{(m)}$ and $g_{j}^{(m)}$ as in Lemma \ref{prop4.13}.
Then, by Lemma \ref{prop4.13}, we know that,
for any $m\in\nn$, and any $J_1$, $J_2\in\nn$,
\begin{align*}
f^{(m)}-\sum_{j=-J_1}^{J_2}\lf[g_{j+1}^{(m)}-g_{j}^{(m)}\r]
&=f^{(m)}-\lf[g_{J_2+1}^{(m)}-g_{-J_1}^{(m)}\r]\\
&=b_{J_2+1}^{(m)}+g_{-J_1}^{(m)}\rightarrow0
\end{align*}
in $(\icgg)'$, as $J_1$, $J_2\rightarrow\fz$,
which further implies that
\begin{align}\label{7.21.x3}
f^{(m)}=\sum_{j\in\zz}\lf[g_{j+1}^{(m)}-g_{j}^{(m)}\r]
\end{align}
in $(\icgg)'$. For any $m\in\nn$, $j\in\zz$, $k\in I_j$,
and $\ell\in I_{j+1}$, let $\phi_{j,k}$, $P_{j,k}^{(m)}$,
and $b_{j,k}^{(m)}$ be as in \eqref{5.28.x1},
$$L_{j+1,k}^{(m),\ell}:=\frac{1}{\|\phi_{j+1,{\ell}}\|_{L^1(\cx)}}
\int_{\cx}\lf[f^{(m)}(\xi)-P_{j+1,{\ell}}^{(m)}(\xi)\r]
\phi_{j,k}(\xi)\phi_{j+1,{\ell}}(\xi)\,d\mu(\xi),$$
and
$$h^{(m)}_{j,k}:=b_{j,k}^{(m)}-\sum_{{\ell}\in I_{j+1}}
\lf[b_{j+1,{\ell}}^{(m)}\phi_{j,k}-L_{j+1,k}^{(m),{\ell}}
\phi_{j+1,{\ell}}\r].$$
Then, by \eqref{7.21.x3} and an argument similar to that used in
the proof of \cite[p.\,214]{zhy}, we conclude that, for any $m\in\nn$,
$j\in\zz$, and $k\in I_j$,
\begin{align}\label{7.5.x1}
\lf\|h_{j,k}^{(m)}\r\|_{L^{\fz}(\cx)}\ls 2^j,\
\supp h_{j,k}^{(m)}\st B(x_{j,k},Ar_{j,k}),\
\int_{\cx}h_{j,k}^{(m)}(x)\,d\mu(x)=0,
\end{align}
and
\begin{align}\label{7.3.y1}
g_{j+1}^{(m)}-g_{j}^{(m)}=\sum_{k\in I_j}h^{(m)}_{j,k}
\end{align}
in $(\icgg)'$, where $A:=16A_0^4$.
Moreover, by \eqref{7.5.x1} and the Banach--Alaoglu theorem
(see, for instance, \cite[Theorem 3.17]{rudin91}),
together with the well-known diagonal rule for series two times,
we conclude that there exist a subsequence
$\{m_q\}_{q=1}^{\fz}\st\nn$ and a sequence
$\{h_{j,k}\}_{k\in I_j}\st L^{\fz}(\cx)$ with any $j\in\zz$
such that, for any $j\in\zz$ and  $k\in I_j$,
\begin{align}\label{3.30.x1}
\lf\|h_{j,k}\r\|_{L^{\fz}(\cx)}\ls2^j,
\quad\supp h_{j,k}\st B(x_{j,k},Ar_{j,k}),
\quad\int_{\cx}h_{j,k}(x)\,d\mu(x)=0,
\end{align}
and, for any $g\in L^1(\cx)$,
\begin{align*}
\lf\langle h_{j,k}^{(m_q)},g\r\rangle\rightarrow
\lf\langle h_{j,k},g\r\rangle
\end{align*}
as $q\rightarrow\fz$. Now, we claim that, for any given $j\in\zz$,
$\sum_{k\in I_j}h_{j,k}$ converges in $(\icgg)'$.
Indeed, by \eqref{3.30.x1}, (ii) and (v) of Lemma \ref{helem1},
and Lemma \ref{omegafinite}(ii),
we know that, for any given $j\in\zz$ and $r\in(1,\fz)$,
\begin{align*}
\lf\|\sum_{k\in I_j}h_{j,k}\r\|_{L^r_w(\cx)}
&\ls 2^j\lf\|\sum_{k\in I_j}\ch1_{B(x_{j,k},Ar_{j,k})}\r\|_{L^r_w(\cx)}
\ls 2^j\lf\|\ch1_{\Omega_j}\r\|_{L^r_w(\cx)}
\sim 2^j\lf[\mu_{w}(\Omega_j)\r]^{\frac1r}<\fz,
\end{align*}
where $w$ and $\mu_{w}$ are as in Lemma \ref{omegafinite}.
This further implies the above claim.

For any $j\in\zz$ and $k\in I_j$, let
\begin{align}\label{7.22.x1}
\lz_{j,k}:=2^j\lf\|\ch1_{B(x_{j,k},Ar_{j,k})}\r\|_{Y(\cx)}
\quad\mbox{and}\quad a_{j,k}:=\frac{h_{j,k}}{\lz_{j,k}}.
\end{align}
Then, by \eqref{3.30.x1}, we know that $a_{j,k}$ is a
$(Y(\cx),\fz)$-atom supported in $B(x_{j,k},Ar_{j,k})$.
Next, we prove
\begin{align}\label{3.26.x1}
f=\sum_{j\in\zz}\sum_{k\in I_j}\lz_{j,k}a_{j,k}
\end{align}
in $(\icgg)'$, and
\begin{align}\label{21.7.24.x1}
{\rm{K}}:=\lf\|\lf\{\sum_{j\in\zz}\sum_{k\in I_j}
\lf[\frac{\lz_{j,k}}{\|\ch1_{B(x_{j,k},Ar_{j,k})}\|_{Y(\cx)}}\r]^{d}
\ch1_{B(x_{j,k},Ar_{j,k})}\r\}^{\frac{1}{d}}\r\|_{Y(\cx)}
\ls\|f\|_{H_Y^*(\cx)}.
\end{align}
To this end, for any $m\in\nn$ and $j\in\zz$,
let $f_j^{(m)}:=g_{j+1}^{(m)}-g_{j}^{(m)}$
and $f_j:=\sum_{k\in I_j}h_{j,k}$ in $(\icgg)'$.
It suffices to show that
\begin{align}\label{12.30.x1}
\lim_{J_3\rightarrow-\fz}\sup_{m\in\nn}
\lf\|\sum_{j\le J_3}f_j^{(m)}\r\|_{H^1_{w}(\cx)}=0
\quad\mbox{and}\quad
\lim_{J_4\rightarrow\fz}\sup_{m\in\nn}
\lf\|\sum_{j\ge J_4}f_j^{(m)}\r\|_{H^{p_1}_{w}(\cx)}=0
\end{align}
for some $p_1\in(\frac{\oz}{\oz+\eta},\tz_0)$,
where $w$ is as in Lemma \ref{embedding}.
Assuming this for the moment, by \eqref{12.30.x1},
Lemma \ref{fulem}, and an argument similar to that
used in \cite[pp.\,215-217]{zhy}, we obtain
\begin{align*}
\lim_{q\rightarrow\fz}\sum_{j\in\zz}f_j^{(m_q)}
=\sum_{j\in\zz}f_j
\end{align*}
in $(\icgg)'$, which, combined with \eqref{7.21.y1}, \eqref{7.21.x3},
and \eqref{7.22.x1}, further implies that
\begin{align*}
f&=\lim_{m\rightarrow\fz}f^{(m)}
=\lim_{m\rightarrow\fz}\sum_{j\in\zz}\lf[g_{j+1}^{(m)}-g_{j}^{(m)}\r]
=\lim_{q\rightarrow\fz}\sum_{j\in\zz}f_j^{(m_q)}\\
&=\sum_{j\in\zz}f_j
=\sum_{j\in\zz}\sum_{k\in I_j}h_{j,k}
=\sum_{j\in\zz}\sum_{k\in I_j}\lz_{j,k}a_{j,k}
\end{align*}
in $(\icgg)'$, and hence \eqref{3.26.x1} holds true.
Moreover, from Definition \ref{qbs}(ii),
(ii) and (v) of Lemma \ref{helem1},
and the fact that $f^\star\sim f^*$, we deduce that
\begin{align}\label{7.5.x2}
{\rm{K}}&\ls\lf\|\lf\{\sum_{j\in\zz}2^{jd}\ch1_{\Omega_{j}}\r\}
^{\frac{1}{d}}\r\|_{Y(\cx)}
\sim\lf\|\lf\{\sum_{j\in\zz}\sum_{r=j}^{\fz}
2^{jd}\ch1_{\Omega_{r}\setminus\Omega_{r+1}}\r\}
^{\frac{1}{d}}\r\|_{Y(\cx)}\noz\\
&\sim\lf\|\lf\{\sum_{r\in\zz}2^{rd}\sum_{j=-\fz}^{r}
2^{(j-r)d}\ch1_{\Omega_{r}\setminus\Omega_{r+1}}\r\}
^{\frac{1}{d}}\r\|_{Y(\cx)}
\sim\lf\|\lf\{\sum_{r\in\zz}2^{rd}
\ch1_{\Omega_{r}\setminus\Omega_{r+1}}\r\}
^{\frac{1}{d}}\r\|_{Y(\cx)}\noz\\
&\ls\lf\|f^\star\lf\{\sum_{r\in\zz}
\ch1_{\Omega_{r}\setminus\Omega_{r+1}}\r\}
^{\frac{1}{d}}\r\|_{Y(\cx)}
\sim\|f^*\|_{Y(\cx)}\sim\|f\|_{H_Y^*(\cx)}.
\end{align}
Thus, \eqref{21.7.24.x1} holds true, which then completes
the proof of this proposition.

Now, we show \eqref{12.30.x1}.
On one hand, for any $x\in\cx$ and $t\in[0,\fz)$,
let $\vz_1(x,t):=w(x)t$ with $w$ as in Lemma \ref{embedding}.
Obviously, $\vz_1$ satisfies all the assumptions of
Lemma \ref{futheorem}. From this, \eqref{7.3.y1}, \eqref{7.5.x1},
Lemma \ref{futheorem}, (ii) and (v) of Lemma \ref{helem1},
the fact that $f^\star\sim f^*$, and Lemma \ref{embedding},
we deduce that, for any $m\in\nn$ and $J_3\in\zz$,
\begin{align*}
\lf\|\sum_{j\le J_3}f_j^{(m)}\r\|_{H^1_{w}(\cx)}
&\le\sum_{j\le J_3}\sum_{k\in I_j}
\lf\|h_{j,k}^{(m)}\r\|_{H^1_{w}(\cx)}\\
&=\sum_{j\le J_3}\sum_{k\in I_j}
2^j\lf\|\ch1_{B(x_{j,k},Ar_{j,k})}\r\|_{L^1_{w}(\cx)}
\lf\|\frac{h_{j,k}^{(m)}}
{2^j\|\ch1_{B(x_{j,k},Ar_{j,k})}\|_{L^1_{w}(\cx)}}
\r\|_{H^1_{w}(\cx)}\noz\\
&\ls\sum_{j\le J_3}\sum_{k\in I_j}2^jw(B(x_{j,k},Ar_{j,k}))
\ls\sum_{j\le J_3}2^jw(\Omega_j)\noz\\
&\ls\sum_{j\le J_3}2^j\int_{\Omega_j}
\lf[\frac{f^\star(x)}{2^j}\r]^{\tz_0}w(x)\,d\mu(x)
\ls\sum_{j\le J_3}2^{j(1-\tz_0)}
\lf\|f^{\star}\r\|^{\tz_0}_{L^{\tz_0}_{w}(\cx)}\noz\\
&\ls\sum_{j\le J_3}2^{j(1-\tz_0)}
\lf\|f^{\star}\r\|^{\tz_0}_{Y(\cx)}
\sim 2^{J_3(1-\tz_0)}\lf\|f\r\|^{\tz_0}_{H_Y^*(\cx)}\noz
\end{align*}
with the implicit positive constants independent of $m$ and $J_3$,
which, combined with $\tz_0\in(0,1)$ and $J_3\rightarrow-\fz$,
further implies the first estimate of \eqref{12.30.x1}.

On the other hand, choose $p_1\in(\frac{\oz}{\oz+\eta},\tz_0)$
and let $\vz_2(x,t):=w(x)t^{p_1}$ for any $x\in\cx$
and $t\in[0,\fz)$ with $w$ as in Lemma \ref{embedding}.
Obviously, $\vz_2$ also satisfies all the assumptions of
Lemma \ref{futheorem}.
Using this, \eqref{7.5.x1}, Lemma \ref{futheorem},
(ii) and (v) of Lemma \ref{helem1}, the fact that $f^\star\sim f^*$,
and Lemma \ref{embedding}, we know that,
for any $m\in\nn$ and $J_4\in\zz$,
\begin{align*}
\lf\|\sum_{j\ge J_4}f_j^{(m)}\r\|^{p_1}_{H^{p_1}_{w}(\cx)}
&\le\sum_{j\ge J_4}\sum_{k\in I_j}
\lf\|h_{j,k}^{(m)}\r\|^{p_1}_{H^{p_1}_{w}(\cx)}\\
&=\sum_{j\ge J_4}\sum_{k\in I_j}
2^{jp_1}\lf\|\ch1_{B(x_{j,k},Ar_{j,k})}\r\|^{p_1}
_{L^{p_1}_{w}(\cx)}
\lf\|\frac{h_{j,k}^{(m)}}
{2^j\|\ch1_{B(x_{j,k},Ar_{j,k})}\|_{L^{p_1}_{w}(\cx)}}
\r\|^{p_1}_{H^{p_1}_{w}(\cx)}\noz\\
&\ls\sum_{j\ge J_4}\sum_{k\in I_j}
2^{jp_1}w(B(x_{j,k},Ar_{j,k}))
\ls\sum_{j\ge J_4}2^{jp_1}w(\Omega_j)\noz\\
&\ls\sum_{j\ge J_4}2^{jp_1}\int_{\Omega_j}
\lf[\frac{f^\star(x)}{2^j}\r]^{\tz_0}w(x)\,d\mu(x)
\ls\sum_{j\ge J_4}2^{j(p_1-\tz_0)}
\lf\|f^{\star}\r\|^{\tz_0}_{L^{\tz_0}_{w}(\cx)}\noz\\
&\ls\sum_{j\ge J_4}2^{j(p_1-\tz_0)}
\lf\|f^{\star}\r\|^{\tz_0}_{Y(\cx)}
\sim 2^{J_4(p_1-\tz_0)}\lf\|f\r\|^{\tz_0}_{H_Y^*(\cx)}\noz
\end{align*}
with the implicit positive constants independent of $m$ and $J_4$,
which, combined with $p_1\in(\frac{\oz}{\oz+\eta},\tz_0)$
and $J_4\rightarrow\fz$, further implies the second estimate of \eqref{12.30.x1}
and hence completes the proof of Proposition \ref{atde}.
\end{proof}

Next, we introduce the atomic Hardy spaces
associated with $Y(\cx)$.

\begin{definition}\label{atomhy}
Let $Y(\cx)$ be a ball quasi-Banach function space on $\cx$ satisfying Assumption \ref{assump1} with $p_-\in({\omega}/(\omega+\eta),\fz)$,
where $\omega$ is as in \eqref{eq-doub} and $\eta$ as in
Definition \ref{expati}. Further assume that $Y(\cx)$
satisfies Assumption \ref{assump2} with the same $p_-$
as in Assumption \ref{assump1},
$\tz_0\in({\omega}/(\omega+\eta),\underline{p})$,
and $p_0\in(\tz_0,\fz)$, where $\underline{p}$ is as in \eqref{2.1y}.
Let $q\in(\max\{p_0,1\},\fz]$, $d\in(0,\tz_0]$,
and $\bz$, $\gz\in(\omega(1/{\tz_0}-1),\eta)$.
The \emph{atomic Hardy space} $H_{\mathrm{atom}}^{Y,q,d}(\cx)$
is defined to be the set of all $f\in(\icgg)'$ satisfying that there
exist a sequence $\{\lz_j\}_{j\in\nn}$ of non-negative numbers and
a sequence $\{a_j\}_{j\in\nn}$ of $(Y(\cx),q)$-atoms supported,
respectively, in balls $\{B_j \}_{j\in\nn}$ of $\cx$ such that
\begin{align*}
f&=\sum_{j\in\nn}\lz_j a_j
\end{align*}
in $(\icgg)'$, and
$$\lf\|\lf\{\sum_{j\in\nn}
\lf[\frac{\lz_j}{\|\ch1_{B_j}\|_{Y(\cx)}}\r]^{d}\ch1_{B_j}\r\}
^{\frac{1}{d}}\r\|_{Y(\cx)}<\fz.$$

Moreover, for any $f\in H_{\mathrm{atom}}^{Y,q,d}(\cx)$, let
$$\|f\|_{H_{\mathrm{atom}}^{Y,q,d}(\cx)}:=\inf\lf\{\lf\|\lf\{\sum_{j\in\nn}
\lf[\frac{\lz_j}{\|\ch1_{B_j}\|_{Y(\cx)}}\r]^{d}\ch1_{B_j}\r\}
^{\frac{1}{d}}\r\|_{Y(\cx)}\r\},$$
where the infimum is taken over all decompositions of $f$ as above.
\end{definition}

Combining Propositions \ref{atre} and \ref{atde},
we obtain the following conclusion.

\begin{theorem}\label{atthm}
Let $Y(\cx)$ be a ball quasi-Banach function space on $\cx$ satisfying Assumption \ref{assump1} with $p_-\in({\omega}/(\omega+\eta),\fz)$,
where $\omega$ is as in \eqref{eq-doub} and $\eta$ as in
Definition \ref{expati}. Further assume that $Y(\cx)$
satisfies Assumption \ref{assump2} with the same $p_-$
as in Assumption \ref{assump1},
$\tz_0\in({\omega}/(\omega+\eta),\underline{p})$,
and $p_0\in(\tz_0,\fz)$, where $\underline{p}$ is as in \eqref{2.1y}.
Let $q\in(\max\{p_0,1\},\fz]$, $d\in(0,\tz_0]$,
and $\bz$, $\gz\in(\omega(1/{\tz_0}-1),\eta)$.
Then
$$\lf[H_Y^*(\cx)\cap(\icgg)'\r]
=\lf[H_{\mathrm{atom}}^{Y,q,d}(\cx)\cap(\icgg)'\r]$$
with equivalent quasi-norms.
\end{theorem}

As a corollary of Theorems \ref{maxchprop} and \ref{atthm},
the following conclusion shows that the atomic Hardy space
in Definition \ref{atomhy} is independent of the choices
of $(\icgg)'$.

\begin{corollary}\label{atomcor}
Let $Y(\cx)$, $q$, $d$, $\eta$, $\oz$, and $\tz_0$
be as in Theorem \ref{atthm}.
Then $H_{\mathrm{atom}}^{Y,q,d}(\cx)$ is independent of the
choices of $(\icgg)'$ whenever $\bz$, $\gz\in(\omega(1/{\tz_0}-1),\eta)$.
\end{corollary}

\begin{remark}\label{atomre}
By the definitions of $Y^{1/\tz_0}(\cx)$ and
$(Y^{1/\tz_0})'(\cx)$, we know that \eqref{5.14.y1}
is equivalent to
\begin{align*}
\lf\|\cm(f)\r\|_{[(Y^{1/{\tz_0}})']^{1/(p_0/\tz_0)'}(\cx)}
\ls\|f\|_{[(Y^{1/{\tz_0}})']^{1/(p_0/\tz_0)'}(\cx)},
\end{align*}
where the implicit positive constant is independent of $f$.
Next, we apply Theorem \ref{atthm} to several concrete
examples of ball quasi-Banach function spaces on $\cx$,
namely, classical Lebesgue spaces, Lorentz spaces,
weighted Lebesgue spaces, Orlicz spaces,
and variable Lebesgue spaces.
\begin{enumerate}
\item[{\rm (i)}]
Let $p\in({\omega}/(\omega+\eta),1]$
and $\bz$, $\gz\in(\omega(1/p-1),\eta)$ with $\omega$
as in \eqref{eq-doub} and $\eta$ as in Definition \ref{expati}.
Assume that $Y(\cx):=L^{p}(\cx)$.
Choose $p_-\in(\max\{{\omega}/(\omega+\bz),{\omega}/(\omega+\gz)\},p]$,
$\tz_0\in(\max\{{\omega}/(\omega+\bz),{\omega}/(\omega+\gz)\},p_-)$,
and $p_0\in(p,\fz)$.
By this and an argument similar to that used in
\cite[Remark 2.7(a)]{wyy}, we know that
$$[(Y^{1/{\tz_0}})']^{1/(p_0/\tz_0)'}(\cx)
=L^{(p/\tz_0)'/(p_0/\tz_0)'}(\cx)$$
and $(p/\tz_0)'/(p_0/\tz_0)'>1$.
From this, Remark \ref{qbsdefrem}(i), \cite[(3.6)]{cw77},
and \cite[Theorem 1.2]{gly09},
we deduce that $L^{p}(\cx)$ satisfies all the assumptions
of Theorem \ref{atthm}. In this case,
a different variant of Theorem \ref{atthm} was obtained
in \cite[Theorem 4.2]{hhllyy} which uses the same atoms,
but equips a different quasi-norm.

\item[{\rm (ii)}]
Let $r\in(0,\fz)$, $p\in({\omega}/(\omega+\eta),1]$,
and $\bz$, $\gz\in(\omega(1/p-1),\eta)$ with $\omega$
as in \eqref{eq-doub} and $\eta$ as in Definition \ref{expati}.
Assume that $Y(\cx):=L^{p,r}(\cx)$ is defined as in
Remark \ref{qbsdefrem}(ii).
Choose $p_-\in(\max\{{\omega}/(\omega+\bz),{\omega}/(\omega+\gz)\},p]$,
$\tz_0\in(\max\{{\omega}/(\omega+\bz),{\omega}/(\omega+\gz)\},p_-)$,
and $p_0\in(p,\fz)$.
By this and an argument similar to that used in
\cite[Remark 2.7(f)]{wyy}, we know that
$$[(Y^{1/{\tz_0}})']^{1/(p_0/\tz_0)'}(\cx)
=L^{(p/\tz_0)'/(p_0/\tz_0)',(r/\tz_0)'/(p_0/\tz_0)'}(\cx)$$
and $(p/\tz_0)'/(p_0/\tz_0)'>1$.
From this, Remark \ref{qbsdefrem}(ii),
\cite[Lemma 3.5]{zhy}, and \cite[Theorem 7.3]{zhy},
we deduce that $L^{p,r}(\cx)$ satisfies all the assumptions
of Theorem \ref{atthm}.
Note that the definition of atomic Hardy spaces here
is different from \cite[Definition 4.2]{zhy}. Thus,
Theorem \ref{atthm} is a different variant of
\cite[Theorem 4.4]{zhy} which uses the same atoms,
but equips a different quasi-norm.

\item[{\rm (iii)}]
Let $p\in({\omega}/(\omega+\eta),1]$, $r\in(1,\fz)$,
$w\in A_{r}(\cx)$ satisfy
$$r_w:=\inf\lf\{r\in[1,\fz):\
w\in A_r(\cx)\r\}\in(1,p(\omega+\eta)/{\omega}),$$
and $\bz$, $\gz\in(\omega(r_w/p-1),\eta)$ with $\omega$
as in \eqref{eq-doub} and $\eta$ as in Definition \ref{expati}.
Assume that $Y(\cx):=L^p_w(\cx)$ is defined as in Section \ref{emd}.
Choose $p_-\in(\max\{{\omega}/(\omega+\bz),{\omega}/(\omega+\gz)\},p/r_w]$
and
$\tz_0\in(\max\{{\omega}/(\omega+\bz),{\omega}/(\omega+\gz)\},p_-)$.
By this and \eqref{21.7.14.x1}, we obtain
$w^{1-(p/\tz_0)'}\in A_{(p/\tz_0)'}(\cx)$.
On one hand, choose $p_0\in(p,\fz)$ large enough such that
\begin{align}\label{21.7.11.x1}
w^{1-(p/\tz_0)'}\in A_{(p/\tz_0)'/(p_0/\tz_0)'}(\cx).
\end{align}
On the other hand, by an argument similar to that used in
\cite[Remark 2.7(b)]{wyy}, we conclude that
$$[(Y^{1/{\tz_0}})']^{1/(p_0/\tz_0)'}(\cx)
=L^{(p/\tz_0)'/(p_0/\tz_0)'}_{w^{1-(p/\tz_0)'}}(\cx).$$
From this, Remark \ref{qbsdefrem}(iii), \eqref{21.7.11.x1},
\cite[Theorem 4.11]{fmy19}, and \cite[Theorem 6.5(ii)]{fmy19},
we deduce that $L^p_w(\cx)$ satisfies all the assumptions of
Theorem \ref{atthm}. In this case,
a different variant of Theorem \ref{atthm} was obtained
in \cite[Theorem 5.4]{fmy19} which uses different atoms
and equips a different quasi-norm.

\item[{\rm (iv)}]
Let $\Phi$ be an Orlicz function as in
Remark \ref{qbsdefrem}(iii) with positive lower type
$p_{\Phi}^-$ and positive upper type $p_{\Phi}^+$.
Recall that the \emph{Orlicz space} $L^{\Phi}(\cx)$
is defined to be the set of all $\mu$-measurable functions
$f$ on $\cx$ such that
$$\|f\|_{L^{\Phi}(\cx)}:=\inf\lf\{\lz\in(0,\fz):\
\int_{\cx}\Phi\lf(\frac{|f(x)|}{\lz}\r)\,d\mu(x)\le1\r\}<\infty.$$
Let $p_{\Phi}^-\in({\omega}/(\omega+\eta),1]$, $p_{\Phi}^+=1$,
and $\bz$, $\gz\in(\omega(1/p_{\Phi}^--1),\eta)$ with $\omega$
as in \eqref{eq-doub} and $\eta$ as in Definition \ref{expati}.
Assume that $Y(\cx):=L^{\Phi}(\cx)$.
Choose $p_-\in(\max\{{\omega}/(\omega+\bz),{\omega}/(\omega+\gz)\},p_{\Phi}^-]$,
$\tz_0\in(\max\{{\omega}/(\omega+\bz),{\omega}/(\omega+\gz)\},p_-)$,
and $p_0\in(1,\fz)$.
By this and an argument similar to that used in
\cite[Remark 2.7(g)]{wyy}, we know that
$$[(Y^{1/{\tz_0}})']^{1/(p_0/\tz_0)'}(\cx)
=L^{\widetilde{\Psi}}(\cx)$$
and $p_{\widetilde{\Psi}}^-=(p_{\Phi}^+/\tz_0)'/(p_0/\tz_0)'>1$,
where, for any $t\in[0,\fz)$,
$$\widetilde{\Psi}(t):=
\sup_{r\in[0,\fz)}\lf[rt^{1/(p_0/\tz_0)'}-\Phi\lf(r^{1/s}\r)\r].$$
By this, Remark \ref{qbsdefrem}(iii), \cite[Theorem 4.11]{fmy19},
and \cite[Theorem 6.6]{fmy19}, we know that $L^{\Phi}(\cx)$
satisfies all the assumptions of Theorem \ref{atthm}.
In this case, a different variant of Theorem \ref{atthm} was obtained
in \cite[Theorem 5.4]{fmy19} which uses the same atoms,
but equips a different quasi-norm.

\item[{\rm (v)}]
Let $(\cx,\rho,\mu)$ be an RD-space,
$p(\cdot)\in C^{\log}(\cx)$ satisfy
$\widetilde{p_-}\in(\oz/(\oz+\eta),\fz)$,
and $\bz$, $\gz\in(\omega[1/\widetilde{p_-}-1]_+,\eta)$
with $\omega$ as in \eqref{eq-doub},
$\eta$ as in Definition \ref{expati},
and $[1/\widetilde{p_-}-1]_+:=\max\{1/\widetilde{p_-}-1,0\}$.
Assume that $Y(\cx):=L^{p(\cdot)}(\cx)$.
If $\widetilde{p_-}\in({\omega}/(\omega+\eta),1)$,
we choose $p_-\in(\max\{{\omega}/(\omega+\bz),{\omega}/(\omega+\gz)\},\widetilde{p_-}]$,
$\tz_0\in(\max\{{\omega}/(\omega+\bz),{\omega}/(\omega+\gz)\},p_-)$,
and $p_0\in(\widetilde{p_+},\fz)$;
if $\widetilde{p_-}\in[1,\fz)$,
we choose $p_-\in[1,\widetilde{p_-}]$,
$\tz_0\in(0,1)$, and $p_0\in(\widetilde{p_+},\fz)$.
By this and an argument similar to that used in
\cite[Remark 2.7(f)]{wyy}, we know that
$$[(Y^{1/{\tz_0}})']^{1/(p_0/\tz_0)'}(\cx)
=L^{(p(\cdot)/\tz_0)'/(p_0/\tz_0)'}(\cx)$$
and $(p(\cdot)/\tz_0)'/(p_0/\tz_0)'
\ge(\widetilde{p_+}/\tz_0)'/(p_0/\tz_0)'>1$.
From this, Remark \ref{qbsdefrem}(iv),
\cite[Lemma 2.5]{zsy}, and \cite[Theorem 2.7]{zsy},
we deduce that $L^{p(\cdot)}(\cx)$ satisfies all the assumptions
of Theorem \ref{atthm}. In this case, Theorem \ref{atthm}
improves \cite[Theorem 4.3]{zsy} by removing the reverse
doubling assumption of $\mu$.
\end{enumerate}

Let $\vz$ be a growth function as in Remark \ref{qbsdefrem}(iii).
Although the Musielak--Orlicz space $L^{\vz}(\cx)$ is also a
ball quasi-Banach function space
[see Remark \ref{qbsdefrem}(iii) above],
Theorem \ref{atthm} is not applicable to the
Musielak--Orlicz Hardy space $H^{*,\vz}(\cx)$.
Indeed, the space variable $x$ and the growth variable
$t$ in $\vz(x,t)$ are in general not separable,
which makes Assumption \ref{assump2} fail on $L^{\vz}(\cx)$.
Recently, Fu et al. \cite{fmy19} obtained a corresponding conclusion
of Theorem \ref{atthm} in case that $Y(\cx):=L^{\vz}(\cx)$
independently.
\end{remark}

\subsection{Finite atomic characterizations\label{sfinatom}}

In this subsection, we establish finite atomic
characterizations of $H_Y^*(\cx)$. To this end,
we first introduce the notion of finite
atomic Hardy spaces associated with $Y(\cx)$.

\begin{definition}\label{finatom}
Let $Y(\cx)$ be a ball quasi-Banach function space on $\cx$ satisfying Assumption \ref{assump1} with $p_-\in({\omega}/(\omega+\eta),\fz)$,
where $\omega$ is as in \eqref{eq-doub} and $\eta$ as in
Definition \ref{expati}. Further assume that $Y(\cx)$
satisfies Assumption \ref{assump2} with the same $p_-$
as in Assumption \ref{assump1},
$\tz_0\in({\omega}/(\omega+\eta),\underline{p})$,
and $p_0\in(\tz_0,\fz)$, where $\underline{p}$ is as in \eqref{2.1y}.
Let $q\in(\max\{p_0,1\},\fz]$ and $d\in(0,\tz_0]$.
The \emph{finite atomic Hardy space} $H_{\mathrm{fin}}^{Y,q,d}(\cx)$
associated with $Y(\cx)$ is defined to be the set of all finite
linear combinations of $(Y(\cx),q)$-atoms. The quasi-norm
$\|\cdot\|_{H_{\mathrm{fin}}^{Y,q,d}(\cx)}$ in $H_{\mathrm{fin}}^{Y,q,d}(\cx)$
is defined by setting, for any $f\in H_{\mathrm{fin}}^{Y,q,d}(\cx)$,
\begin{align*}
\|f\|_{H_{\mathrm{fin}}^{Y,q,d}(\cx)}&:=\inf\lf\{\lf\|\lf\{\sum_{j=1}^{m}
\lf[\frac{\lz_j}{\|\ch1_{B_j}\|_{Y(\cx)}}\r]^{d}\ch1_{B_j}\r\}
^{\frac{1}{d}}\r\|_{Y(\cx)}:\ m\in\nn,\r.\\
&\qquad\qquad\qquad\qquad
\lf.f=\sum_{j=1}^{m}\lz_ja_j,\ \{\lz_j\}_{j=1}^{m}\subset[0,\infty)\r\},
\end{align*}
where the infimum is taken over all finite linear
combinations of $f$ via $(Y(\cx),q)$-atoms as above.
\end{definition}

Next, we establish the finite atomic characterization of $H_Y^*(\cx)$
(see \cite[Theorem 5.12]{yyy20b} for the corresponding Euclidean case).
In what follows, denote by the \emph{symbol $\uu\cc(\cx)$} the \emph{set
of all uniformly continuous functions on $\cx$}, that is, a function
$f\in\uu\cc(\cx)$ if and only if, for any given $\ez\in(0,\fz)$,
there exists a $\delta\in(0,\fz)$ such that $|f(x)-f(y)|<\ez$
whenever $\rho(x,y)<\delta$.
\begin{theorem}\label{finatomeq}
Let $Y(\cx)$ be a ball quasi-Banach function space on $\cx$ satisfying Assumption \ref{assump1} with $p_-\in({\omega}/(\omega+\eta),\fz)$,
where $\omega$ is as in \eqref{eq-doub} and $\eta$ as in
Definition \ref{expati}. Further assume that $Y(\cx)$
satisfies Assumption \ref{assump2} with the same $p_-$
as in Assumption \ref{assump1},
$\tz_0\in({\omega}/(\omega+\eta),\underline{p})$,
and $p_0\in(\tz_0,\fz)$, where $\underline{p}$ is as in \eqref{2.1y}.
Let $q\in(\max\{p_0,1\},\fz]$ and $d\in(0,\tz_0]$.
\begin{enumerate}
\item[{\rm (i)}]  If $q\in(\max\{p_0,1\},\fz)$,
then $\|\cdot\|_{H_{\mathrm{fin}}^{Y,q,d}(\cx)}$ and $\|\cdot\|_{H_Y^*(\cx)}$
are equivalent quasi-norms on the space $H_{\mathrm{fin}}^{Y,q,d}(\cx)$.

\item[{\rm (ii)}] If $q=\infty$, then
$\|\cdot\|_{H_{\mathrm{fin}}^{Y,\fz,d}(\cx)}$
and $\|\cdot\|_{H_Y^*(\cx)}$ are equivalent quasi-norms on the space
$H_{\mathrm{fin}}^{Y,\fz,d}(\cx)\cap\uu\cc(\cx)$.
\end{enumerate}
\end{theorem}

\begin{proof}
Let all the symbols be as in the present theorem.
From Theorem \ref{atthm}, we deduce that,
for any $q\in(\max\{p_0,1\},\fz]$,
$$H_{\mathrm{fin}}^{Y,q,d}(\cx)\st H_{\mathrm{atom}}^{Y,q,d}(\cx)\st H_Y^*(\cx)$$
and, for any $f\in H_{\mathrm{fin}}^{Y,q,d}(\cx)$,
$$\|f\|_{H_Y^*(\cx)}\ls\|f\|_{H_{\mathrm{atom}}^{Y,q,d}(\cx)}
\ls\|f\|_{H_{\mathrm{fin}}^{Y,q,d}(\cx)}.$$
Thus, to complete the proof of this theorem,
it suffices to show that
$$\|f\|_{H_{\mathrm{fin}}^{Y,q,d}(\cx)}\ls\|f\|_{H_Y^*(\cx)}$$
for any $f\in H_{\mathrm{fin}}^{Y,q,d}(\cx)$ when
$q\in(\max\{p_0,1\},\fz)$, or any
$f\in H_{\mathrm{fin}}^{Y,\fz,d}(\cx)\cap\uu\mathcal{C}(\cx)$.

Assume that $q\in(\max\{p_0,1\},\fz]$. By the homogeneity of both
$\|\cdot\|_{H_Y^*(\cx)}$ and $\|\cdot\|_{H_{\mathrm{fin}}^{Y,q,d}(\cx)}$,
without loss of generality, we may assume that
$f\in H_{\mathrm{fin}}^{Y,q,d}(\cx)$ and $\|f\|_{H_Y^*(\cx)}=1$.
Since $f$ is a finite linear combination of $(Y(\cx),q)$-atoms,
it follows that there exist an $x_1\in\cx$ and a $K\in(0,\fz)$
such that $\supp f\st B(x_1,K)$. In the remainder of this proof,
let $B_0:=B(x_1,AK)$ with $A:=16A_0^4$
and all the natation be as in the proof of Proposition \ref{atde}.
Let $x\notin B_0$ and $\vz\in\icgg$ with
$\|\vz\|_{\cg(x,r,\bz,\gz)}\le 1$ for some $r\in(0,\fz)$.
We consider the following two cases on $r$.

If $r\ge 4A_0^2\rho(x_1,x)/3$, then, by an argument similar
to that used in the proof of \cite[(7.1)]{hhllyy},
we conclude that, for any $y\in B(x,\rho(x_1,x))$,
\begin{align}\label{8.9.x2}
|\langle f,\varphi\rangle|\ls f^\ast(y).
\end{align}
Let $h\in(0,\underline{p})$. From the fact that $\ch1_{B_0}\le[C_{(\mu)}2^{-\oz}
\cm(\ch1_{B(x,\rho(x_1,x))})]^{1/h}$
with $C_{(\mu)}$ and $\oz$ as in \eqref{eq-doub},
Definition \ref{qbs}(ii), and Assumption \ref{assump1},
we deduce that
\begin{align*}
\lf\|\ch1_{B_0}\r\|_{Y(\cx)}
&\le\lf[C_{(\mu)}2^{-\oz}\r]^{1/h}
\lf\|\lf[\cm\lf(\ch1_{B(x,\rho(x_1,x))}\r)\r]^{1/h}\r\|_{Y(\cx)}
\ls\lf\|\ch1_{B(x,\rho(x_1,x))}\r\|_{Y(\cx)}.
\end{align*}
By this and \eqref{8.9.x2}, we conclude that
\begin{align}\label{7.5.y1}
|\langle f,\varphi\rangle|\ls\inf_{y\in B(x,\rho(x,x_1))}f^*(y)
\ls\lf\|f^*\r\|_{Y(\cx)}
\lf\|\ch1_{B(x,\rho(x_1,x))}\r\|_{Y(\cx)}^{-1}
\ls\lf\|\ch1_{B_0}\r\|_{Y(\cx)}^{-1}.
\end{align}
This is the desired estimate.

If $r<4A_0^2\rho(x_1,x)/3$, then, by an argument similar
to that used in the proof of \cite[(7.1)]{hhllyy},
we conclude that, for any $y\in B(x_1,\rho(x_1,x))$,
$|\langle f,\varphi\rangle|\ls f^\ast(y)$.
From this, we deduce that
\begin{align}\label{8.9.x1}
|\langle f,\varphi\rangle|
\ls\inf_{y\in B(x_1,\rho(x,x_1))}f^*(y)
\ls\inf_{y\in B_0}f^*(y)
\ls\lf\|f^*\r\|_{Y(\cx)}
\lf\|\ch1_{B_0}\r\|_{Y(\cx)}^{-1}
\sim\lf\|\ch1_{B_0}\r\|_{Y(\cx)}^{-1}.
\end{align}
This is also the desired estimate.

Combining \eqref{7.5.y1}, \eqref{8.9.x1},
the arbitrariness of $\vz$, and
the fact that $f^\star\sim f^*$,
we know that,
for any $x\in B_0^{\com}$,
\begin{align*}
f^\star(x)&\le Cf^*(x)\le c_0
\lf\|\ch1_{B_0}\r\|_{Y(\cx)}^{-1},
\end{align*}
where $c_0$ and $C$ are positive constants independent of $f$ and $x$.
Denote by $j'$ the largest integer $j$ such that
$2^j<c_0\|\ch1_{B_0}\|_{Y(\cx)}^{-1}$.
Then, for any $j\in(j',\fz)\cap\zz$,
\begin{align}\label{7.6.x2}
\Omega_j\subset B_0.
\end{align}
Let
\begin{align}\label{3.16.y3}
h:&=\sum_{j=-\fz}^{j'}\sum_{k\in I_j}\lz_{j,k}\ajk\quad\mbox{and}
\quad \ell:=\sum_{j=j'+1}^{\fz}\sum_{k\in I_j}\lz_{j,k}\ajk,
\end{align}
where the series converge both in $(\icgg)'$ and almost everywhere in $\cx$.
By \eqref{7.6.x2}, \eqref{3.16.y3}, and Lemma \ref{helem1}(ii),
we know that $\supp \ell\subset B_0$,
and hence $\supp h\subset B_0$.
Moreover, from \eqref{7.22.x1} and Lemma \ref{helem1}(v),
we deduce that
\begin{align*}
\|h\|_{L^{\fz}(\cx)}\ls\sum_{j=-\fz}^{j'}
\lf\|\sum_{k\in I_j}\lz_{j,k}\ajk\r\|_{L^{\fz}(\cx)}
\ls\sum_{j=-\fz}^{j'}2^j\ls 2^{j'}
\ls\lf\|\ch1_{B_0}\r\|_{Y(\cx)}^{-1}.
\end{align*}
Notice that $f\in L^{\widetilde{q}}(\cx)$ with
$\widetilde{q}:=q$ if $q<\fz$,
and $\widetilde{q}:=2$ if $q=\fz$.
From this and \cite[Theorem 3.4]{hhllyy},
we deduce that $f^*\in L^{\widetilde{q}}(\cx)$,
which, combined with the H\"older inequality,
\eqref{7.6.x2}, \eqref{7.22.x1}, and an argument
similar to that used in the proof of \eqref{7.5.x2},
further implies that
\begin{align}\label{3.22.y1}
\lf\|\sum_{j=j'+1}^{\fz}\sum_{k\in I_j}
\lf|\lz_{j,k}\ajk\r|\r\|_{L^{1}(\cx)}
&\ls\lf\|\sum_{j=j'+1}^{\fz}\sum_{k\in I_j}
\lf|\lz_{j,k}\ajk\r|\r\|_{L^{\widetilde{q}}(\cx)}
\ls\lf\|\sum_{j=j'+1}^{\fz}\sum_{k\in I_j}
2^j\ch1_{B(x_{j,k},Ar_{j,k})}\r\|_{L^{\widetilde{q}}(\cx)}\noz\\
&\ls\lf\|\sum_{j=j'+1}^{\fz}2^j
\ch1_{\Omega_{j}}\r\|_{L^{\widetilde{q}}(\cx)}
\ls\lf\|f^*\r\|_{L^{\widetilde{q}}(\cx)}<\fz.
\end{align}
By this, the Lebesgue dominated convergence theorem,
and the cancellation of $\ajk$, we conclude that
\begin{align}\label{7.7.x1}
\int_{\cx}\ell(x)\,d\mu(x)
=\sum_{j=j'+1}^{\fz}\sum_{k\in I_j}\int_{\cx}
\lz_{j,k}(x)\ajk(x)\,d\mu(x)=0,
\end{align}
which further implies that
$$\int_{\cx}h(x)\,d\mu(x)=0.$$
Thus, $h$ is a harmless constant multiple of a $(Y(\cx),\fz)$-atom
supported in $B_0$.
For $\ell$, we consider two cases on $q$.

{\it Case 1)} $q\in(\max\{p_0,1\},\fz)$.
In this case, $f\in H_{\mathrm{fin}}^{Y,q,d}(\cx)$.
For any $i\in\nn$, let
$$F_i:=\lf\{(j,k)\in\zz\times\nn:\ j>j',\,k\in I_j,\,|j|+k\le i\r\},$$
and $\ell_i:=\sum_{(j,k)\in F_i}\lz_{j,k}\ajk$.
By \eqref{3.22.y1}, we know that $\ell$ converges in $L^{q}(\cx)$,
which further implies that there exists a positive integer $i_0$
such that
\begin{align}\label{8.2.x1}
\lf\|\ell-\ell_{i_0}\r\|_{L^{q}(\cx)}
&\le\frac{[\mu(B_0)]^{1/q}}
{\|\ch1_{B_0}\|_{Y(\cx)}}.
\end{align}
Moreover, from \eqref{7.7.x1}, we deduce that
\begin{align*}
\int_{\cx}\lf[\ell(x)-\ell_{i_0}(x)\r]\,d\mu(x)
=\int_{\cx}\ell(x)\,d\mu(x)-\int_{\cx}\ell_{i_0}(x)\,d\mu(x)=0.
\end{align*}
By this, the fact that $\supp(\ell-\ell_{i_0})\subset B_0$,
and \eqref{8.2.x1}, we conclude that $\ell-\ell_{i_0}$
is a $(Y(\cx),q)$-atom supported in $B_0$.
Therefore,
$$f=h+\ell=h+(\ell-\ell_{i_0})+\ell_{i_0}$$
is a finite decomposition of $f$ in terms of $(Y(\cx),q)$-atoms.
From this, Remark \ref{r-ar}, \eqref{7.22.x1},
and \eqref{7.5.x2}, we deduce that
there exists a $\nu\in(0,1)$ such that
\begin{align*}
\|f\|^{\nu}_{H_{\mathrm{fin}}^{Y,q,d}(\cx)}
&\ls\lf\|\lf\{\lf[\frac{1}
{\|\ch1_{B_0}\|_{Y(\cx)}}\r]^{d}\ch1_{B_0}
+\sum_{(j,k)\in F_{i_0}}\lf[\frac{\lz_{j,k}}
{\|\ch1_{B(x_{j,k},Ar_{j,k})}\|_{Y(\cx)}}\r]^{d}
\ch1_{B(x_{j,k},Ar_{j,k})}\r\}^{\frac1{d}}\r\|^{\nu}_{{Y(\cx)}}\noz\\
&\ls1+\lf\|\lf\{\sum_{j\in\zz}\sum_{k\in I_j}2^{jd}
\ch1_{B(x_{j,k},Ar_{j,k})}\r\}^{\frac1{d}}\r\|^{\nu}_{{Y(\cx)}}
\ls1+\|f\|^{\nu}_{H_Y^*(\cx)}\ls\|f\|^{\nu}_{H_Y^*(\cx)},\noz
\end{align*}
which completes the proof of (i).

{\it Case 2)} $q=\fz$. In this case,
$f\in H_{\mathrm{fin}}^{Y,\fz,d}(\cx)\cap\uu\cc(\cx)$.
Assume that $\|f\|_{H_Y^*(\cx)}=1$. Since $f$ is bounded,
it follows from the fact that $f^\star\sim f^*$,
\cite[Proposition 3.9]{gly08}, and the boundedness of $\cm$
on $L^{\fz}(\cx)$ (see, for instance, \cite[(3.6)]{cw77}) that
\begin{align*}
\lf\|f^\star\r\|_{L^{\fz}(\cx)}&\sim\lf\|f^*\r\|_{L^{\fz}(\cx)}
\ls\lf\|\cm(f)\r\|_{L^{\fz}(\cx)}\ls\|f\|_{L^{\fz}(\cx)}<\fz.
\end{align*}
Thus, there exists a positive
integer $j''>j'$ such that, for any $j\in\{j''+1,\,j''+2,\,\ldots\}$,
$\mu(\Omega_j)=0$. Consequently, in this case, we have
$$\ell=\sum_{j=j'+1}^{j''}\sum_{k\in I_j}\lz_{j,k}\ajk$$
almost everywhere in $\cx$. Let $\epsilon\in(0,\fz)$. By the fact that $f\in\uu\cc(\cx)$,
we know that there exists a $\delta\in(0,\fz)$ such that,
for any $x$, $y\in\cx$ with $\rho(x,y)<\delta$,
$|f(x)-f(y)|<\epsilon$. Write
$\ell=\ell_1^\epsilon+\ell_2^\epsilon$
with $\ell_1^\epsilon:=\sum_{(j,k)\in J_1}\lz_{j,k}\ajk$
and $\ell_2^\epsilon:=\sum_{(j,k)\in J_2}\lz_{j,k}\ajk$,
where
$$J_1:=\lf\{(j,k)\in\zz\times\nn:\ j'<j\le j'',
\,k\in I_j,\,12A_0^3r_{j,k}\ge\delta\r\}$$
and
$$J_2:=\lf\{(j,k)\in\zz\times\nn:\ j'<j\le j'',
\,k\in I_j,\,12A_0^3r_{j,k}<\delta\r\}.$$
On one hand, by \eqref{7.6.x2}, we know that, for any $j\in(j',j'']\cap\zz$,
$\Omega_j$ is bounded, which, combined with (i) and Lemma \ref{helem1}(vi),
further implies that $\ell_1^\epsilon$ is a finite linear
combination of $(Y(\cx),\fz)$-atoms, and
$\|\ell_1^\epsilon\|_{H_{\mathrm{fin}}^{Y,q,d}(\cx)}\ls\|f\|_{H_Y^*(\cx)}$.
On the other hand, it is easy to see
$\supp \ell_2^\epsilon\st B_0$
and $\int_{\cx}\ell_2^\epsilon(x)\,d\mu(x)=0$.
By an argument similar to that used in the proof of
\cite[Theorem 5.7(ii)]{zhy}, we obtain
$\|\ell_2^\epsilon\|_{L^{\fz}(\cx)}\ls\epsilon$.
Thus, $\ell_2^\epsilon$ is a small constant multiple
of a $(Y(\cx),\fz)$-atom.
Therefore, $f=h+\ell_1^\epsilon+\ell_2^\epsilon$
is a finite decomposition of $f$ in terms of $(Y(\cx),\fz)$-atoms.
Furthermore, using this fact and repeating the proof of (i),
we obtain
$$\|f\|_{H_{\mathrm{fin}}^{Y,\fz,d}(\cx)}\ls\|f\|_{H_Y^*(\cx)}.$$
This finishes the proof of (ii) and hence of Theorem \ref{finatomeq}.
\end{proof}

\begin{remark}\label{finatomre}
We now give several applications of Theorem \ref{finatomeq} as follows.
\begin{enumerate}
\item[{\rm (i)}]
Let $p\in({\omega}/(\omega+\eta),1]$ with $\omega$ as
in \eqref{eq-doub} and $\eta$ as in Definition \ref{expati}.
If $Y(\cx):=L^{p}(\cx)$, then, by Remark \ref{atomre}(i),
we know that $L^{p}(\cx)$ satisfies all the assumptions
of Theorem \ref{finatomeq}. In this case, Theorem \ref{finatomeq}
is a different variant of \cite[Theorem 7.1]{hhllyy}
which uses the same atoms, but equips a different quasi-norm.

\item[{\rm (ii)}]
Let $p\in({\omega}/(\omega+\eta),1]$ with $\omega$ as in
\eqref{eq-doub} and $\eta$ as in Definition \ref{expati},
and $r\in(0,\fz)$. If $Y(\cx):=L^{p,r}(\cx)$,
then, by Remark \ref{atomre}(ii),
we know that $L^{p,r}(\cx)$ satisfies all the assumptions
of Theorem \ref{finatomeq}.
In this case, Theorem \ref{finatomeq} is a different variant of
\cite[Theorem 5.7]{zhy} which uses the same atoms, but equips
a different quasi-norm.

\item[{\rm (iii)}]
Let $p\in({\omega}/(\omega+\eta),1]$, $r\in(1,\fz)$,
$w\in A_{r}(\cx)$ satisfy
$$r_w:=\inf\lf\{r\in[1,\fz):\
\vz\in A_r(\cx)\r\}\in(1,p(\omega+\eta)/{\omega})$$
with $\omega$ as in \eqref{eq-doub} and
$\eta$ as in Definition \ref{expati}.
If $Y(\cx):=L^p_w(\cx)$, then, by Remark \ref{atomre}(iii),
we know that $L^p_w(\cx)$ satisfies all the assumptions
of Theorem \ref{finatomeq}. In this case,
Theorem \ref{finatomeq} is a different variant of
\cite[Theorem 7.2]{fmy19} which uses different
atoms and equips a different quasi-norm.

\item[{\rm (iv)}]
Let $\Phi$ be an Orlicz function, $p_{\Phi}^+=1$,
and $p_{\Phi}^-\in({\omega}/(\omega+\eta),1]$
with $\omega$ as in \eqref{eq-doub} and
$\eta$ as in Definition \ref{expati}.
If $Y(\cx):=L^{\Phi}(\cx)$,
then, by Remark \ref{atomre}(iv),
we know that $L^{\Phi}(\cx)$ satisfies all the assumptions
of Theorem \ref{finatomeq}. In this case,
Theorem \ref{finatomeq} is a different variant of
\cite[Theorem 7.2]{fmy19} which uses the same atoms,
but equips a different quasi-norm.

\item[{\rm (v)}]
Let $(\cx,\rho,\mu)$ be an RD-space,
and $p(\cdot)\in C^{\log}(\cx)$ satisfy
$\widetilde{p_-}\in(\oz/(\oz+\eta),\fz)$
with $\omega$ as in \eqref{eq-doub} and
$\eta$ as in Definition \ref{expati}.
If $Y(\cx):=L^{p(\cdot)}(\cx)$,
then, by Remark \ref{atomre}(v),
we know that $L^{p(\cdot)}(\cx)$ satisfies all the assumptions
of Theorem \ref{finatomeq}.
In this case, Theorem \ref{finatomeq}
improves \cite[Theorem 4.24]{zsy} by removing the reverse
doubling assumption of $\mu$.
\end{enumerate}

Let $\vz$ be a growth function as in Remark \ref{qbsdefrem}(iii).
As was mentioned in Remark \ref{atomre},
Theorem \ref{finatomeq} is not applicable to $H^{*,\vz}(\cx)$.
Fu et al. \cite{fmy19} obtained a similar
conclusion of Theorem \ref{finatomeq} in case that
$Y(\cx):=L^{\vz}(\cx)$ independently.
\end{remark}

\section{Molecular characterizations of $H_Y^*(\cx)$\label{s-mole}}
In this section, we establish the molecular characterization
of $H_Y^*(\cx)$. To this end, we first introduce the notion of molecules.
\begin{definition}\label{mold}
Let $Y(\cx)$ be a ball quasi-Banach function space on $\cx$ satisfying Assumption \ref{assump1} with $p_-\in({\omega}/(\omega+\eta),\fz)$,
where $\omega$ is as in \eqref{eq-doub} and $\eta$ as in
Definition \ref{expati}. Further assume that $Y(\cx)$
satisfies Assumption \ref{assump2} with the same $p_-$
as in Assumption \ref{assump1},
$\tz_0\in({\omega}/(\omega+\eta),\underline{p})$,
and $p_0\in(\tz_0,\fz)$, where $\underline{p}$ is as in \eqref{2.1y}.
Let $q\in (\max\{p_0,1\},\fz]$, $\epsilon\in (0,\fz)$,
and $\dz\in(0,1)$ be as in Lemma \ref{daydic}.
\begin{enumerate}
\item[{\rm (i)}] A $\mu$-measurable function $m$ on $\cx$ is called
a $(Y(\cx),q,\epsilon)$-\emph{molecule centered at a ball
$B\subset\cx$} if
\begin{enumerate}
\item[{$\rm (i)_1$}] for any $j\in\zz_+$,
$$\|m\|_{L^q(U_j(B))}\le\dz^{j\ez}
\lf[\mu\lf(\dz^{-j}B\r)\r]^{\frac{1}{q}}
\lf\|\ch1_B\r\|_{Y(\cx)}^{-1};$$
here and thereafter, $U_0(B):=B$ and, for any $j\in\nn$,
$U_j(B):=(\dz^{-j}B) \setminus (\dz^{-j+1} B)$;

\item[{$\rm (i)_2$}] $\int_{\cx}m(x)\,d\mu(x)=0$.
\end{enumerate}
\item[{$\rm (ii)$}] Let $d\in(0,\tz_0]$ and $\bz$,
$\gz\in(\omega(1/{\tz_0}-1),\eta)$.
The \emph{molecular Hardy space}
$H_{\mathrm{mol}}^{Y,q,d,\ez}(\cx)$ is defined to be the set of
all $f\in(\icgg)'$ satisfying that there exist a sequence
$\{\lz_j\}_{j\in\nn}\subset[0,\fz)$ and a sequence $\{m_j\}_{j\in\nn}$
of $(Y(\cx),q,\epsilon)$-molecules centered, respectively,
at balls $\{B_j \}_{j\in\nn}$ of $\cx$ such that
$f=\sum_{j\in\nn}\lz_j m_j$ in $(\icgg)'$, and
$$\lf\|\lf\{\sum_{j\in\nn}
\lf[\frac{\lz_j}{\|\ch1_{B_j}\|_{Y(\cx)}}\r]^{d}\ch1_{B_j}\r\}
^{\frac{1}{d}}\r\|_{Y(\cx)}<\fz.$$
Moreover, for any $f\in H_{\mathrm{mol}}^{Y,q,d,\ez}(\cx)$,
let
$$\|f\|_{H_{\mathrm{mol}}^{Y,q,d,\ez}(\cx)}
:=\inf\lf\{\lf\|\lf\{\sum_{j\in\nn}
\lf[\frac{\lz_j}{\|\ch1_{B_j}\|_{Y(\cx)}}\r]^{d}\ch1_{B_j}\r\}
^{\frac{1}{d}}\r\|_{Y(\cx)}\r\},$$
where the infimum is taken over all decompositions of $f$ as above.
\end{enumerate}
\end{definition}

The following lemma shows that a $(Y(\cx),q,\epsilon)$-molecule
can be decomposed into a sum of a sequence of $(Y(\cx),q)$-atoms.
\begin{lemma}\label{moldec}
Let $Y(\cx)$ be a ball quasi-Banach function space on $\cx$ satisfying Assumption \ref{assump1} with $p_-\in({\omega}/(\omega+\eta),\fz)$,
where $\omega$ is as in \eqref{eq-doub} and $\eta$ as in
Definition \ref{expati}. Further assume that $Y(\cx)$
satisfies Assumption \ref{assump2} with the same $p_-$
as in Assumption \ref{assump1},
$\tz_0\in({\omega}/(\omega+\eta),\underline{p})$,
and $p_0\in(\tz_0,\fz)$, where $\underline{p}$ is as in \eqref{2.1y}.
Let $q\in (\max\{p_0,1\},\fz]$, $\dz\in(0,1)$ be as in Lemma \ref{daydic},
$\oz$ as in \eqref{eq-doub}, and $\ez\in (\oz,\fz)$.
If $m$ is a $(Y(\cx),q,\ez)$-molecule centered at a ball $B:=B(x_0,r_0)$
for some $x_0\in\cx$ and $r_0\in(0,\fz)$, then it holds true that
$$m=\sum_{j=0}^{\fz}u_jb_j+\sum_{j=1}^{\fz}\sum_{k=j}^{\fz}v_{j,k}c_{j,k}$$
in $(\icgg)'$ with $\bz$, $\gz\in(\omega(1/{\tz_0}-1),\eta)$,
where, for any $j\in\zz_+$ and $k\in\zz\cap[j,\fz)$,
$$u_j:=\frac{2\dz^{\ez j}\|\ch1_{\dz^{-j}B}\|_{Y(\cx)}}
{\|\ch1_B\|_{Y(\cx)}}\quad\mbox{and}\quad
v_{j,k}:=\frac{2C_{(\mu)}\dz^{(\ez-\oz)k+\oz j-\oz}
\|\ch1_{\dz^{-j}B}\|_{Y(\cx)}}{\|\ch1_B\|_{Y(\cx)}}$$
with $C_{(\mu)}$ as in \eqref{eq-doub}, and
$b_j$ and $c_{j,k}$ are both $(Y(\cx),q)$-atoms supported in $\dz^{-j}B$.
\end{lemma}
\begin{proof}
Let $m$ be a $(Y(\cx),q,\ez)$-molecule centered at some ball $B\st\cx$
with $Y(\cx)$, $q$, and $\ez$ as in the present lemma.
For any $j\in\zz_+$, let
$$m_j:=\frac{\ch1_{\dz^{-j}B}}{\mu(\dz^{-j}B)}
\int_{\cx}m(y)\ch1_{U_j(B)}(y)\,d\mu(y)$$
and $M_j:=m\ch1_{U_j(B)}-m_j$. Obviously,
\begin{align}\label{5.14.x1}
m&=\sum_{j=0}^{\fz}M_j+\sum_{j=0}^{\fz}m_j
\end{align}
pointwisely.

Now, we consider the first sum of \eqref{5.14.x1}.
Fix $j\in\zz_+$. We claim that $M_j$ is a multiple of a $(Y(\cx),q)$-atom.
Indeed, it is easy to see that
\begin{align}\label{3.27.x1}
\supp M_j\subset\dz^{-j}B\quad\mbox{and}\quad\int_{\cx}M_j(x)\,d\mu(x)=0.
\end{align}
Moreover,
by the Minkowski and the H\"older inequalities,
we obtain
\begin{align}\label{5.12.x1}
\lf\|M_j\r\|_{L^{q}(\cx)}&\le\lf\|m\ch1_{U_j(B)}\r\|_{L^{q}(\cx)}
+\lf[\mu\lf(\dz^{-j}B\r)\r]^{-\frac{1}{q'}}
\lf\|m\ch1_{U_j(B)}\r\|_{L^{1}(\cx)}\noz\\
&\le2\lf\|m\ch1_{U_j(B)}\r\|_{L^{q}(\cx)}
\le2\dz^{\ez j}\lf[\mu\lf(\dz^{-j}B\r)\r]^{\frac{1}{q}}
\lf\|\ch1_B\r\|_{Y(\cx)}^{-1}\noz\\
&=\frac{2\dz^{\ez j}\|\ch1_{\dz^{-j}B}\|_{Y(\cx)}}
{\lf\|\ch1_B\r\|_{Y(\cx)}}
\lf[\mu\lf(\dz^{-j}B\r)\r]^{\frac{1}{q}}
\lf\|\ch1_{\dz^{-j}B}\r\|_{Y(\cx)}^{-1}.
\end{align}
Now, for any $j\in\zz_+$, let
$$u_j:=\frac{2\dz^{\ez j}\|\ch1_{\dz^{-j}B}\|_{Y(\cx)}}
{\lf\|\ch1_B\r\|_{Y(\cx)}}
\quad\mbox{and}\quad b_j:=\frac{M_j}{u_j}.$$
Then, by \eqref{3.27.x1} and \eqref{5.12.x1}, we know
that $b_j$ is a $(Y(\cx),q)$-atom supported in $\dz^{-j}B$.
From the Minkowski inequality, \eqref{5.12.x1}, \eqref{eq-doub},
and the fact that $\ez>\oz$, we deduce that
\begin{align*}
\sum_{j=0}^{\fz}\lf\|M_j\r\|_{L^{q}(\cx)}
&\ls\sum_{j=0}^{\fz}\dz^{\ez j}
\lf[\mu\lf(\dz^{-j}B\r)\r]^{\frac{1}{q}}
\lf\|\ch1_B\r\|_{Y(\cx)}^{-1}
\ls\sum_{j=0}^{\fz}\dz^{(\ez-\oz/q)j}[\mu(B)]^{\frac{1}{q}}
\lf\|\ch1_B\r\|_{Y(\cx)}^{-1}\\
&\ls[\mu(B)]^{\frac{1}{q}}\lf\|\ch1_B\r\|_{Y(\cx)}^{-1}<\fz.
\end{align*}
Thus, $\sum_{j=0}^{\fz}M_j$ converges in $L^{q}(\cx)$ and
hence in $(\icgg)'$ with $\bz$ and $\gz$ as in the present lemma.

Next, we consider the second sum of \eqref{5.14.x1}.
For any $j\in\zz_+$, let
$$\ch1_j:=\frac{\ch1_{\dz^{-j}B}}{\mu(\dz^{-j}B)},\quad
\widetilde{m}_j:=\int_{\cx}m(y)\ch1_{U_j(B)}(y)\,d\mu(y),\quad
\mbox{and}\quad N_j:=\sum_{k=j}^{\fz}\widetilde{m}_k.$$
Then, by the cancellation of $m$, we obtain
\begin{align*}
N_0&=\sum_{k=0}^{\fz}\widetilde{m}_k=\int_{\cx}m(y)\,d\mu(y)=0,
\end{align*}
which further implies that
\begin{align}\label{5.14.x2}
\sum_{j=0}^{\fz}m_j&=\sum_{j=0}^{\fz}\ch1_j\widetilde{m}_j
=\sum_{j=0}^{\fz}\ch1_j\lf(N_j-N_{j+1}\r)
=\sum_{j=1}^{\fz}\lf(\ch1_j-\ch1_{j-1}\r)N_j\noz\\
&=\sum_{j=1}^{\fz}\sum_{k=j}^{\fz}
\lf(\ch1_j-\ch1_{j-1}\r)\widetilde{m}_k
=:\sum_{j=1}^{\fz}\sum_{k=j}^{\fz}d_{j,k}.
\end{align}
Fix $j\in\nn$ and $k\in[j,\fz)\cap\nn$. We claim that
$d_{j,k}$ is a multiple of a $(Y(\cx),q)$-atom.
Indeed, it is easy to see that
\begin{align}\label{3.27.x2}
\supp d_{j,k}\subset\dz^{-j}B\quad\mbox{and} \quad\int_{\cx}d_{j,k}(x)\,d\mu(x)=0.
\end{align}
Moreover, by the Minkowski inequality, the H\"older inequality,
$k\ge j$, and \eqref{eq-doub}, we know that
\begin{align}\label{5.12.x2}
\lf\|d_{j,k}\r\|_{L^{q}(\cx)}
&\le\frac{2}{\lf[\mu(\dz^{-j+1}B)\r]^{1/q'}}
\lf|\widetilde{m}_k\r|\le\frac{2}{\lf[\mu(\dz^{-j+1}B)\r]^{1/q'}}
\lf\|m\ch1_{U_k(B)}\r\|_{L^{1}(\cx)}\noz\\
&\le\frac{2\lf[\mu(U_k(B))\r]^{1/q'}}{\lf[\mu(\dz^{-j+1}B)\r]^{1/q'}}
\lf\|m\ch1_{U_k(B)}\r\|_{L^{q}(\cx)}
\le\frac{2\dz^{\ez k}\mu(\dz^{-k}B)}{\lf[\mu(\dz^{-j+1}B)\r]^{1/q'}}
\lf\|\ch1_B\r\|_{Y(\cx)}^{-1}\noz\\
&\le2C_{(\mu)}\dz^{(\ez-\oz)k+\oz j-\oz}\lf[\mu\lf(\dz^{-j+1}B\r)\r]^{1/q}
\lf\|\ch1_B\r\|_{Y(\cx)}^{-1}\noz\\
&\le\frac{2C_{(\mu)}\dz^{(\ez-\oz)k+\oz j-\oz}
\|\ch1_{\dz^{-j}B}\|_{Y(\cx)}}
{\|\ch1_B\|_{Y(\cx)}}\lf[\mu\lf(\dz^{-j}B\r)\r]^{1/q}
\lf\|\ch1_{\dz^{-j}B}\r\|_{Y(\cx)}^{-1},
\end{align}
where $C_{(\mu)}$ is as in \eqref{eq-doub}.
Thus, we find that $d_{j,k}$ is a multiple of a $(Y(\cx),q)$-atom.
Now, for any $j\in\nn$ and $k\in[j,\fz)\cap\nn$, define
$$v_{j,k}:=\frac{2C_{(\mu)}\dz^{(\ez-\oz)k+\oz j-\oz}
\|\ch1_{\dz^{-j}B}\|_{Y(\cx)}}{\lf\|\ch1_B\r\|_{Y(\cx)}}
\quad\mbox{and}\quad c_{j,k}:=\frac{d_{j,k}}{v_{j,k}}.$$
Then, by \eqref{3.27.x2} and \eqref{5.12.x2}, we know that
$c_{j,k}$ is a $(Y(\cx),q)$-atom supported in $\dz^{-j}B$.
From \eqref{5.12.x2}, \eqref{eq-doub},
and the fact that $\ez>\oz$,
we deduce that
\begin{align*}
\sum_{j=1}^{\fz}\sum_{k=j}^{\fz}\lf\|d_{j,k}\r\|_{L^{q}(\cx)}
&\ls\sum_{j=1}^{\fz}\sum_{k=j}^{\fz}\dz^{(\ez-\oz)k+\oz j-\oz}
\lf[\mu\lf(\dz^{-j+1}B\r)\r]^{1/q}\lf\|\ch1_B\r\|_{Y(\cx)}^{-1}\\
&\ls\sum_{j=1}^{\fz}\dz^{\ez j}\lf[\mu\lf(\dz^{-j+1}B\r)\r]^{1/q}
\lf\|\ch1_B\r\|_{Y(\cx)}^{-1}
\ls\sum_{j=1}^{\fz}\dz^{(\ez-\oz/q)j}\lf[\mu(B)\r]^{1/q}
\lf\|\ch1_B\r\|_{Y(\cx)}^{-1}\\
&\ls\lf[\mu(B)\r]^{1/q}\lf\|\ch1_B\r\|_{Y(\cx)}^{-1}<\fz.
\end{align*}
Thus, $\sum_{j=1}^{\fz}\sum_{k=j}^{\fz}d_{j,k}$
converges in $L^{q}(\cx)$, and hence in $(\icgg)'$.
By the convergences of both $\sum_{j=0}^{\fz}M_j$ and
$\sum_{j=1}^{\fz}\sum_{k=j}d_{j,k}$, \eqref{5.14.x1},
and \eqref{5.14.x2}, we conclude that
\begin{align*}
m&=\sum_{j=0}^{\fz}M_j+\sum_{j=0}^{\fz}m_j
=\sum_{j=0}^{\fz}u_jb_j+\sum_{j=1}^{\fz}
\sum_{k=j}^{\fz}v_{j,k}c_{j,k}
\end{align*}
converges in $(\icgg)'$,
which completes the proof of Lemma \ref{moldec}.
\end{proof}

Now, we establish the molecular characterization
of $H_Y^*(\cx)$ as follows (see \cite[Theorem 3.9]{shyy17}
for the corresponding Euclidean case).

\begin{theorem}\label{mol}
Let $Y(\cx)$ be a ball quasi-Banach function space on $\cx$ satisfying Assumption \ref{assump1} with $p_-\in({\omega}/(\omega+\eta),\fz)$,
where $\omega$ is as in \eqref{eq-doub} and $\eta$ as in
Definition \ref{expati}. Further assume that $Y(\cx)$
satisfies Assumption \ref{assump2} with the same $p_-$
as in Assumption \ref{assump1},
$\tz_0\in({\omega}/(\omega+\eta),\underline{p})$,
and $p_0\in(\tz_0,\fz)$, where $\underline{p}$ is as in \eqref{2.1y}.
Let $q\in (\max\{p_0,1\},\fz]$, $d\in(0,\tz_0]$, $\ez\in(\oz/{d},\fz)$,
and $\bz$, $\gz\in(\omega(1/{\tz_0}-1),\eta)$.
Then
$$\lf[H_Y^*(\cx)\cap(\icgg)'\r]
=\lf[H_{\mathrm{mol}}^{Y,q,d,\ez}(\cx)\cap(\icgg)'\r]$$
with equivalent quasi-norms.
\end{theorem}
\begin{proof}
Let $Y(\cx)$, $q$, and $d$ be as in the present theorem.
It is easy to see that, for any given $\ez\in(0,\fz)$,
any $(Y(\cx),\fz)$-atom is a $(Y(\cx),q,\ez)$-molecule.
Thus, by Theorem \ref{atthm}, we have
$$H_Y^*(\cx)\subset H_{\mathrm{atom}}^{Y,\fz,d}(\cx)\subset H_{\mathrm{mol}}^{Y,q,d,\ez}(\cx)$$
and, for any $f\in H_Y^*(\cx)$,
$$\|f\|_{H_{\mathrm{mol}}^{Y,q,d,\ez}(\cx)}
\ls\|f\|_{H_{\mathrm{atom}}^{Y,\fz,d}(\cx)}
\ls\|f\|_{H_Y^*(\cx)}.$$

Next, we show that $H_{\mathrm{mol}}^{Y,q,d,\ez}(\cx)\subset H_Y^*(\cx)$
with $\ez\in(\oz/d,\fz)$.
By Theorem \ref{atthm}, it suffices to prove
$H_{\mathrm{mol}}^{Y,q,d,\ez}(\cx)\subset H_{\mathrm{atom}}^{Y,q,d}(\cx)$.
Let $f\in H_{\mathrm{mol}}^{Y,q,d,\ez}(\cx)$. Then,
by Definition \ref{mold}, we know that there exist
sequences $\{\lz_i\}_{i\in\nn}\subset[0,\fz)$ and $\{m_i\}_{i\in\nn}$
of $(Y(\cx),q,\epsilon)$-molecules centered, respectively,
at balls $\{B_i\}_{i\in\nn}$ of $\cx$ such that
$f=\sum_{i\in\nn}\lz_i m_i$ in $(\icgg)'$ and
\begin{align}\label{9.1.x1}
\lf\|\lf\{\sum_{i\in\nn}\lf[\frac{\lz_i}{\|\ch1_{B_i}\|
_{Y(\cx)}}\r]^{d}\ch1_{B_i}\r\}^{\frac{1}{d}}\r\|_{Y(\cx)}
\ls\|f\|_{H_{\mathrm{mol}}^{Y,q,d,\ez}(\cx)}.
\end{align}
From this and Lemma \ref{moldec}, we deduce that
\begin{align*}
f&=\sum_{i\in\nn}\sum_{j=0}^{\fz}\lz_iu_{i,j}b_{i,j}
+\sum_{i\in\nn}\sum_{j=1}^{\fz}\sum_{k=j}^{\fz}\lz_iv_{i,j,k}c_{i,j,k}
\end{align*}
in $(\icgg)'$, where, for any $i\in\nn$, $j\in\zz_+$,
and $k\in\zz_+\cap[j,\fz)$,
$$u_{i,j}:=\frac{2\dz^{\ez j}\|\ch1_{\dz^{-j}B_i}\|_{Y(\cx)}}
{\|\ch1_{B_i}\|_{Y(\cx)}}\quad\mbox{and}\quad
v_{i,j,k}:=\frac{2C_{(\mu)}\dz^{(\ez-\oz)k+\oz j-\oz}
\|\ch1_{\dz^{-j}B_i}\|_{Y(\cx)}}{\|\ch1_{B_i}\|_{Y(\cx)}}$$
with $C_{(\mu)}$ as in \eqref{eq-doub}, and both $b_{i,j}$ and
$c_{i,j,k}$ are $(Y(\cx),q)$-atoms supported in $\dz^{-j}B_i$.
Moreover,
\begin{align}\label{5.14.x3}
\|f\|_{H_{\mathrm{atom}}^{Y,q,d}(\cx)}
&\ls\lf\|\lf\{\sum_{i\in\nn}\sum_{j=0}^{\fz}
\lf[\frac{\lz_iu_{i,j}}{\|\ch1_{\dz^{-j}B_i}\|_{Y(\cx)}}\r]
^{d}\ch1_{\dz^{-j}B_i}\r\}^{\frac{1}{d}}\r\|_{Y(\cx)}\noz\\
&\qquad+\lf\|\lf\{\sum_{i\in\nn}\sum_{j=1}^{\fz}\sum_{k=j}^{\fz}
\lf[\frac{\lz_iv_{i,j,k}}{\|\ch1_{\dz^{-j}B_i}\|_{Y(\cx)}}\r]
^{d}\ch1_{\dz^{-j}B_i}\r\}^{\frac{1}{d}}\r\|_{Y(\cx)}\noz\\
&=:{\rm I}_1+{\rm I}_2.
\end{align}

We first estimate ${\rm I}_1$. Choose $h\in(\frac{\oz}{\ez d},1)$.
By Lemma \ref{lem6.2}, \eqref{eq-doub}, Assumption \ref{assump1},
and \eqref{9.1.x1}, we know that
\begin{align}\label{5.14.x4}
{\rm I}_1&\ls\lf\|\lf\{\sum_{i\in\nn}\sum_{j=0}^{\fz}
\lf[\frac{\lz_i\dz^{\ez j}}{\|\ch1_{B_i}\|_{Y(\cx)}}\r]
^{d}\ch1_{\dz^{-j}B_i}\r\}^{\frac{1}{d}}\r\|_{Y(\cx)}\noz\\
&\ls\lf\|\lf\{\sum_{i\in\nn}\sum_{j=0}^{\fz}
\lf[\frac{\lz_i}{\|\ch1_{B_i}\|_{Y(\cx)}}\r]^{d}
\dz^{(\ez d-\oz/h)j}\lf[\cm\lf(\ch1_{B_i}\r)\r]^{\frac{1}{h}}
\r\}^{\frac{1}{d}}\r\|_{Y(\cx)}\noz\\
&\ls\lf\|\lf\{\sum_{i\in\nn}
\lf[\frac{\lz_i}{\|\ch1_{B_i}\|_{Y(\cx)}}\r]^{d}
\lf[\cm\lf(\ch1_{B_i}\r)\r]^{\frac{1}{h}}
\r\}^{\frac{1}{d}}\r\|_{Y(\cx)}\noz\\
&\sim\lf\|\lf\{\sum_{i\in\nn}
\lf[\frac{\lz_i}{\|\ch1_{B_i}\|_{Y(\cx)}}\r]^{d}
\lf[\cm\lf(\ch1_{B_i}\r)\r]^{\frac{1}{h}}
\r\}^{h}\r\|^{\frac{1}{hd}}_{Y^{\frac{1}{hd}}(\cx)}\noz\\
&\ls\lf\|\lf\{\sum_{i\in\nn}\lf[\frac{\lz_i}{\|\ch1_{B_i}\|
_{Y(\cx)}}\r]^{d}\ch1_{B_i}\r\}^{\frac{1}{d}}\r\|_{Y(\cx)}
\ls\|f\|_{H_{\mathrm{mol}}^{Y,q,d,\ez}(\cx)}.
\end{align}

For ${\rm I}_2$, from \eqref{5.14.x4}, we deduce that
\begin{align}\label{5.14.x5}
{\rm I}_2&\ls\lf\|\lf\{\sum_{i\in\nn}\sum_{j=1}^{\fz}\sum_{k=j}^{\fz}
\lf[\frac{\lz_i\dz^{(\ez-\oz)k+\oz j-\oz}}{\|\ch1_{B_i}\|_{Y(\cx)}}\r]
^{d}\ch1_{\dz^{-j}B_i}\r\}^{\frac{1}{d}}\r\|_{Y(\cx)}\noz\\
&\ls\lf\|\lf\{\sum_{i\in\nn}\sum_{j=1}^{\fz}
\lf[\frac{\lz_i\dz^{\ez j}}{\|\ch1_{B_i}\|_{Y(\cx)}}\r]
^{d}\ch1_{\dz^{-j}B_i}\r\}^{\frac{1}{d}}\r\|_{Y(\cx)}
\ls\|f\|_{H_{\mathrm{mol}}^{Y,q,d,\ez}(\cx)}.
\end{align}
Combining \eqref{5.14.x3}, \eqref{5.14.x4}, and \eqref{5.14.x5},
we obtain
$\|f\|_{H_{\mathrm{atom}}^{Y,q,d}(\cx)}
\ls\|f\|_{H_{\mathrm{mol}}^{Y,q,d,\ez}(\cx)}$,
which completes the proof of
$H_{\mathrm{mol}}^{Y,q,d,\ez}(\cx)\subset H_Y^*(\cx)$
and hence of Theorem \ref{mol}.
\end{proof}

As a corollary of Theorems \ref{maxchprop} and \ref{mol},
the following conclusion shows that the molecular Hardy space
in Definition \ref{mold}(ii) is independent of the choices
of $(\icgg)'$.

\begin{corollary}\label{molecor}
Let $Y(\cx)$, $q$, $d$, $\ez$, $\eta$, $\oz$, and $\tz_0$
be as in Theorem \ref{mol}. Then $H_{\mathrm{mol}}^{Y,q,d,\ez}(\cx)$
is independent of the choices of $(\icgg)'$ whenever
$\bz$, $\gz\in(\omega(1/{\tz_0}-1),\eta)$.
\end{corollary}

\begin{remark}\label{molere}
We give several applications of Theorem \ref{mol} as follows.
\begin{enumerate}
\item[{\rm (i)}]
Let $p\in({\omega}/(\omega+\eta),1]$ with $\omega$ as
in \eqref{eq-doub} and $\eta$ as in Definition \ref{expati}.
If $Y(\cx):=L^{p}(\cx)$,
then, by Remark \ref{atomre}(i), we know that
$L^{p}(\cx)$ satisfies all the assumptions of
Theorem \ref{mol}. In this case,
Theorem \ref{mol} is a different variant of
\cite[Proposition 5.5]{hlyy19} which uses the different
molecules and equips a different quasi-norm.

\item[{\rm (ii)}]
Let $p\in({\omega}/(\omega+\eta),1]$ with $\omega$
as in \eqref{eq-doub} and $\eta$ as in Definition \ref{expati},
and $r\in(0,\fz)$. If $Y(\cx):=L^{p,r}(\cx)$,
then, by Remark \ref{atomre}(ii),
we know that $L^{p,r}(\cx)$ satisfies all the assumptions
of Theorem \ref{mol}. In this case,
Theorem \ref{mol} is a different variant of
\cite[Theorem 6.4]{zhy} which uses the same molecules,
but equips a different quasi-norm.

\item[{\rm (iii)}]
Let $p\in({\omega}/(\omega+\eta),1]$, $r\in(1,\fz)$,
$w\in A_{r}(\cx)$ satisfy
$$r_w:=\inf\lf\{r\in[1,\fz):\
\vz\in A_r(\cx)\r\}\in(1,p(\omega+\eta)/{\omega})$$
with $\omega$ as in \eqref{eq-doub} and
$\eta$ as in Definition \ref{expati}.
If $Y(\cx):=L^p_w(\cx)$, then, by Remark \ref{atomre}(iii),
we know that $L^p_w(\cx)$ satisfies all the assumptions
of Theorem \ref{mol}. In this case, Theorem \ref{mol} is
a different variant of \cite[Theorem 6.8]{fmy19} which uses
same molecules, but equips a different quasi-norm.

\item[{\rm (iv)}]
Let $\Phi$ be an Orlicz function, $p_{\Phi}^+=1$,
and $p_{\Phi}^-\in({\omega}/(\omega+\eta),1]$
with $\omega$ as in \eqref{eq-doub} and
$\eta$ as in Definition \ref{expati}.
If $Y(\cx):=L^{\Phi}(\cx)$, then, by Remark \ref{atomre}(iv),
we know that $L^{\Phi}(\cx)$ satisfies all the assumptions
of Theorem \ref{mol}. In this case, Theorem \ref{mol} is a
different variant of \cite[Theorem 6.8]{fmy19} which uses
the same molecules, but equips a different quasi-norm.

\item[{\rm (v)}]
Let $(\cx,\rho,\mu)$ be an RD-space,
$p(\cdot)\in C^{\log}(\cx)$ satisfy
$\widetilde{p_-}\in(\oz/(\oz+\eta),\fz)$
with $\omega$ as in \eqref{eq-doub}
and $\eta$ as in Definition \ref{expati}.
If $Y(\cx):=L^{p(\cdot)}(\cx)$,
then, by Remark \ref{atomre}(v), we know that
$L^{p(\cdot)}(\cx)$ satisfies all the assumptions of
Theorem \ref{mol}. In this case, Theorem \ref{mol} is new.
\end{enumerate}

Let $\vz$ be a growth function as in Remark \ref{qbsdefrem}(iii).
As was mentioned in Remark \ref{atomre},
Theorem \ref{mol} is not applicable to $H^{*,\vz}(\cx)$.
Fu et al. \cite{fmy19} obtained a similar
conclusion of Theorem \ref{mol} in case that
$Y(\cx):=L^{\vz}(\cx)$ independently.
\end{remark}

\section{Dual space of $H_Y^*(\cx)$ \label{s-lipa}}

In this section, we give the dual space of $H_Y^*(\cx)$.
To this end, we first introduce the notion of absolutely
continuous quasi-norms as follows
(see, for instance, \cite[Chapter 1, Definition 3.1]{bs88}
and \cite[Definition 3.2]{wyy} for the corresponding Euclidean case).
\begin{definition}\label{ably}
Let $Y(\cx)$ be a ball quasi-Banach function space on $\cx$.
A function $f\in Y(\cx)$ is said to have an
\emph{absolutely continuous quasi-norm} in $Y(\cx)$ if
$\|f\ch1_{E_j}\|_{Y(\cx)}\downarrow0$
whenever $\{E_j\}_{j\in\nn}$ is a sequence of $\mu$-measurable
sets satisfying $E_j\supset E_{j+1}$ for any $j\in\nn$, and
$\bigcap_{j\in\nn}E_j=\emptyset$. Moreover, $Y(\cx)$ is said to
have an \emph{absolutely continuous quasi-norm} if, for any
$f\in Y(\cx)$, $f$ has an absolutely continuous quasi-norm in $Y(\cx)$.
\end{definition}

\begin{remark}\label{ablyrem}
\begin{enumerate}
\item[{\rm (i)}]
Assume that $Y(\cx)$ is a ball quasi-Banach function space
on $\cx$ and has an absolutely continuous quasi-norm.
Let $p\in(0,\fz)$. Then, by Definition \ref{cvex} and
Remark \ref{covrem}, we know that $Y^p(\cx)$ also has
an absolutely continuous quasi-norm.

\item[{\rm (ii)}] Observe that, in Definition \ref{ably},
if we replace $\bigcap_{j\in\nn}E_j=\emptyset$ by
$\mu(\bigcap_{j\in\nn}E_j)=0$, we obtain its another
equivalent formulation.
\end{enumerate}

\end{remark}

We prove the following dominated convergence
theorem on $Y(\cx)$ based on the assumption that
$Y(\cx)$ has an absolutely continuous quasi-norm.
For the case of Banach function spaces on $\rn$,
this conclusion is a simple corollary of
\cite[Chapter 1, Proposition 3.6]{bs88}.

\begin{lemma}\label{dominate}
Assume that $Y(\cx)$ is a ball quasi-Banach function space on $\cx$
and has an absolutely continuous quasi-norm.
Let $g\in Y(\cx)$ and $\{f_m\}_{m\in\nn}$ be a sequence of $\mu$-measurable
functions satisfying that $|f_m|\le|g|$ for any $m\in\nn$,
and $\lim_{m\rightarrow\fz}f_m=f$ almost everywhere in $\cx$.
Then
\begin{align*}
\lim_{m\rightarrow\fz}\lf\|f_m-f\r\|_{Y(\cx)}=0.
\end{align*}
\end{lemma}

\begin{proof}
Without loss of generality,
we may assume that $\lim_{m\rightarrow\fz}f_m=f$ pointwisely.
First, it is easy to see that, for any $x\in\cx$,
\begin{align}\label{6.1x}
|f(x)|&\le |g(x)|.
\end{align}
Now, we prove this lemma in case that $\supp f\st B$,
where $B$ is a ball of $\cx$. For any $n$, $j\in\nn$, define
$$E_{n,j}:=\lf\{x\in B:\
\lf|f_m(x)-f(x)\r|<2^{-j},\ \forall\, m\ge n\r\}.$$
It is easy to see that, for any given $j\in\nn$,
\begin{align*}
\bigcup_{n\in\nn}E_{n,j}=B,
\end{align*}
$E_{n,j}\st E_{n+1,j}$ for any $n\in\nn$, and
$$\lim_{n\rightarrow\fz}\mu\lf(E_{n,j}\r)=\mu(B).$$
Therefore, for any $j\in\nn$, there exists an $n_j\in\nn$ such that
\begin{align}\label{6.21.y3}
\mu\lf(B\setminus E_{n_j,j}\r)<2^{-j}.
\end{align}
For any $J\in\nn$, define
$$A_J:=\bigcap_{j=J}^{\fz}E_{n_j,j}\quad
\mbox{and}\quad B_J:=B\setminus A_J.$$
Then, for any $J\in\nn$,
\begin{align}\label{6.21.x3}
B_{J+1}&=B\setminus A_{J+1}
=\bigcup_{j=J+1}^{\fz}\lf(B\setminus E_{n_j,j}\r)
\st\bigcup_{j=J}^{\fz}\lf(B\setminus E_{n_j,j}\r)
=B\setminus A_{J}=B_J;
\end{align}
moreover, by \eqref{6.21.y3}, we know that,
for any $J\in\nn$,
\begin{align*}
\mu\lf(B_{J}\r)
=\mu\lf(\bigcup_{j=J}^{\fz}\lf(B\setminus E_{n_j,j}\r)\r)
\le\sum_{j=J}^{\fz}\mu\lf(B\setminus E_{n_j,j}\r)
<\sum_{j=J}^{\fz}2^{-j}\ls2^{-J},
\end{align*}
which, combined with \eqref{6.21.x3}, further implies that
\begin{align}\label{6.21.y2}
\mu\lf(\bigcap_{J\in\nn}B_{J}\r)=0.
\end{align}
On one hand, from \eqref{6.21.x3}, \eqref{6.21.y2},
and Remark \ref{ablyrem}(ii), we deduce that,
for any $\ez\in(0,\fz)$, there exists a $J_1\in\nn$ such that
\begin{align}\label{6.21.x1}
\lf\|g\ch1_{B_{J_1}}\r\|_{Y(\cx)}<\frac{\ez}2.
\end{align}
On the other hand, it is easy to see that, for any $\ez\in(0,\fz)$,
there exists a $J_2\in(J_1,\fz)\cap\nn$ such that
\begin{align}\label{6.21.x2}
2^{-J_2}<\frac{\ez}{2\|\ch1_{B}\|_{Y(\cx)}},
\end{align}
which further implies that there exists an $n_{J_2}\in\nn$ such that,
for any $x\in A_{J_2}\st E_{n_{J_2},J_2}$ and any $m>n_{J_2}$,
\begin{align}\label{6.21.y1}
\lf|f_m(x)-f(x)\r|<2^{-J_2}.
\end{align}
Combining Definition \ref{qbs}(ii), \eqref{6.21.y1}, \eqref{6.1x},
\eqref{6.21.x2}, \eqref{6.21.x3}, and \eqref{6.21.x1},
we conclude that, for any $m>n_{J_2}$,
\begin{align*}
\lf\|f_m-f\r\|_{Y(\cx)}
&\ls\lf\|\lf(f_m-f\r)\ch1_{A_{J_2}}\r\|_{Y(\cx)}
+\lf\|\lf(f_m-f\r)\ch1_{B_{J_2}}\r\|_{Y(\cx)}\\
&\ls 2^{-J_2}\lf\|\ch1_{B}\r\|_{Y(\cx)}
+\lf\|g\ch1_{B_{J_2}}\r\|_{Y(\cx)}\\
&\ls\frac{\ez}2
+\lf\|g\ch1_{B_{J_1}}\r\|_{Y(\cx)}\ls\ez.
\end{align*}
Thus,
\begin{align}\label{6.21.z1}
\lim_{m\rightarrow\fz}\lf\|f_m-f\r\|_{Y(\cx)}=0.
\end{align}

Next, we prove this lemma for any $f\in Y(\cx)$.
Let $\{B_n\}_{n\in\nn}$ be a sequence of balls satisfying that
$\cx=\bigcup_{n\in\nn}B_n$ and, for any $n\in\nn$, $B_n\st B_{n+1}$.
On one hand, by Definition \ref{ably}, we know that,
for any $\ez\in(0,\fz)$, there exists an $N_1\in\nn$ such that
\begin{align}\label{6.7y}
\lf\|g\lf(1-\ch1_{B_{N_1}}\r)\r\|_{Y(\cx)}&<\frac{\ez}2.
\end{align}
On the other hand, from \eqref{6.21.z1}, we deduce that
there exists an $N_2\in\nn$ such that, for any $m>N_2$,
\begin{align*}
\lf\|\lf(f_m-f\r)\ch1_{B_{N_1}}\r\|_{Y(\cx)}&<\frac{\ez}2,
\end{align*}
which, combined with Definition \ref{qbs}(ii), \eqref{6.1x},
and \eqref{6.7y}, further implies that, for any $m>N_2$,
\begin{align*}
\lf\|f_m-f\r\|_{Y(\cx)}
&\ls\lf\|f_m\lf(1-\ch1_{B_{N_1}}\r)\r\|_{Y(\cx)}
+\lf\|\lf(f_m-f\r)\ch1_{B_{N_1}}\r\|_{Y(\cx)}
+\lf\|f\lf(\ch1_{B_{N_1}}-1\r)\r\|_{Y(\cx)}\\
&\ls\lf\|g\lf(1-\ch1_{B_{N_1}}\r)\r\|_{Y(\cx)}
+\lf\|\lf(f_m-f\r)\ch1_{B_{N_1}}\r\|_{Y(\cx)}
\ls\frac{\ez}2+\frac{\ez}2\sim\ez.\noz
\end{align*}
This shows that $\lim_{m\rightarrow\fz}\lf\|f_m-f\r\|_{Y(\cx)}=0$
and hence finishes the proof of Lemma \ref{dominate}.
\end{proof}

If we assume that $Y(\cx)$ has an absolutely continuous quasi-norm,
then we have the following density conclusions of $H_Y^*(\cx)$.

\begin{lemma}\label{dense}
Let $Y(\cx)$ be a ball quasi-Banach function space on $\cx$ satisfying Assumption \ref{assump1} with $p_-\in({\omega}/(\omega+\eta),\fz)$,
where $\omega$ is as in \eqref{eq-doub} and $\eta$ as in
Definition \ref{expati}. Further assume that $Y(\cx)$
has an absolutely continuous quasi-norm and
satisfies Assumption \ref{assump2} with the same $p_-$
as in Assumption \ref{assump1},
$\tz_0\in({\omega}/(\omega+\eta),\underline{p})$,
and $p_0\in(\tz_0,\fz)$, where $\underline{p}$ is as in \eqref{2.1y}.
Let $q\in(\max\{p_0,1\},\fz]$ and $d\in(0,\tz_0]$.
Then
\begin{enumerate}
\item[{\rm (i)}] $H_{\mathrm{fin}}^{Y,q,d}(\cx)$
is dense in $H_Y^*(\cx)$;

\item[{\rm (ii)}] $H_{\mathrm{fin}}^{Y,\fz,d}(\cx)\cap\uu\cc(\cx)$
is dense in $H_Y^*(\cx)$.
\end{enumerate}
\end{lemma}

\begin{proof}
Let all the symbols be as in the present lemma.
We first prove (i). Let $f\in H_Y^*(\cx)$ and
$\bz$, $\gz\in(\omega(1/{\tz_0}-1),\eta)$.
Then, by Theorem \ref{atthm}, we know that
there exist a sequence $\{\lz_j\}_{j\in\nn}$ of
non-negative numbers and a sequence $\{a_j\}_{j\in\nn}$
of $(Y(\cx),q)$-atoms supported, respectively, in balls
$\{B_j \}_{j\in\nn}$ of $\cx$ such that
\begin{align*}
f&=\sum_{j\in\nn}\lz_j a_j
\end{align*}
in $(\icgg)'$, and
$$\lf\|\lf\{\sum_{j\in\nn}
\lf[\frac{\lz_j}{\|\ch1_{B_j}\|_{Y(\cx)}}\r]^{d}\ch1_{B_j}\r\}
^{\frac{1}{d}}\r\|_{Y(\cx)}\sim\|f\|_{H_Y^*(\cx)}<\fz.$$
Let
$$F:=\lf\{\sum_{j\in\nn}\lf[\frac{\lz_j}{\|\ch1_{B_j}\|
_{Y(\cx)}}\r]^{d}\ch1_{B_j}\r\}^{\frac{1}{d}}$$
and, for any $N\in\nn$, let
$$f_N:=\sum_{j=1}^N\lz_j a_j\quad\mbox{and}\quad
F_N:=\lf\{\sum_{j=1}^N\lf[\frac{\lz_j}{\|\ch1_{B_j}\|
_{Y(\cx)}}\r]^{d}\ch1_{B_j}\r\}^{\frac{1}{d}}.$$
Then $F^d\in Y^{\frac1d}(\cx)$, $f_N\in H_{\mathrm{fin}}^{Y,q,d}(\cx)$
and $F_N^d\le F^d$ for any $N\in\nn$, and
$\lim_{N\rightarrow\fz}F_N^d=F^d$ almost everywhere in $\cx$.
From this, Theorem \ref{atthm}, Remark \ref{ablyrem}(i),
and Lemma \ref{dominate}, we deduce that
\begin{align*}
\lf\|f-f_N\r\|_{H_Y^*(\cx)}
&\sim\lf\|f-f_N\r\|_{H_{\mathrm{atom}}^{Y,q,d}(\cx)}
\ls\lf\|\lf\{\sum_{j=N+1}^{\fz}
\lf[\frac{\lz_j}{\|\ch1_{B_j}\|_{Y(\cx)}}\r]^{d}\ch1_{B_j}\r\}
^{\frac{1}{d}}\r\|_{Y(\cx)}\\
&\sim\lf\|\lf(F^d-F_N^d\r)^{\frac{1}{d}}\r\|_{Y(\cx)}
\sim\lf\|F^d-F_N^d\r\|_{Y^{\frac{1}{d}}(\cx)}^{\frac{1}{d}}
\rightarrow0
\end{align*}
as $N\rightarrow\fz$, which further implies (i).
Thus, to prove the present lemma, it remains to show (ii).
By (i), we only need to show that
$H_{\mathrm{fin}}^{Y,\fz,d}(\cx)\cap\uu\cc(\cx)$ is dense in
$H_{\mathrm{fin}}^{Y,\fz,d}(\cx)$ with the quasi-norm
$\|\cdot\|_{H_Y^*(\cx)}$. To this end,
let $a$ be a $(Y(\cx),\fz)$-atom supported in a ball
$B:=B(x_B,r_B)\st\cx$ with $x_B\in\cx$ and $r_B\in(0,\fz)$,
and $k_0\in\zz$ sufficiently large.
Then, by the proof of \cite[Lemma 5.3(ii)]{zhy}
(with the same notation as therein), we know that, for any $k>k_0$,
$S_ka\in\uu\cc(\cx)$, $S_ka$ is a $(Y(\cx),\fz)$-atom supported in
$B(x_B,2A_0r_B)$, and
\begin{align}\label{appoiden}
\lim_{k\rightarrow\fz}\lf\|S_ka-a\r\|_{L^2(\cx)}=0.
\end{align}
Thus,
$$\widetilde{a}:=\frac{(S_ka-a)[\mu(B(x_B,2A_0r_B))]^{1/2}}
{\|S_ka-a\|_{L^2(\cx)}\|\ch1_{B(x_B,2A_0r_B)}\|_{Y(\cx)}}$$
is a $(Y(\cx),2)$-atom. From this, Theorem \ref{atthm},
and \eqref{appoiden}, we deduce that
\begin{align*}
\lim_{k\rightarrow\fz}\lf\|S_ka-a\r\|_{H_Y^*(\cx)}
&\sim\lim_{k\rightarrow\fz}\lf\|S_ka-a\r\|_{H_{\mathrm{atom}}^{Y,2,d}(\cx)}\\
&\sim\frac{\|\ch1_{B(x_B,2A_0r_B)}\|_{Y(\cx)}}{[\mu(B(x_B,2A_0r_B))]^{1/2}}
\lim_{k\rightarrow\fz}\lf\|S_ka-a\r\|_{L^2(\cx)}=0.
\end{align*}
This finishes the proof of (iii) and hence of Lemma \ref{dense}.
\end{proof}

Let $q\in(0,{\infty})$.
Recall that the \emph{space} $L^q_{\mathbb{B}}({\cx})$
is defined to be the set of all $\mu$-measurable functions
$f$ on $\cx$ such that, for any ball $B\st\cx$,
$$\int_B|f(x)|^q\,d\mu(x)<\fz.$$
Next, we introduce the notion of ball Campanato-type
function spaces associated with $Y(\cx)$
(see \cite[Definition 2.4]{hyy21} and \cite[Definition 3.2]{zhyy21},
respectively, for the cases of anisotropic mixed-norm Lebesgue spaces
and ball quasi-Banach function spaces on $\rn$).

\begin{definition}\label{cqb}
Let $Y(\cx)$ be a ball quasi-Banach function space on $\cx$, $q\in[1,{\infty})$, and $d\in(0,\infty)$.
Then the \emph{ball Campanato-type function space}
${\cl}_{Y,q,d}(\cx)$, associated with $Y(\cx)$,
is defined to be the set of all $f\in L^q_{\mathbb{B}}({\cx})$
such that
\begin{align*}
\|f\|_{{\cl}_{Y,q,d}(\cx)}
:=&\,\sup
\lf\|\lf\{\sum_{i=1}^m
\lf[\frac{{\lambda}_i}{\|{\mathbf{1}}_{B_i}\|_{Y(\cx)}}\r]^d
{\mathbf{1}}_{B_i}\r\}^{\frac1d}\r\|_{Y(\cx)}^{-1}
\\
&\quad\quad\times\sum_{j=1}^m\lf\{\frac{{\lambda}_j\mu(B_j)}
{\|{\mathbf{1}}_{B_j}\|_{Y(\cx)}}\lf[\frac1{\mu(B_j)}\int_{B_j}
\lf|f(x)-m_{B_j}(f)\r|^q \,d\mu(x)\r]^\frac1q\r\}<\fz,
\end{align*}
where, for any ball $B\st\cx$,
$$m_{B}(f):=\frac{1}{\mu(B)}\int_{B}f(x)\,d\mu(x),$$
and the supremum is taken over all $m\in\nn$,
balls $\{B_j\}_{j=1}^m$ of $\cx$, and $\{\lambda_j\}_{j=1}^m\subset[0,\infty)$
with $\sum_{j=1}^m\lambda_j\neq0$.
\end{definition}

Now, we give the dual space of $H_Y^*(\cx)$,
which is the main result of this section.

\begin{theorem}\label{qbs-dual}
Assume that $Y(\cx)$ is a ball quasi-Banach function space on
$\cx$ having an absolutely continuous quasi-norm.
Further assume that $Y(\cx)$ satisfies Assumption
\ref{assump1} with $p_-\in({\omega}/(\omega+\eta),\fz)$,
where $\omega$ is as in \eqref{eq-doub} and $\eta$ as in
Definition \ref{expati}, and Assumption \ref{assump2}
with the same $p_-$ as in Assumption \ref{assump1},
$\tz_0\in({\omega}/(\omega+\eta),\underline{p})$,
and $p_0\in(\tz_0,\fz)$, where $\underline{p}$ is as in \eqref{2.1y}.
Let $q\in(\max\{p_0,1\},\fz]$ and $d\in(0,\tz_0]$.
Then the dual space of $H_Y^*(\cx)$, denoted by $(H_Y^*(\cx))^*$,
is ${\cl}_{Y,q',d}(\cx)$ with $1/q+1/q'=1$ in the following sense:
\begin{enumerate}
\item[{\rm (i)}] if $f\in{\cl}_{Y,q',d}(\cx)$,
then the linear functional
\begin{align}\label{2te1}
T_f:\ g\mapsto T_f(g):=\int_{\cx}f(x)g(x)\,d\mu(x),
\end{align}
initially defined for any $g\in H_{\mathrm{fin}}^{Y,q,d}(\cx)$,
has a bounded extension to $H_Y^*(\cx)$;

\item[{\rm (ii)}] conversely, any continuous linear
functional on $H_Y^*(\cx)$ arises as in \eqref{2te1}
with a unique $f\in{\cl}_{Y,q',d}(\cx)$.
\end{enumerate}
Moreover,
$\|f\|_{{\cl}_{Y,q',d}(\cx)}\sim\|T_f\|_{(H_Y^*(\cx))^*}$,
where the positive equivalence constants
are independent of $f$.
\end{theorem}

\begin{proof}
Let all the symbols be as in the present theorem.
We first prove (i) in the case $q\in(\max\{p_0,1\},\fz)$.
Let $f\in{\cl}_{Y,q',d}(\cx)$ and $g\in H_{\mathrm{fin}}^{Y,q,d}(\cx)$.
By Definition \ref{finatom}, we know that there exist an $N\in\nn$,
a sequence $\{\lz_j\}_{j=1}^N\subset[0,\fz)$, and
a sequence $\{a_j\}_{j=1}^N$ of $(Y(\cx),q)$-atoms
supported, respectively, in balls $\{B_j\}_{j=1}^N$ such that
$g=\sum_{j=1}^N\lz_ja_j$ almost everywhere in $\cx$ and
\begin{equation*}
\lf\|\lf\{\sum_{j=1}^{N}\lf[\frac{\lz_j}
{\|\ch1_{B_j}\|_{Y(\cx)}}\r]^{d}\ch1_{B_j}\r\}^{\frac1{d}}\r\|_{Y(\cx)}
\ls\|g\|_{H_{\mathrm{fin}}^{Y,q,d}(\cx)}.
\end{equation*}
By this, Definition \ref{atom}, the H\"older inequality,
and Theorem \ref{finatomeq}(i), we obtain
\begin{align*}
\lf|T_f(g)\r|&=\lf|\int_{\cx}f(x)g(x)\,d\mu(x)\r|
\le\sum_{j=1}^N\lz_j\lf|\int_{B_j}f(x)a_j(x)\,d\mu(x)\r|\\
&\le\sum_{j=1}^N\lz_j\int_{B_j}
\lf|\lf[f(x)-m_{B_j}(f)\r]a_j(x)\r|\,d\mu(x)\\
&\le\sum_{j=1}^N\lz_j\lf\|a_j\r\|_{L^q(\cx)}
\lf\{\int_{B_j}\lf|f(x)-m_{B_j}(f)\r|^{q'}\,d\mu(x)\r\}^{1/q'}\\
&\le\sum_{j=1}^N\frac{\lz_j\mu(B_j)}{\|\ch1_{B_j}\|_{Y(\cx)}}
\lf[\frac{1}{\mu(B_j)}\int_{B_j}\lf|f(x)-m_{B_j}(f)\r|
^{q'}\,d\mu(x)\r]^{1/q'}\\
&\le\lf\|\lf\{\sum_{j=1}^N
\lf[\frac{{\lambda}_j}{\|{\mathbf{1}}_{B_j}\|_{Y(\cx)}}\r]^d
{\mathbf{1}}_{B_j}\r\}^{\frac1d}\r\|_{Y(\cx)}
\|f\|_{{\cl}_{Y,q',d}(\cx)}\\
&\ls\|g\|_{H_{\mathrm{fin}}^{Y,q,d}(\cx)}
\|f\|_{{\cl}_{Y,q',d}(\cx)}
\sim\|g\|_{H_Y^*(\cx)}\|f\|_{{\cl}_{Y,q',d}(\cx)}.
\end{align*}
From this, Lemma \ref{dense}(ii), and a standard density argument,
we deduce that $T_f$ can uniquely be extended to a bounded linear
functional on $H_Y^*(\cx)$ and
$$\lf\|T_f\r\|_{(H_Y^*(\cx))^*}
\ls\|f\|_{{\cl}_{Y,q',d}(\cx)},$$
which then completes the proof of (i)
in the case $q\in(\max\{p_0,1\},\fz)$.

Next, we prove (i) in the case $q=\fz$. Indeed, using
Theorem \ref{finatomeq}(ii) and Lemma \ref{dense}(iii),
and repeating the previous proof,
we conclude that any $f\in{\cl}_{Y,1,d}(\cx)$ induces
a bounded linear functional on $H_Y^*(\cx)$, which is initially defined on
$H_{\fin}^{Y,\infty,d}(\cx)\cap\uu\cc{(\cx)}$
given by setting, for any
$g\in H_{\fin}^{Y,\infty,d}(\cx)\cap\uu\cc{(\cx)}$,
\begin{equation}\label{opl}
T_f:\ g\mapsto\ T_f(g):=\int_{\cx}f(x)g(x)\,d\mu(x),
\end{equation}
and then has a bounded linear extension to $H_Y^*(\cx)$.
Let $f\in{\cl}_{Y,1,d}(\cx)$. Thus, it remains to prove that,
for any $g\in H_{\fin}^{Y,\infty,d}(\cx)$,
\begin{equation}\label{oo9}
T_f(g)=\int_{\cx}f(x)g(x)\,d\mu(x).
\end{equation}
Indeed, suppose $g\in H_{\fin}^{Y,\infty,d}(\cx)$
and $\supp g\st B(x_0,R)$ for some $x_0\in\cx$ and $R\in(0,\fz)$.
By the proof of Lemma \ref{dense}(iii), we know that there exists a
sequence $\{S_k\}_{k\in\nn}$ of bounded operators on $L^2(\cx)$
such that $S_kg\in H_{\mathrm{fin}}^{Y,\fz,d}(\cx)\cap\uu\cc(\cx)$,
$\lim_{k\rightarrow\fz}S_kg(x)=g(x)$ for almost every $x\in\cx$, and
\begin{align*}
\lim_{k\rightarrow\fz}\lf\|S_kg-g\r\|_{H_Y^*(\cx)}=0.
\end{align*}
From this, \eqref{opl}, the fact that
$|fS_kg|\ls\|g\|_{L^{\fz}(\cx)}\ch1_{B(x_0,2A_0R)}|f|$
for any $k\in\nn$, and the Lebesgue dominated convergence theorem,
we deduce that
$$T_f(g)=\lim_{k\rightarrow\fz}T_f(S_kg)
=\lim_{k\rightarrow\fz}\int_{\cx}f(x)S_kg(x)\,d\mu(x)
=\int_{\cx}f(x)g(x)\,d\mu(x),$$
which shows that \eqref{oo9} holds true and hence completes the proof of (i).

Next, we prove (ii).
Let $T\in(H_Y^*(\cx))^*=(H_{\mathrm{atom}}^{Y,q,d}(\cx))^*$.
For any given $q\in(\max\{p_0,1\},\fz]$ and any ball $B\st\cx$,
let $L^q(B)$ be the set of all $L^q(\cx)$ functions
vanishing outside $B$, and
$$L^{q}_0(B):=\lf\{g\in L^q(B):\ \int_{\cx}g(x)\,d\mu(x)=0\r\}.$$
Then $L^{q}_0(B)\subset H_{\mathrm{atom}}^{Y,q,d}(\cx)$ and, for any
$g\in L^{q}_0(B)$ with $\|g\|_{L^{q}(B)}\neq0$,
$a:=\frac{[\mu(B)]^{1/q}}{\|g\|_{L^q(\cx)}\|\ch1_B\|_{Y(\cx)}}g$
is a $(Y(\cx),q)$-atom; moreover,
$$|T(g)|=\frac{\|g\|_{L^q(\cx)}\|\ch1_B\|_{Y(\cx)}}{[\mu(B)]^{1/q}}
|T(a)|\le\frac{\|g\|_{L^q(\cx)}\|\ch1_B\|_{Y(\cx)}}{[\mu(B)]^{1/q}}
\|T\|_{(H_{\mathrm{atom}}^{Y,q,d}(\cx))^*},$$
which implies that
$$\|T\|_{(L^{q}_0(B))^*}\le\frac{\|\ch1_B\|_{Y(\cx)}
\|T\|_{(H_{\mathrm{atom}}^{Y,q,d}(\cx))^*}}{[\mu(B)]^{1/q}}$$
and hence $T$ is a bounded linear functional on $L^{q}_0(B)$.
From this and the Hahn--Banach theorem,
we deduce that $T$ can be extended to the whole
space $L^q(B)$ without increasing its norm.

If $q\in(\max\{p_0,1\},\fz)$, by the duality
$[L^q(B)]^*=L^{q'}(B)$ with $1/q+1/q'=1$,
we know that there exists an $F\in L^{q'}(B)$ such that,
for any $g\in L^{q}_0(B)$,
\begin{align*}
T(g)=\int_{\cx}F(x)g(x)\,d\mu(x)\
\end{align*}
and
\begin{align*}
\|F\|_{L^{q'}(B)}=\|T\|_{(L^q(B))^*}
\le\frac{\|\ch1_B\|_{Y(\cx)}
\|T\|_{(H_{\mathrm{atom}}^{Y,q,d}(\cx))^*}}{[\mu(B)]^{1/q}}.
\end{align*}
For the case $q=\fz$, let $\widetilde{q}\in(\max\{p_0,1\},\fz)$.
By Theorem \ref{atthm}, we know that
$T\in(H_{\mathrm{atom}}^{Y,\fz,d}(\cx))^*$
implies $T\in(H_{\mathrm{atom}}^{Y,\widetilde{q},d}(\cx))^*$
with equivalent norms.
Thus, there exists an $F\in L^{\widetilde{q}'}(B)\subset L^{1}(B)$
such that, for any $g\in L^{\fz}_0(B)$,
$T(g)=\int_{\cx}F(x)g(x)\,d\mu(x)$ and
$$\|F\|_{L^{1}(B)}\le\|F\|_{L^{\widetilde{q}'}(B)}
[\mu(B)]^{1/\widetilde{q}}\le\frac{\|\ch1_B\|_{Y(\cx)}
\|T\|_{(H_{\mathrm{atom}}^{Y,\widetilde{q},d}(\cx))^*}}
{[\mu(B)]^{1/\widetilde{q}}}[\mu(B)]^{1/\widetilde{q}}
\ls\|\ch1_B\|_{Y(\cx)}\|T\|_{(H_{\mathrm{atom}}^{Y,\fz,d}(\cx))^*}.$$
Altogether, we find that,
for any given $q\in(\max\{1,p_0\},\infty]$,
there exists an $F\in L^{q'}(B)$ such that,
for any $g\in L^q_0(B)$,
\begin{align}\label{10.18.x1}
T(g)=\int_{\cx}F(x)g(x)\,d\mu(x)\
\end{align}
and
\begin{align*}
\|F\|_{L^{q'}(B)}\ls\frac{\|\ch1_B\|_{Y(\cx)}
\|T\|_{(H_{\mathrm{atom}}^{Y,q,d}(\cx))^*}}{[\mu(B)]^{1/q}}.
\end{align*}

Next, we claim that, if there exists another function
$\widetilde{F}\in L^{q'}(B)$ such that, for any $g\in L^q_0(B)$,
$T(g)=\int_{B}\widetilde{F}(x)g(x)\,d\mu(x)$, then
$\widetilde{F}-F=m_B(\widetilde{F}-F)$ almost everywhere in $B$.
Indeed, for any $g\in L^q_0(B)$,
$$\int_{B}\lf[\widetilde{F}(x)-F(x)\r]g(x)\,d\mu(x)=0,$$
which, together with the fact that $h-m_B(h)\in L^{q}_0(B)$
for any $h\in L^q(B)$, further implies that, for any
$h\in L^q(B)$,
\begin{align}\label{10.17.x1}
\int_{B}\lf[\widetilde{F}(x)-F(x)\r]\lf[h(x)-m_B(h)\r]\,d\mu(x)=0.
\end{align}
Moreover, for any $H\in L^q(\cx)$ and $G\in L^{q'}(\cx)$,
by the definitions of $m_B(H)$ and $m_B(G)$,
we obtain
\begin{align}\label{4.1.y1}
&\int_{B}\lf[m_B(H)G(x)-m_B(G)H(x)\r]\,d\mu(x)\noz\\
&\hs=\int_{B}\lf\{m_B(H)\lf[G(x)-m_B(G)\r]
+m_B(G)\lf[m_B(H)-H(x)\r]\r\}\,d\mu(x)=0,
\end{align}
which, together with $h\in L^q(B)$ and $F$,
$\widetilde{F}\in L^{q'}(B)$, further implies that
\begin{align}\label{10.17.x2}
\int_{B}m_B(h)\lf[\widetilde{F}(x)-F(x)\r]\,d\mu(x)
=\int_{B}h(x)m_B\lf(\widetilde{F}-F\r)\,d\mu(x).
\end{align}
Combining \eqref{10.17.x1} and \eqref{10.17.x2},
we conclude that, for any $h\in L^q(B)$,
\begin{align*}
\int_{B}h(x)\lf[\widetilde{F}(x)-F(x)
-m_B\lf(\widetilde{F}-F\r)\r]\,d\mu(x)=0,
\end{align*}
which implies that, for almost every $x\in B$,
\begin{align*}
\widetilde{F}(x)-F(x)=m_B\lf(\widetilde{F}-F\r)(x).
\end{align*}

Now, we take a sequence $\{B_j\}_{j\in\nn}$ of balls such that
$B_1\subset B_2\subset\cdots\subset B_j\subset\cdots$
and $\bigcup_{j\in\nn}B_j=\cx$. Then, by \eqref{10.18.x1},
we know that there exists a sequence $\{F_j\}_{j\in\nn}$ of
measurable functions such that, for any $j\in\nn$,
$F_j\in L^{q'}(B_j)$ and, for any $g\in L^{q}_0(B_j)$,
\begin{align}\label{10.18.x2}
T(g)=\int_{\cx}F_j(x)g(x)\,d\mu(x).
\end{align}
Let $\widetilde{F}_1:=F_1$ and
$\widetilde{F}_{j+1}:=F_{j+1}+m_{B_j}(\widetilde{F}_{j}-{F}_{j+1})$
for any $j\in\nn$.
Then, by the above claim, we find that,
for any $j\in\nn$ and almost every $x\in B_j$, $\widetilde{F}_{j+1}(x)=\widetilde{F}_{j}(x)$
and $\widetilde{F}_{j}\in L^{q'}(B_j)$.
Let $g\in H_{\mathrm{fin}}^{Y,q,d}(\cx)$ and
$f$ be a measurable function satisfying that,
for any $j\in\nn$ and $x\in B_j$, $f(x)=\widetilde{F}_{j}(x)$.
Then there exists some $j_0\in\nn$ such that $\supp g\subset B_{j_0}$.
By this, the cancellation of $g$, and \eqref{10.18.x2},
we know that
\begin{align}\label{4.1.y2}
T(g)=\int_{B_{j_0}}F_{j_0}(x)g(x)\,d\mu(x)
=\int_{B_{j_0}}\widetilde{F}_{j_0}(x)g(x)\,d\mu(x)
=\int_{\cx}f(x)g(x)\,d\mu(x).
\end{align}
It remains to prove that $f\in{\cl}_{Y,q',d}(\cx)$.
To this end, for any given $m\in\nn$, any given balls
$\{B_k\}_{k=1}^m$ of $\cx$,
and any given $\{\lambda_k\}_{k=1}^m\subset[0,\infty)$
with $\sum_{k=1}^m\lambda_k\neq0$, and for any $k\in\{1,\ldots,m\}$,
let $h_k\in L^q(B_k)$ with $\|h_k\|_{L^q(B_k)}=1$ be such that
\begin{align}\label{2e5}
\lf[\int_{B_k}\lf|f(x)-m_{B_k}(f)\r|^{q'} \,d\mu(x)\r]^\frac1{q'}
=\int_{B_k}\lf[f(x)-m_{B_k}(f)\r]h_k(x)\,d\mu(x)
\end{align}
and, for any $x\in\cx$, let
\begin{align*}
a_k(x):=
\begin{cases}
\displaystyle0
\ \ &\text{if}\ h_k\ \text{is a constant function},\\
\displaystyle \frac{[\mu(B_k)]^{\frac{1}{q}}
[h_k(x)-m_{B_k}(h_k)]{\mathbf{1}}_{B_k}}
{\|{\mathbf{1}}_{B_k}\|_{Y(\cx)}
\|h_k-m_{B_k}(h_k)\|_{L^q(B_k)}} \ \ &\text{otherwise}.
\end{cases}
\end{align*}
Then it is easy to see that, for any $k\in\{1,\ldots,m\}$,
$a_k$ is a $(Y(\cx),q)$-atom. From this and Theorem \ref{atthm},
it follows that $\sum_{k=1}^m \lambda_k a_k\in H_Y^*(\cx)$.
This, together with \eqref{2e5}, \eqref{4.1.y1}, \eqref{4.1.y2},
Theorem \ref{finatomeq}, $T\in(H_Y^*(\cx))^*$, and
$\|h_k-m_{B_k}(h_k)\|_{L^q(B_k)}\ls1$ for any $k\in\{1,\ldots,m\}$,
further implies that
\begin{align*}
&\sum_{k=1}^m\frac{{\lambda}_k\mu(B_k)}{\|{\mathbf{1}}_{B_k}\|_{Y(\cx)}}
\lf[\frac1{\mu(B_k)}\int_{B_k}
\lf|f(x)-m_{B_k}(f)\r|^{q'} \,d\mu(x)\r]^\frac1{q'}\\
&\quad=\sum_{k=1}^m\frac{{\lambda}_k[\mu(B_k)]^{\frac1q}}
{\|{\mathbf{1}}_{B_k}\|_{Y(\cx)}}\int_{B_k}
\lf[f(x)-m_{B_k}(f)\r]h_k(x)\,d\mu(x)\noz\\
&\quad=\sum_{k=1}^m\frac{{\lambda}_k[\mu(B_k)]^{\frac1q}}
{\|{\mathbf{1}}_{B_k}\|_{Y(\cx)}}
\int_{B_k}\lf[h_k(x)-m_{B_k}(h_k)\r]f(x){\mathbf{1}}_{B_k}(x)\,d\mu(x)\noz\\
&\quad\ls\sum_{k=1}^m{\lambda}_k\int_{B_k}a_k(x)f(x)\,d\mu(x)
\sim\sum_{k=1}^m{\lambda}_k T(a_k)
\sim T\lf(\sum_{k=1}^m{\lambda}_k a_k\r)\noz\\
&\quad\ls\|T\|_{(H_Y^*(\cx))^*}
\lf\|\sum_{k=1}^m{\lambda}_k a_k\r\|_{H_Y^*(\cx)}
\ls\|T\|_{(H_Y^*(\cx))^*}\lf\|\lf\{\sum_{k=1}^m
\lf[\frac{{\lambda}_k}{\|{\mathbf{1}}_{B_k}\|_{Y(\cx)}}\r]^d
{\mathbf{1}}_{B_k}\r\}^{\frac1d}\r\|_{Y(\cx)}.\noz
\end{align*}
This implies that $f\in{\cl}_{Y,q',d}(\cx)$ and
hence finishes the proof of (ii) and also of Theorem \ref{qbs-dual}.
\end{proof}

As a consequence of Theorem \ref{qbs-dual}, we have the following
equivalence of ball Campanato-type function spaces;
we omit the details.

\begin{corollary}\label{2c1}
Let $Y(\cx)$, $d$, and $p_0$ be as in Theorem \ref{qbs-dual}.
Suppose $q\in[1,\fz)$ when $p_0\in(0,1)$,
or $q\in[1,p_0')$ when $p_0\in[1,\fz)$.
Then $${\cl}_{Y,q,d}(\cx)={\cl}_{Y,1,d_0}(\cx)$$
with equivalent quasi-norms, where $d_0\in(0,\tz_0)$
with $\tz_0$ as in Theorem \ref{qbs-dual}.
\end{corollary}

\begin{remark}\label{6.6x}
\begin{enumerate}
\item[{\rm (i)}] We point out that the proof of Theorem \ref{qbs-dual}
strongly depends on the fact that $H_{\mathrm{fin}}^{Y,q,d}(\cx)$
is dense in $H_Y^*(\cx)$, which is guaranteed by $Y(\cx)$ having
absolutely continuous quasi-norm (see Lemma \ref{dense}). Therefore,
the assumption that $Y(\cx)$ has an absolutely continuous quasi-norm
in Theorem \ref{qbs-dual} is necessary.

\item[{\rm (ii)}] Observe that we establish a continuous
embedding of $Y(\cx)$ into the weighted Lebesgue space to
overcome the difficulty caused by the lack of the good dense
subset of $H_{Y}^*({\mathcal X})$ in Section \ref{s-atom}.
However, this method does not work anymore in the proof of
Theorem \ref{qbs-dual} because, in Section \ref{s-atom},
we only need the convergence in the distributional sense
of the atomic decomposition under consideration,
but now we need its convergence in the quasi-norm
$\|\cdot\|_{H_Y^*(\cx)}$.

\item[{\rm (iii)}]
Recall that, for any $\az\in(0,\fz)$, the \emph{Lipschitz space}
$\mathcal{L}_{\az}(\cx)$ is defined to be the set of
all measurable functions $f$ on $\cx$ such that
\begin{align*}
\|f\|_{\mathcal{L}_{\az}(\cx)}
:=\sup_{x\neq y}\frac{|f(x)-f(y)|}{[V(x,y)]^{\az}}<\fz.
\end{align*}
Let $p\in({\omega}/(\omega+\eta),1]$ with $\omega$ as
in \eqref{eq-doub} and $\eta$ as in Definition \ref{expati}.
On one hand, by \cite[Theorem B]{cw77} and
\cite[Theorems 4.2 and 4.16]{hlyy19}, we know that
the dual space of $H^{*,p}(\cx)$ is the space $\mathcal{L}_{1/p-1}(\cx)$
when $p<1$, or $\mathop\mathrm{BMO}\,(\cx)$ when $p=1$
(see \cite[p.\,593]{cw77} for the definition of $\mathop\mathrm{BMO}\,(\cx)$).
On the other hand, if $Y(\cx):=L^{p}(\cx)$, then,
by Remark \ref{atomre}(i) and Theorem \ref{qbs-dual},
we conclude that the dual space of $H^{*,p}(\cx)$ is
${\cl}_{Y,q',d}(\cx)$ with $Y(\cx)$ replaced by $L^p(\cx)$.
Thus, Theorem \ref{qbs-dual} give an equivalent characterization
of $\mathcal{L}_{1/p-1}(\cx)$ when $p<1$, or $\mathop\mathrm{BMO}\,(\cx)$ when $p=1$.

\item[{\rm (iv)}]
Let $p\in({\omega}/(\omega+\eta),1]$ with $\omega$ as in
\eqref{eq-doub} and $\eta$ as in Definition \ref{expati},
and $r\in(0,\fz)$. If $Y(\cx):=L^{p,r}(\cx)$,
then, by Remark \ref{atomre}(ii), we know that $L^{p,r}(\cx)$
satisfies all the assumptions of Theorem \ref{qbs-dual}.
In this case, Theorem \ref{qbs-dual} is new.

\item[{\rm (v)}]
Let $p\in({\omega}/(\omega+\eta),1]$, $r\in(1,\fz)$,
$w\in A_{r}(\cx)$ satisfy
$$r_w:=\inf\lf\{r\in[1,\fz):\
\vz\in A_r(\cx)\r\}\in(1,p(\omega+\eta)/{\omega})$$
with $\omega$ as in \eqref{eq-doub} and
$\eta$ as in Definition \ref{expati}.
If $Y(\cx):=L^p_w(\cx)$, then, by Remark \ref{atomre}(iii),
we know that $L^p_w(\cx)$ satisfies all the assumptions of
Theorem \ref{qbs-dual}. From this and a generalization of
\cite[Proposition 3.7]{zhyy21} on $\rn$ to $\cx$,
we deduce that \cite[Theorem 8.1]{fmy19}
is a special case of Theorem \ref{qbs-dual}.

\item[{\rm (vi)}]
Let $\Phi$ be an Orlicz function, $p_{\Phi}^+=1$,
and $p_{\Phi}^-\in({\omega}/(\omega+\eta),1]$
with $\omega$ as in \eqref{eq-doub} and
$\eta$ as in Definition \ref{expati}.
If $Y(\cx):=L^{\Phi}(\cx)$, then, by Remark \ref{atomre}(iv),
we know that $L^{\Phi}(\cx)$ satisfies all the assumptions of
Theorem \ref{qbs-dual}. From this and a generalization of
\cite[Proposition 3.7]{zhyy21} on $\rn$ to $\cx$,
we deduce that \cite[Theorem 8.1]{fmy19}
is a special case of Theorem \ref{qbs-dual}.

\item[{\rm (vii)}]
Let $(\cx,\rho,\mu)$ be an RD-space,
and $p(\cdot)\in C^{\log}(\cx)$ satisfy
$p_+\in(0,1]$ and $\widetilde{p_-}\in(\oz/(\oz+\eta),\fz)$,
where $\omega$ is as in \eqref{eq-doub} and
$\eta$ as in Definition \ref{expati}.
If $Y(\cx):=L^{p(\cdot)}(\cx)$,
then, by Remark \ref{atomre}(v), we know that
$L^{p(\cdot)}(\cx)$ satisfies all the assumptions of
Theorem \ref{qbs-dual}. In this case,
Theorem \ref{qbs-dual} improves the corresponding
results in \cite[Theorem 7.2]{zsy} by removing
the additional restriction $p_+\in(0,1]$ and
the reverse doubling assumption of $\mu$.

\item[{\rm (viii)}]
Let $\vz$ be a growth function as in Remark \ref{qbsdefrem}(iii).
As was mentioned in Remark \ref{atomre},
Theorem \ref{qbs-dual} is not applicable to $H^{*,\vz}(\cx)$.
Fu et al. \cite{fmy19} obtained a similar conclusion
of Theorem \ref{qbs-dual} in case that $Y(\cx):=L^{\vz}(\cx)$
independently.
\end{enumerate}
\end{remark}

\noindent\bf{Acknowledgements}\quad\rm
{The authors would like to thank Fan Wang for
some helpful discussions on the subject of this article.}

\bigskip

\noindent Xianjie Yan, Dachun Yang (Corresponding author) and Wen Yuan

\medskip

\noindent Laboratory of Mathematics and Complex Systems
(Ministry of Education of China),
School of Mathematical Sciences, Beijing Normal University,
Beijing 100875, People's Republic of China

\smallskip

\noindent {\it E-mails}: \texttt{xianjieyan@mail.bnu.edu.cn} (X. Yan)

\noindent\phantom{{\it E-mails:}} \texttt{dcyang@bnu.edu.cn} (D. Yang)

\noindent\phantom{{\it E-mails:}} \texttt{wenyuan@bnu.edu.cn} (W. Yuan)

\bigskip

\noindent Ziyi He

\medskip

\noindent School of Science, Beijing University of Posts
and Telecommunications, Beijing, 100876, People's Republic of China

\smallskip

\noindent {\it E-mail}: \texttt{ziyihe@bupt.edu.cn} (Z. He)

\end{document}